\documentclass[12pt, onecolumn]{IEEEtran}

\usepackage{float}

\usepackage{amsthm,amsmath,amsfonts,amssymb}

\usepackage{xspace,exscale,relsize}
\usepackage{fancybox,shadow}
\usepackage{graphicx}
\usepackage{caption}
\usepackage{subcaption} % conflict with [subfigure]
\usepackage{xcolor}
\usepackage{extarrows}
\usepackage{comment}
\usepackage{booktabs} % To thicken table lines
\usepackage{makecell, multirow}
\usepackage[noadjust]{cite}

\usepackage{mathrsfs}
\makeatletter
\let\over=\@@over \let\overwithdelims=\@@overwithdelims
\let\atop=\@@atop \let\atopwithdelims=\@@atopwithdelims
\let\above=\@@above \let\abovewithdelims=\@@abovewithdelims
\makeatother
\usepackage{rotating}
	\usepackage{ifpdf}
	
	\usepackage{psfrag}
	
	\usepackage{prettyref}

	\usepackage{tikz}
	\usetikzlibrary{arrows}
	\tikzstyle{int}=[draw, fill=blue!20, minimum size=2em]
	\tikzstyle{dot}=[circle, draw, fill=blue!20, minimum size=2em]
	\tikzstyle{init} = [pin edge={to-,thin,black}]

	\usepackage[%dvips,
	CJKbookmarks=true,
	bookmarksnumbered=true,
	bookmarksopen=true,
	colorlinks=true,
	citecolor=red,
	linkcolor=blue,
	anchorcolor=red,
	urlcolor=blue
	]{hyperref}
	
	\usepackage[all]{xy}
	
	\usepackage{paralist}
	\usepackage{enumitem}
	
	\usepackage{algorithm}
	\usepackage{algpseudocode}
	
	\usepackage{mathtools}
	\usepackage{ifthen}

	\ifx\eqref\undefined
	\newcommand{\eqref}[1]{~(\ref{#1})}
	\fi
	\ifx\mod\undefined
	\def\mod{\mathop{\rm mod}}
	\fi
	
	\usepackage{bm}

	\newcommand{\norm}[1]{{\left\Vert #1 \right\Vert}}

	\def\argmin{\mathop{\rm argmin}}

	\def\exp{\mathop{\rm exp}}

	\def\EE{\Expect}

	\def\PP{\mathbb{P}}

	\def\eqdef{\triangleq}

	\newcommand{\abs}[1]{\left| #1 \right|}

	\makeatletter
	\def\bbordermatrix#1{\begingroup \m@th
		\@tempdima 4.75\p@
		\setbox\z@\vbox{%
			\def\cr{\crcr\noalign{\kern2\p@\global\let\cr\endline}}%
			\ialign{$##$\hfil\kern2\p@\kern\@tempdima&\thinspace\hfil$##$\hfil
				&&\quad\hfil$##$\hfil\crcr
				\omit\strut\hfil\crcr\noalign{\kern-\baselineskip}%
				#1\crcr\omit\strut\cr}}%
		\setbox\tw@\vbox{\unvcopy\z@\global\setbox\@ne\lastbox}%
		\setbox\tw@\hbox{\unhbox\@ne\unskip\global\setbox\@ne\lastbox}%
		\setbox\tw@\hbox{$\kern\wd\@ne\kern-\@tempdima\left[\kern-\wd\@ne
			\global\setbox\@ne\vbox{\box\@ne\kern2\p@}%
			\vcenter{\kern-\ht\@ne\unvbox\z@\kern-\baselineskip}\,\right]$}%
		\null\;\vbox{\kern\ht\@ne\box\tw@}\endgroup}
	\makeatother
	
	\newcommand{\stepa}[1]{\overset{\rm (a)}{#1}}
	\newcommand{\stepb}[1]{\overset{\rm (b)}{#1}}
	\newcommand{\stepc}[1]{\overset{\rm (c)}{#1}}
	\newcommand{\stepd}[1]{\overset{\rm (d)}{#1}}

	\newcommand{\subG}{\mathsf{SubG}}
	\newcommand{\out}{\mathsf{out}}

	\newcommand{\Unif}{\mathrm{Uniform}}

	\newcommand{\floor}[1]{{\left\lfloor {#1} \right \rfloor}}
	\newcommand{\ceil}[1]{{\left\lceil {#1} \right \rceil}}

	\newcommand{\reals}{\mathbb{R}}

	\newcommand{\Expect}{\mathbb{E}}

	\newcommand{\TV}{{\rm TV}}

	\newcommand{\iid}{iid\xspace}

	\newcommand{\pth}[1]{\left( #1 \right)}
	\newcommand{\qth}[1]{\left[ #1 \right]}
	\newcommand{\sth}[1]{\left\{ #1 \right\}}

	\newcommand{\iiddistr}{{\stackrel{\text{\iid}}{\sim}}}

	\newcommand{\Binom}{\text{Binom}}

	\newcommand{\indc}[1]{{\mathbf{1}_{\left\{{#1}\right\}}}}

	\definecolor{myblue}{rgb}{.8, .8, 1}
	\definecolor{mathblue}{rgb}{0.2472, 0.24, 0.6} % mathematica's Color[1, 1--3]
	\definecolor{mathred}{rgb}{0.6, 0.24, 0.442893}
	\definecolor{mathyellow}{rgb}{0.6, 0.547014, 0.24}

	\newcommand{\blue}{\color{blue}}
	\newcommand{\nb}[1]{{\sf\blue[#1]}}

	\newcommand{\sfM}{{\mathsf{M}}}

	\newcommand{\calB}{{\mathcal{B}}}
	
	\newcommand{\calD}{{\mathcal{D}}}
	\newcommand{\calE}{{\mathcal{E}}}
	\newcommand{\calF}{{\mathcal{F}}}

	\newcommand{\calP}{{\mathcal{P}}}
	\newcommand{\calQ}{{\mathcal{Q}}}
	\newcommand{\calR}{{\mathcal{R}}}
	\newcommand{\calS}{{\mathcal{S}}}

	\newcommand{\calZ}{{\mathcal{Z}}}

	\newrefformat{eq}{(\ref{#1})}
	\newrefformat{thm}{Theorem~\ref{#1}}
	\newrefformat{th}{Theorem~\ref{#1}}
	\newrefformat{chap}{Chapter~\ref{#1}}
	\newrefformat{sec}{Section~\ref{#1}}
	\newrefformat{seca}{Section~\ref{#1}}
	\newrefformat{algo}{Algorithm~\ref{#1}}
	\newrefformat{fig}{Fig.~\ref{#1}}
	\newrefformat{tab}{Table~\ref{#1}}
	\newrefformat{rmk}{Remark~\ref{#1}}
	\newrefformat{clm}{Claim~\ref{#1}}
	\newrefformat{def}{Definition~\ref{#1}}
	\newrefformat{cor}{Corollary~\ref{#1}}
	\newrefformat{lmm}{Lemma~\ref{#1}}
	\newrefformat{prop}{Proposition~\ref{#1}}
	\newrefformat{pr}{Proposition~\ref{#1}}
	\newrefformat{property}{Property~\ref{#1}}
	\newrefformat{app}{Appendix~\ref{#1}}
	\newrefformat{apx}{Appendix~\ref{#1}}
	\newrefformat{ex}{Example~\ref{#1}}
	\newrefformat{exer}{Exercise~\ref{#1}}
	\newrefformat{soln}{Solution~\ref{#1}}
	\def\unifto{\mathop{{\mskip 3mu plus 2mu minus 1mu%
				\setbox0=\hbox{$\mathchar"3221$}%
				\raise.6ex\copy0\kern-\wd0%
				\lower0.5ex\hbox{$\mathchar"3221$}}\mskip 3mu plus 2mu minus 1mu}}

	\ifx\lesssim\undefined
	\def\simleq{{{\mskip 3mu plus 2mu minus 1mu%
				\setbox0=\hbox{$\mathchar"013C$}%
				\raise.2ex\copy0\kern-\wd0%
				\lower0.9ex\hbox{$\mathchar"0218$}}\mskip 3mu plus 2mu minus 1mu}}
	\else
	\def\simleq{\lesssim}
	\fi
	
	\ifx\gtrsim\undefined
	\def\simgeq{{{\mskip 3mu plus 2mu minus 1mu%
				\setbox0=\hbox{$\mathchar"013E$}%
				\raise.2ex\copy0\kern-\wd0%
				\lower0.9ex\hbox{$\mathchar"0218$}}\mskip 3mu plus 2mu minus 1mu}}
	\else
	\def\simgeq{\gtrsim}
	\fi

		\newtheorem{theorem}{Theorem}
		\newtheorem{lemma}[theorem]{Lemma}
		\newtheorem{corollary}[theorem]{Corollary}
		
		\newtheorem{proposition}[theorem]{Proposition}

		\theoremstyle{definition}

		\newtheorem{remark}{Remark}
		
		\newif\ifmapx
		{\catcode`/=0 \catcode`\\=12/gdef/mkillslash\#1{#1}}
		\edef\jobnametmp{\expandafter\string\csname ic_apx\endcsname}
		\edef\jobnameapx{\expandafter\mkillslash\jobnametmp}
		\edef\jobnameexpand{\jobname}
		\ifx\jobnameexpand\jobnameapx
		\mapxtrue
		\else
		\mapxfalse
		\fi

		\newcommand{\decay}{G}

		\renewcommand{\hat}{\widehat}
		\renewcommand{\tilde}{\widetilde}

		\newcommand{\cod}{{\sf COD}~}
		\newcommand{\codelta}{{$\sf COD_\delta$}~}

		\newcommand{\snr}{\mathsf{SNR}}
		
		\newcommand{\enorm}{\calE^{\sf{norm}}}
		\newcommand{\econ}{\calE^{\sf{con}}}

		\newcommand{\hdp}{${\sf HDP}$}
		\newcommand{\dist}{{\sf dist}}
		\newcommand{\totdist}{{\sf totdist}}
		\newcommand{\iod}{{\sf IOD}~}

\begin{document}

\ifpdf
\DeclareGraphicsExtensions{.pgf,.jpg,.pdf}
\graphicspath{{figures/}{plots/}}
\fi

\title{A provable initialization and robust clustering method for general mixture models}

\author{Soham Jana, Jianqing Fan, and Sanjeev Kulkarni\thanks{S.~J. is with the Department of Applied and Computational Mathematics and Statistics, University of Notre Dame, Notre Dame, IN, USA, email:\url{soham.jana@nd.edu}.  J.~F. and S.~K. are with the Department of Operations Research and Financial Engineering and Department of Electric and Computer Engineering, Princeton University, Princeton, NJ, USA email:\url{jqfan@princeton.edu}, \ \url{kulkarni@princeton.edu}. J.F. is partially supported by NSF grants DMS-2210833, DMS-2053832, DMS-2052926 and ONR grant N00014-22-1-2340.}
			}
			
\date{\today}
	
%\author{}

\maketitle

	\begin{abstract}
			Clustering is a fundamental tool in statistical machine learning in the presence of heterogeneous data. Most recent results focus primarily on optimal mislabeling guarantees when data are distributed around centroids with sub-Gaussian errors. Yet, the restrictive sub-Gaussian model is often invalid in practice since various real-world applications exhibit heavy tail distributions around the centroids or suffer from possible adversarial attacks that call for robust clustering with a robust data-driven initialization. In this paper, we present initialization and subsequent clustering methods that provably guarantee near-optimal mislabeling for general mixture models when the number of clusters and data dimensions are finite. We first introduce a hybrid clustering technique with a novel multivariate trimmed mean type centroid estimate to produce mislabeling guarantees under a weak initialization condition for general error distributions around the centroids. A matching lower bound is derived, up to factors depending on the number of clusters. In addition, our approach also produces similar mislabeling guarantees even in the presence of adversarial outliers. Our results reduce to the sub-Gaussian case in finite dimensions when errors follow sub-Gaussian distributions. To solve the problem thoroughly, we also present novel data-driven robust initialization techniques and show that, with probabilities approaching one, these initial centroid estimates are sufficiently good for the subsequent clustering algorithm to achieve the optimal mislabeling rates. Furthermore, we demonstrate that the Lloyd algorithm is suboptimal for more than two clusters even when errors are Gaussian and for two clusters when error distributions have heavy tails. Both simulated data and real data examples further support our robust initialization procedure and clustering algorithm.
	\end{abstract}

	\noindent%

	\section{Introduction}

\subsection{Problem}

Clustering is an essential task in statistics and machine learning \cite{hastie2009elements,xu2005survey} that has diverse practical applications (e.g., wireless networks \cite{ABBASI20072826,sasikumar2012k}, grouping biological species \cite{maravelias1999habitat,pigolotti2007species}, medical imaging \cite{ng2006medical,ajala2012fuzzy} and defining peer firms in finance \cite{beatty2013spillover,fan2023unearthing}). One of the simplest and most studied clustering models is the additive $k$-centroid setup, where data points are distributed around one of the centroids according to some unknown additive error distribution. Mathematically, this model can be described as
\begin{align}\label{eq:model}
	Y_i = \theta_{z_i} + w_i, \quad z_1,\dots,z_n\in \{1,\dots,k\}, \quad \theta_1,\dots,\theta_k\in \reals^d,
\end{align}
for a given $k$.
Here $Y_1,\dots, Y_n$ are the data, $\theta_g$ is the centroid corresponding to the $g$-th cluster, $z=(z_1,\dots,z_n)$ is the unknown label vector of the data describing which clusters they belong to, and $w_1,\dots,w_n$ are unknown independent errors. Recent advances in the literature have focused on recovering the labels $z$. Given any estimate $\hat z = \pth{\hat z_1,\dots,\hat z_n}$ of $z$, define the mislabeling error as the proportion of label estimates that do not match the correct ones (up to a permutation of labels): 
\begin{align}\label{eq:mislabeling}
	\ell(\hat z,z)=\inf_{\pi\in \calS_k}\qth{\frac 1n\sum_{i=1}^n\indc{\pi(z_i)\neq \hat z_i}},
\end{align}
where $\calS_k$ is the collection of all mappings from $[k]$ to $[k]$. Then, clustering algorithms are constructed to partition the data into $k$ groups such that the data labels are recovered with a small mislabeling error.

%The clustering problem is distinct from the mixture density estimation problems \cite{lindsay1993multivariate,moitra2010settling,doss2020optimal,jana2024optimalempiricalbayesestimation}. In the above literature, the focus is often on recovering the data generating mixture density without prioritizing on recovering the exact components. The mixture components that are difficult to distinguish can be estimated consistently using one single distribution. However, the mislabeling error in such cases be bounded away from zero with a high probability, making the clustering task more challenging.

A surge of recent work focuses on constructing provable algorithms to guarantee smaller mislabeling errors as the minimum separation of centroids increases. Let $\|\cdot \|$ denote the Euclidean norm for the rest of the paper, unless mentioned otherwise. The majority of this work studies specific light-tailed models. For example, in the Gaussian mixture model, \cite{lu2016statistical} established that the Lloyd algorithm with spectral initialization can achieve the mislabeling rate $\exp\pth{-(1+o(1)){\Delta^2\over 8\sigma^2}}$, where  $\Delta = \min_{g\neq h\in [k]}\|\theta_g-\theta_h\|$ and $\sigma$ is the common standard deviation of the error coordinates. They also show that the above rate is minimax optimal in finite dimensions. Extending on the above work, \cite{loffler2021optimality,abbe2022} show that spectral clustering algorithms can achieve the above rate of error in high dimensions. In the very specific case of $k=2$, when $d$ is much larger than the sample size $n$, \cite{ndaoud2022sharp} used a variant of Lloyd's algorithm with spectral initialization to achieve a more precise error rate $\exp\pth{-\Theta\pth{\Delta^4/\sigma^4\over {\Delta^2/\sigma^2 + d/n}}}$. Recently \cite{dreveton2024universal} extended these results to Laplace distributions for the $w_i$. They showed that a variant of the spectral initialization combined with a maximum likelihood-based estimation strategy can achieve the minimax mislabeling error $\exp\pth{-(1+o(1))\pth{\min_{g\neq h\in [k]}\|\theta_g-\theta_h\|_1\over \sigma}}$, where $\|\cdot\|_1$ denotes the $L_1$ norm of vectors and $\sigma$ is the common standard deviation in each coordinate of the error vector. Their work also addresses the clustering problem for specific types of sub-exponential mixture distributions, for which they used variants of Lloyd's algorithms based on the Bregman divergence. Notably, the mislabeling in the sub-exponential case is significantly different than in the Gaussian case.

The above exposition illustrates that given any fixed constellation of centroids, the statistical guarantees for mislabeling depend significantly on the properties of the tail of the distributions of the $w_i$. In particular, given two data-generating setups with the same centroids, the heavy-tailed setup is expected to have a higher mislabeling error. The clustering and initialization methods in the existing literature, such as spectral techniques, are mostly geared to tackle specific light-tailed distributional setups, and it is unknown whether their guarantees extend to general heavy-tailed errors where outliers are prevalent. We address the above issues in the current work. Our proposed algorithm is a significant departure from the existing literature as it is designed to adapt to heavy-tailed error distributions, a scenario that has yet to be extensively explored. To summarize, our paper deviates from most of the existing literature in trying to answer the following:
\begin{center}
	\begin{minipage}{0.8\linewidth}
		\bigskip	
		{\it Can we construct an initialization method and subsequent clustering algorithm that provably adapt to the decay conditions for general heavy-tailed error distributions and produce (nearly) optimal mislabeling errors under generic assumptions?}
		\bigskip
	\end{minipage}
\end{center}
To address the above question,
%as a first paper in this direction,
we primarily focus on generalizing the existing results concerning the tails of the error distributions. As a first work in this direction, we restrict ourselves to the scenario where the actual clusters are of similar sizes (see, e.g., \cite{abbe2022} for similar assumptions), and $d,k$ are finite. In particular, we assume the following:

\begin{enumerate}[label=(C\arabic*)]
	\item There is a constant $C$ such that $k,d\leq C$.
	\item \label{cond:alpha} $\alpha=\min_{g\in [k]}{\sum_{i=1}^n \indc{z_i=g}/ n}\geq c/k$ for some $c>0$
	
	\item \label{cond:G-decay} ($G$-decay condition) The error distributions in \eqref{eq:model} satisfy
	\begin{align}
		\PP\qth{\|w_i\|>x}\leq G(x/\sigma),\quad \sigma>0,x\geq 0, i\in \{1,\dots,n\}, \quad \text{$G(\cdot)$ is decreasing on $\reals_+\cup \sth{0}$}.
	\end{align}
	\item (Parameter space) The parameters $(z,\{\theta_g\}_{g\in [k]})$ is generated from the parameter space $\calP_\Delta$
	$$
	\calP_{\Delta} = \sth{(z,\{\theta_g\}_{g\in [k]}): z\in [k]^n, \min_{g\neq h\in [k]}\|\theta_g-\theta_h\|\geq \Delta}.
	$$
\end{enumerate}
Note that we do not assume much knowledge about the decay function $G$, making our result general in the true sense. We propose a novel initialization algorithm \iod (Initialization via Ordered Distances, see \prettyref{algo:init} for $k=2$ and \prettyref{algo:init-k} for $k\geq 3$) and a novel subsequent iterative clustering technique \cod (Clustering via Ordered Distances, provided in \prettyref{algo:cod}) that are also oblivious to the decay condition $G$. In addition to the observations $\{Y_1,\dots,Y_n\}$, the only extra information our algorithms use is knowledge about a lower bound on $\alpha$. This is often assumed to be known in practice, e.g., the smallest cluster contains at least 5-10\% of the data points. Our main result is the following.
\begin{theorem}\label{thm:main}
	Suppose that we know $\alpha>c/k$ for some $c>0$ and let $\hat z$ be the output label vector when we run the \cod clustering method with \iod initialization scheme that knows this $c$. There exist constants $\{c_{k,i}\}_{i=1}^4, c_{G,\alpha} C_{G,\alpha}$ such that the following are satisfied. Whenever $\Delta\geq \sigma C_{G,\alpha} $, we have
	$$
	\sup_{\calP_\Delta}\EE\qth{\ell(\hat z,z)}
	\leq c_{k,1} G\pth{{\Delta\over 2\sigma}-c_{G,\alpha}}
	+c_{k,2}e^{-n\alpha/4}.
	$$
	The entire algorithm runs in $c_{k,3}n^2(d+\log n)$ time. In addition, under minor smoothness condition on $G$, whenever $\alpha\in \pth{{c\over k},{\bar c\over k}}$ for some constants $0<c<\bar c<1$, there exists $c_G>0$ such that
	$$\inf_{\hat z}\sup_{\calP_\Delta}\EE\qth{\ell(\hat z,z)}\geq c_{k,4}G\pth{{\Delta\over 2\sigma}+c_{G}}.$$
\end{theorem}
\begin{remark}
	The closest work to ours that we could find is \cite{dreveton2024universal}, which attempts to provide algorithms to achieve similar generalized guarantees. However, the construction of their algorithm hinges on knowing the data-generating model (Algorithm 1 in their paper), and even then, their method works only in the presence of a good initialization. Their work can achieve such initialization via spectral methods for mixture models based on specific parametric families of sub-exponential distributions. Generalizing such results to heavy-tailed models is unknown in the literature. 
\end{remark}
\begin{remark}
	Our result aims to provide a provable adaptive clustering technique that guarantees consistent clustering (i.e., with vanishing mislabeling) under general decay conditions. The vanishing mislabeling is only achieved in the regime where the signal-to-noise ratio $\snr = \Delta/2\sigma$ is significant, and we do not explore the territory of small $\snr$. In particular, our result indicates that when $k\leq C$ for some constant $C>0$, our algorithm achieves the minimax mislabeling error rate (the additive term $c_{k,2}e^{-n\alpha/4}$ can be ignored as, for large $n$ such that $c_{k,2}ne^{-n\alpha/4}$ is much smaller than 1, this corresponds to no extra mislabeled points with a high probability). The decay condition \ref{cond:G-decay} on the errors $w_i$ is presented based on the Euclidean norm to simplify the theoretical results. The above decay condition translates to decay conditions on the distribution of each coordinate of the $w_i$ when the data dimension $d$ is at most a constant. Hence, our results easily extend to the sub-Gaussian and sub-exponential mixture models in finite dimensions. The generalization of such results to high-dimensional setups is left to future works. Our results also translate to obtaining the regime of exact recovery, i.e. ${\Delta\over 2\sigma} > c_{G,\alpha}+{G^{-1}\pth{1\over c_{k,1}n}}$ which corresponds to expected mislabeling dropping below $\frac 1n$. However, this phenomenon is traditionally studied in high dimensional regimes \cite{chen2021cutoff,ndaoud2022sharp}, which is beyond the scope of the current work. 
\end{remark}

\begin{remark}[Dependence on $k$]
	Our aim in this work is to generalize existing clustering results in terms of the decay condition of the error distributions and provide provable initialization and clustering algorithms that run in polynomial time for any finite number of clusters. We leave it for later work to optimize the dependence of our results in terms of the number of clusters $k$. To this end, we provide explicit values of the $k$-dependent terms in \prettyref{thm:main}. In particular we have $c_{k,1}=k^2,c_{k,2}=8k,c_{k,4}={1-k\alpha-k/n\over 12}$. This implies that as long as $k=O(1)$ and $n$ are large, our lower and upper bound differs by a multiple of $k^2=O(1)$. In terms of the runtime of the algorithm we have $c_{k,3}=\pth{O(1)k^2\over \alpha^2}^{k-1}$ which is at most a constant if $k=O(1)$. Notably, this dependency stems from the use of a recursive framework in our initialization algorithms (\prettyref{algo:init} and \prettyref{algo:init-k}), as explained in the proof of \prettyref{thm:runtime-init-k}, and is unavoidable for recursive techniques. See, e.g., the classical $(1+\epsilon)$-approximated $k$-means method of \cite{kumar2004simple} which has a runtime of $2^{k^{O(1)}}dn$ for a general $k$.
\end{remark}

\begin{remark}[Comparison with existing initialization algorithms]
	We point out that robust initialization techniques are lacking in the literature that can balance both statistical guarantees and fast runtime, and our paper prioritizes the statistical guarantee part. In particular, initialization schemes that can provably guarantee good outputs for data generated by a general mixture model are almost nonexistent, even for finite $k,d$. For comparing with existing initialization schemes in the literature, consider the classical work of \cite{kumar2004simple}, one of the top choices for initialization algorithms for clustering sub-Gaussian mixtures \cite{loffler2021optimality,abbe2022,chen2021optimal}. Their algorithm uses a similar recursive scheme to ours, although it runs in linear time in $n$. A significant improvement in our results compared to the above work is that by optimizing our iterative schemes, we can guarantee a good initialization with a probability tending to one as $n$ increases, as mentioned in \prettyref{thm:proof-algo-2} and \prettyref{thm:proof-algo-k}.
	In contrast, the final theoretical guarantees of the work mentioned above hold with a probability $\gamma^k$ for a constant $\gamma$ much smaller than one \cite[Theorem 4.1]{kumar2004simple}. This is because, to guarantee a linear runtime, the author first sampled a constant number of data points in their algorithm and then used this sample to find the initialization. Another relevant fast initialization method is the classical $k$-means++ \cite{vassilvitskii2006k,patel2023consistency}. The runtime of the above method has a weaker dependence on $k$. However, $k$-means++ is known to be highly sensitive to outliers and heavy-tailed errors \cite{deshpande2020robust}. The last paper aims to provide a robust modification of $k$-means++. However, their algorithm's runtime has a similar dependence as ours also when they try to output exactly $k$ clusters. Their paper also mentions that recently \cite{bhaskara2019greedy} aimed to produce a robust initialization method with much faster improved runtime dependence on $k$. However, their method requires an initial knowledge of the cost of clustering. It is also popular to use spectral methods for centroid initialization; however, theoretical guarantees for mislabeling minimization tend to exist only in the sub-Gaussian setup. For example, \cite{lu2016statistical} uses the spectral method presented in \cite[Claim 3.4]{kannan2009spectral}, which bounds the centroid estimation error using the Frobenius norm of the error matrix, and then uses a concentration of the sub-Gaussian errors to bound it. Unfortunately, such concentration results fail to work for heavier tails, such as with a polynomial tail. We also provide a general solution to this initialization problem in our work. On a side note, the vanilla version of spectral methods is also vulnerable to noisy setups \cite{bojchevski2017robust,zhang2018understanding}. This creates issues in the modern setting of adversarial data contamination, which has become prevalent recently.
\end{remark}

In addition to the above, we also touch upon the robustness properties of our algorithms to adversarial data contamination. Adversarial data contamination is one of the important challenges of modern machine learning literature, and it often arises due to security concerns. Clustering techniques that are robust to such noises are advantageous in practice and have garnered significant importance in recent trends (see, e.g., \cite{srivastava2023robust,JKY2023} for related works). As a counterpart to the above result in \prettyref{thm:main}, we show that when $k$ is bounded from above with a constant, our initialization algorithm \iod can tolerate a constant fraction of adversarially contaminated data and still produce the desired initialization guarantees as in \prettyref{thm:proof-algo-outliers}. In addition, the subsequent iterative algorithm \cod can also be improvised so that the mislabeling error for the uncorrupted data points can be retained at the minimax optimal level. The result is presented in \prettyref{thm:out-genr}.

\subsection{Our contributions}
\label{sec:contributions}

Our paper serves as a valuable proof of concept for the universality of nearest-neighbor-based approaches for robust initialization and clustering. \begin{enumerate}[label=(\roman*)]
	\item \textbf{Initialization:} In the core of our proposed initialization method is a robust improvement on the classical concept of within-cluster sum of squares (WCSS) used from clustering. Given the data set $\sth{Y_i}_{i=1}^n$ and centroid estimates $\{\hat \theta_j\}_{j=1^k}$, the WCSS is defined as 
	\begin{align}
		\label{eq:wcss}
	\text{WCSS}(\{\hat \theta_j\}_{j=1}^k;\sth{Y_i}_{i=1}^n) 
	= \sum_{i=1}^n \min_{j=1}^k \|Y_i-\hat\theta_j\|^2.
	\end{align}
	
	The primary purpose of the WCSS metric is to assess the quality of the centroid estimates with respect to the data. The classical use of WCSS can be traced to the $k$-means problem \cite{lloyd1982least}, and its other extensions \cite{bradley1996clustering} with different metrics. Most recently, the linear time approximate $k$-means clustering technique of \cite{kumar2004simple} for initialization has come to prominence, as the typical runtime of the exact $k$-means minimization schedule is $n^{k^2}+1$ \cite{loffler2021optimality}. This algorithm is also based on the WCSS metric. 
	
	Although promising in the sub-Gaussian noise model, the WCSS is rendered useless in the heavy-tailed setup. This is essentially due to the fact that the WCSS considers the sum based on all the points, and with a high probability, a fraction of points in heavy-tailed setups behave as outliers and significantly inflates the total sum. To deal with the issue, we propose an adaptive quantile-based version of WCSS (see the quantity $\totdist$ in \prettyref{algo:init} and \prettyref{algo:init-k}), which computes the sum of distances after removing certain outliers. The strategy ensures that the relevant guarantees adhere to different decay conditions to help produce the aforementioned minimax guarantees. As a consequence of this quantile-based trimming strategy, our method is able to tolerate adversarial points to some extent. As a result of this adaptive strategy, our analysis is significantly involved.
	
	\item \textbf{Iterative clustering:} Iterative clustering methods broadly follow two major steps:
	\begin{itemize}
		\item {\it Labeling step:} Given an estimate of the centroids $\hat \theta^{(s)}_h$, construct cluster estimates using the Euclidean distance
		
		\item {\it Estimation step:} For each of the estimated clusters, compute the new centroid estimates $\hat \theta^{(s+1)}_h$ using a suitable estimator.
	\end{itemize}
	The classical Lloyd's algorithm \cite{lu2016statistical} uses the sample means of the clusters to update the centroids. However, the sample mean lacks robustness property. To induce robustness properties in the iterative setup provided above, \cite{JKY2023} uses the coordinatewise median to update centroids. However, the coordinatewise median is often too conservative for outlier removal purposes. 
	
	The novel centroid updating method we propose in \prettyref{algo:cod} centers around the idea of distance-based trimming. To update the centroid of any particular cluster, we first compute the mutual distances of the points and figure out which data point (say $P$) in the cluster is the most central, relative to a quantile value of the distances. We then compute a trimmed mean of the points by removing a fraction of the cluster's points farthest from the point $P$. Our analysis points out that even without knowledge of the decay condition \ref{cond:G-decay}, this strategy of trimmed mean adaptively achieves the required mislabeling guarantees in \prettyref{thm:main}. The algorithm also runs in quadratic time in the sample size $n$, with most time spent computing the mutual distance.
	
	We present a comparison with relevant robust centroid updating strategies. In the presence of outliers, the centroid estimation guarantees for the coordinatewise median have a dependency on the dimension, which prevents us from attaining the dimension-free guarantees in \prettyref{thm:main}. Another useful contender might be the geometric median \cite{minsker2024geometric}, which is also relevant for dealing with data models based on the Euclidean distance. However, the concentration properties of this estimator are lacking in the literature. This makes the analyses challenging, and we leave it for future work. In contrast, the trimmed mean is known to produce sub-Gaussian concentration guarantees, which is the central part of the analysis in this paper. Another well-known robust estimator is Tukey's median, which is established to produce exponential concentration around the true centroid as well \cite{chen2018robust}. Unfortunately, the estimator requires exponential computation time for multiple dimensions. This defeats our purpose of obtaining a polynomial time algorithm.

\end{enumerate}

\subsection{Related work}

There are limited results in the literature that consider the heavy-tail regime, which is the main focus of our work. Notably, \cite{diakonikolas2022clustering} studied the mislabeling minimization problem with different moment constraints. However, their work only aims to produce a mislabeling rate that is a vanishing proportion of the minimum cluster size and does not guarantee optimality or quantify the mislabeling. Robust centroid estimation techniques are also central to our approach. It would be interesting if relevant robust methods in the modern literature, such as mean estimation \cite{depersin2022robust,devroye2016sub,oliveira2023trimmed}, vector mean estimation \cite{lee2022optimal,lugosi2021robust,lugosi2019near}, regression \cite{lugosi2019mean,audibert2011robust} can contribute to improving our results.

A long list of work utilizes a robust centroid estimation technique in clustering. The classical partitioning around the medoid (PAM) algorithm \cite{kaufman2009finding,rousseeuw1987clustering} updates the centroid estimates using a point from the data set (these centroid estimates are referred to as the medoids of the clusters) based on some dissimilarity metric. For example, \cite{rousseeuw1987clustering} used the $\ell_1$ distance and argued the robustness of the corresponding $\ell_1$ based PAM algorithm.

In our paper, we also use the adversarial contamination model. In this model, upon observing the actual data points, powerful adversaries can add new points of their choosing, and our theoretical results depend on the number of outliers added. This contamination model is arguably stronger than the traditional Huber contamination model \cite{huber1965robust,huber1992robust}, which assumes that the outliers originate from a fixed distribution via an \iid mechanism. Our model is similar to the adversarial contamination model studied in \cite{lugosi2021robust,diakonikolas2019robust} for robust mean estimation. For robust clustering of Gaussian mixtures, \cite{liu2023robustly} examines a similar contamination model. However, these works do not study adversarial outliers in the presence of general heavy-tail error distributions, as is being done in this paper.

Another critical related direction is clustering anisotropic mixture models, where the error probabilities decay in particular directions more than others. This differs from our setup, as our decay condition is independent of any direction. Clustering anisotropic mixtures has been studied previously in the sub-Gaussian setup, e.g., in \cite{chen2021optimal} using a variant of the Lloyd algorithm, in \cite{minsker2021minimax} for high-dimensional clustering in the specific case of $k=2$, and \cite{diakonikolas2020robustly,bakshi2022robustly} with the target of approximating the mixture distribution in Total Variation distance. Specific heavy-tail regimes with non-spherical mixtures are also discussed in \cite{bakshi2020outlier2}; however, they do not characterize the mislabeling in terms of the minimum separation distance. It would be interesting to study whether modifications of our clustering methodology can also work in such asymmetric clustering paradigms.

\subsection{Organization}

The rest of the paper is organized as follows. In \prettyref{sec:main-results}, we re-introduce our mathematical model and present the clustering algorithm we use, given a good initialization. The theoretical results, i.e., mislabeling rate upper bound under good initialization conditions and the worst case mislabeling lower bound are presented in \prettyref{sec:upper_bound} and \prettyref{sec:lower_bound} respectively. These two results jointly characterize the expected mislabeling as a function of the minimum centroid separation. As an application of our results, we study the mislabeling errors for the sub-Gaussian distributions and distributions with moment constraints in \prettyref{sec:specific_dist}. We present our initialization algorithm and their theoretical guarantees in \prettyref{sec:initialization}. The results involving the robustness of our algorithms to adversarial outliers are presented in \prettyref{sec:adv_outliers}. In \prettyref{sec:subopt_lloyd}, we show that the Lloyd algorithm might produce non-converging mislabeling errors even with good initialization. We demonstrate the effectiveness of our algorithms with application on actual and simulated data sets in \prettyref{sec:expt}. All the proofs and technical details have been provided in the appendix.

\section{Robust clustering under mislabeling guarantees under good initialization}
\label{sec:main-results}

\subsection{Algorithm}
In this section, we present our iterative clustering algorithm with the assumption that an initial estimate of either the centroids or the label vector is available. As mentioned before in \prettyref{sec:contributions}, our iterative clustering technique involves a centroid estimation step. For the above purpose, we use a novel multivariate trimmed mean algorithm (${\sf TM}_\delta$) based on the ordered distances between all points in the estimated clusters. Algorithm~\ref{algo:cent} presents the ${\sf TM}_\delta$ estimator for finding the trimmed mean of a dataset $S=\{X_1,\dots,X_n\}$.
\begin{algorithm}[H]
	\caption{The Trimmed Mean (${\sf TM}_\delta$) estimator}\label{algo:cent}
	\begin{flushleft}
		{\bf Input}:  Set of points $S=\sth{X_1,\dots,X_m}$, truncation parameter $\delta$
	\end{flushleft}
	\begin{algorithmic}[1]
		\State Create distance matrix $D=\sth{D_{ij}: i,j\in [m],D_{i,j}=\|X_i-X_j\|}$
		\For{Each $i\in [m]$}
		\State Compute $R_i$ as $\ceil{(1-\delta) m}$-th smallest number in $\sth{D_{ij},j\in [m]}$
		\EndFor
		\State Find $i^* = \argmin_{i\in [m]} R_i$. 
		\State Compute an $\ceil{(1-\delta)m}$-sized set $V\subseteq [m]$, with a priority to the points closer to $X_{i^*}$, ties broken arbitrarily
		$$V=\sth{j\in [m]: \|X_j-X_{i^*}\|\leq R_{i^*}}$$ 
	\end{algorithmic}
	\begin{flushleft}
		{\bf Output}: ${\sf TM}_{\delta}(\{X_1,\dots,X_m\})=\frac {\sum_{j\in V} X_j}{\ceil{(1-\delta)m}}$
	\end{flushleft}
\end{algorithm}

Here is an intuitive explanation of the ${\sf TM}_\delta$ estimator. We first aim to find out a point $X_{i^*}$ from the data set $X=\{X_1,\dots, X_m\}$ such that the radius of the ball around $X_{i^*}$ that contains $\ceil{(1-\delta)m}$ many points in $X$ is the smallest. In other words, $X_{i^*}$ is the point among the data that has the tightest neighborhood of size $(1-\delta)m$ points within the set $X$. Then, the algorithm computes an average of data points from the tightest neighborhood. Notably, if all the points in $X$ were independent and were generated via a distribution from the class $\decay_\sigma$ around a centroid $\theta$, then our analysis based on an argument about the quantiles of $G$, shows that with a high probability the tightest neighborhood will be contained in a ball around $\theta$, where, the radius of the ball depends on $G,\sigma,\delta$. Hence, the estimator ${\sf TM}_{\delta}(\{X_1,\dots,X_m\})$ will be close to $\theta$. When $\delta$ is very small, the estimator is approximately the sample mean, which is unbiased for $\theta$.

In the main clustering algorithm, we apply the above estimator on each of the approximated clusters. Consider one of the approximated clusters. In the approximated cluster the data points are not necessarily independent, and there are misclustered points. In such scenarios, we will require the approximated cluster to contain at least half of the points from the corresponding true cluster to make the estimation meaningful. Let us assume that the true cluster contains $m$ points and $\ceil{m(\frac 12+c)}$ many points from it belong to the approximated cluster. Then our analysis, using a union bound to deal with the possible dependency issue, shows that the ${\sf TM}_{\delta}$ algorithm for any $\frac 12 -c <\delta<\frac 12$ can estimate the centroid of the true cluster, although with a bias that depends on $G,\sigma,\delta$. Notably, with a large enough $\Delta$, the mislabeling error will be asymptotically unaffected by this bias.

In view of the above centroid estimation algorithm, we present our primary clustering technique, the {\sf Clustering via Ordered Distances} ({$\sf COD_\delta$}), below in \prettyref{algo:cod}. Our algorithm requires that an initial estimate of the centroids or label vector to kick off the clustering process. With a lousy initialization, the method can converge to local optima. This is similar to most off-the-shelf methods like the $k$-means algorithm, Fuzzy $C$-means algorithm, EM algorithms, etc. \cite[Section 3]{omran2007overview}. We will provide a novel robust centroid initialization technique later in \prettyref{sec:initialization} that will guarantee a global convergence for our {$\sf COD_\delta$} algorithm.

\begin{algorithm}[H]
	\caption{The {\sf Clustering via Ordered Distances} ({$\sf COD_\delta$}) - algorithm}\label{algo:cod}
	\begin{flushleft}
		{\bf Input}:  Data $\{Y_1,\dots,Y_n\}$. Initial centroid estimates $(\hat \theta_1^{(0)},\dots,\hat \theta_k^{(0)})$ (or initial label estimates $\sth{\hat z_i^{(0)}}_{i=1}^n$). Maximum number of iterations $M$. Error threshold $\epsilon$. Truncation level $\delta\in (0,\frac 12)$.
	\end{flushleft}
	\begin{algorithmic}[1]
		\State Set s = 1
		\For{$h\in \sth{1,2,\dots,k}$}
		\State \textbf{Labeling step}: 
		\If{$s=1$ and initial estimate of the label vector is available}
		\State Compute clusters ${T}_h^{(s)}=\{i\in \{1,\dots,n\}:\hat z_i^{(0)}=h\}$
		\Else 
		\State Compute the clusters, with ties broken arbitrarily, 
		$${T}_h^{(s)}=\sth{i\in \sth{1,\dots,n}: \|Y_i-\hat\theta_h^{(s-1)}\|\leq \|Y_i-\hat\theta_a^{(s-1)}\|,\ a\in\{1,\dots,k\}, a\neq h},$$
		
		\EndIf
		\State \textbf{Estimation step:} 
		Update the new estimate of $\theta_h$ as $\hat\theta_h^{(s)}= {\sf TM}_{\delta}(\{Y_j:j\in T_h^{(s)}\})$.
		\EndFor
		\If{$s=1$ or \{$2\leq s<\sfM$ and $\frac 1k\sum_{h=1}^k \|\hat\theta_h^{(s)}-\hat\theta_h^{(s-1)}\|^2 > \epsilon$\}}
		\State Update $s\leftarrow s+1$ and go back to the \textbf{Labeling step} and repeat
		\EndIf
	\end{algorithmic}
	\begin{flushleft}
		\noindent {\bf Output}: $(\hat \theta_1^{(s)},\dots,\hat \theta_k^{(s)})$ and $\hat z_i^{(s)} = \argmin_{h\in\{1,\dots,k\}}\|Y_i-\hat\theta_h^{(s)}\|$.
	\end{flushleft}
\end{algorithm}

\subsection{Mixture model and mislabeling guarantees}

\label{sec:upper_bound}
In this section we present the mislabeling upper bound achieved by the \codelta algorithm when a reasonable initialization is present. Our result is presented in terms of the decay condition $\decay$. To this end, we restate our full data generating model. 

Fix a monotonically decreasing function $\decay$ with $\lim_{x\to \infty} \decay(x)=0$. We say that a random variable $w$ is distributed according to a ${\decay}$-decay condition with a scale parameter $\sigma$, denoted by $w\in {\decay}_\sigma$, if $w$ satisfies the condition \ref{cond:G-decay}. We observe independent samples $Y_1,\dots,Y_n\in \reals^d$ from a mixture of $k$ many ${\decay}_\sigma$ distributions as follows:
\begin{align}\label{eq:model-genr}
	Y_i=\theta_{z_i}+w_i,\ i=1,\dots,n,\quad w_i\in {\decay}_\sigma,
	\quad z_i\in \{1,2,\dots,k\},\quad {\theta_h}\in \reals^d, h\in [k],
\end{align}
where $ z=\sth{z_i}_{i=1}^n\in[k]^n$ denotes the underlying true label vector, $\theta_1,\dots,\theta_k$ are the unknown centroids.
We study the mislabeling loss function between the estimated label vector $ \hat z=\sth{\hat z_i}_{i=1}^n$ and true label vector $ z=\sth{z_i}_{i=1}^n$ given by \eqref{eq:mislabeling}
\begin{align}
	\ell(\hat z,z)=\inf_{\pi\in \calS_k}\qth{\frac 1n\sum_{i=1}^n\indc{\pi(z_i)\neq \hat z_i}},
	\end{align}
where $\calS_k$ is the collection of all mappings from $[k]$ to $[k]$.

To better present our results, we first introduce some notations. For all $h,g\in [k]$, define 
\begin{equation}\label{eq:cluster-def}
	\begin{gathered}
		T_h^*=\sth{i\in [n]:z_i=h},
		{T}_h^{(s)}=\sth{i\in [n]:z_i^{(s-1)}=h}\\
		n^*_h = \abs{T^*_h}, n^{(s)}_h = \abs{T^{(s)}_h}, n^{(s)}_{hg} = \abs{T^*_h \cap T^{(s)}_g}
	\end{gathered}
\end{equation}
Note that for $s\geq 1$ this implies
\begin{align}
	{T}_h^{(s)}=\sth{i\in[n]: \|Y_i-\hat\theta_h^{(s-1)}\|\leq \|Y_i-\hat\theta_a^{(s-1)}\|,a\in[k]}.
\end{align}
with ties broken arbitrarily. Recall the minimum fraction of points in the data set that come from a single component defined previously in \ref{cond:alpha}
\begin{align}
	\alpha=\min_{g\in[k]} {n_{g}^*\over n}.
\end{align}
Define the cluster-wise correct labeling proportion at step $s$ as
\begin{align}
	H_s=\min_{g\in [k]}\sth{\min\sth{{n_{gg}^{(s)}\over n_g^*},{n_{gg}^{(s)}\over n_{g}^{(s)}}}}.
\end{align}
We denote by $\Delta = \min_{g\neq h\in [k]}\|\theta_g-\theta_h\|$ the minimum separation between the centroids. Let 
\begin{align}
	\Lambda_s = \max_{h\in[k]}\frac 1{\Delta} \|\hat \theta_h^{(s)}-\theta_h\|.
\end{align}
be the error rate of estimating the centroids at iteration $s$.
Our results are presented based on the signal-to-noise ratio in the model, defined as 
$$
\snr = {\Delta\over 2\sigma}.
$$
We have the following result.

\begin{theorem}
	\label{thm:main-genr}
	There exists a constant $c_0>0$ such that the following holds.
	\begin{itemize}
		\item  If we have an initial estimate of the label vector satisfying
	$H_0\geq \frac 12+\gamma$ for a $\gamma\in (\frac {10}{n\alpha},\frac 12)$,
 	then whenever $\snr\geq \decay^{-1}\pth{\exp\sth{-\frac {c_0}{\alpha\gamma}}}$, the output $\hat z^{(s)}$ of the \codelta algorithm at iteration $s$ with $\delta=\frac 12-\frac \gamma4$ achieves the expected mislabeling rate $$\EE\qth{\ell(\hat z^{(s)},z)}\leq k^2{\decay}\pth{\snr- \exp\sth{- \frac{c_0}\alpha}}
	+ 8ke^{-{n\alpha\over 4}},
	\quad s\geq 2.
	$$
	\item Instead of initial labels, if we have initial centroid estimates that satisfy
	$$\Lambda_0\leq \frac 12-{\decay^{-1}\pth{\exp\sth{- \frac{c_0}\alpha}} \over \snr},$$
	then the last conclusion holds with $\gamma=0.3$.
	\end{itemize}
\end{theorem}

\begin{remark}
	The above result shows that our algorithm reaches the desired level of mislabeling accuracy after only a couple of steps. This is essentially due to the robustness of the underlying centroid estimation algorithm, which is observed similarly in \cite{JKY2023}. As we increase the number of iterations, we expect to improve the mislabeling error. However, the corresponding theoretical analysis is beyond the scope of the current paper. For all practical purposes, we iterate the clustering process a pre-specified large number of times to obtain the final label estimates. 
\end{remark}

\subsection{Optimality of mislabeling: lower bound}

\label{sec:lower_bound}

In this section we show that when $\snr$ is significantly large, even if we have a good centroid initialization, the mislabeling error can be as high as $\Theta(\decay(\snr))$, up to factors depending on $k$. Suppose that $\theta^*_1,\dots,\theta^*_k$ are the true centroids, that are known to us, with $\min_{h\neq g}\|\theta_h^*-\theta_g^*\| = \|\theta_1^*-\theta_2^*\|=\Delta$. Consider the following set of parameters and label vectors
\begin{align*}
	\calP_0=\sth{\pth{z,\{\theta_i\}_{i=1}^k}: \quad \theta_i=\theta^*_i,i\in [k],\quad |\{i\in [n]:z_i=g\}|\geq {n\alpha},g\in [k]}
\end{align*} 
In addition, we assume that the decay function $\decay$ satisfies the following smoothness condition:
\begin{enumerate}[label=(Q)]
	\item \label{pt:decay-prop} There exists $c_\decay>0$ such that $\decay(\cdot)$ is differentiable in the interval $(c_\decay,\infty)$ and $\abs{\decay'(y)}|_{y\geq c_\decay}$ is monotonically decreasing.
\end{enumerate}
Then we have the following guarantee. The proof is provided in \prettyref{app:lower_bound}.
\begin{theorem}\label{thm:lb-general}
	Suppose that $\alpha\in (c/k,\bar c/k)$ for some constants $0<c<\bar c<1$. Then given any decay function $\decay$ satisfying \ref{pt:decay-prop}, there exists $C_{\decay}>0$ such that
	$$\inf_{\hat z}\sup_{\calP_0}\EE\qth{\ell(\hat z,z)}
	\geq {1-k\alpha-k/n \over 12}\cdot \decay(\snr + C_{\decay}),\quad \Delta\geq \sigma C_{\decay}.$$
\end{theorem}

\section{Applications to specific distributions}
\label{sec:specific_dist}

In this section, we showcase our general results for two specific mixture models with error distributions having sub-Gaussian tails and polynomial tails. 

\subsection{Sub-Gaussian mixture model}

In this model, the observed data $Y_1,\dots,Y_n\in \reals^d$ are distributed as
\begin{align}
	Y_i=\theta_{z_i}+w_i,\ i=1,\dots,n,
\end{align}
where $\sth{z_i}_{i=1}^n\in[k]^n$ denotes the underlying unknown label vector of the points, and $\{w_i\}_{i=1}^n$ denote the error variables distributed independently as zero mean sub-Gaussian vectors with parameter $\sigma>0$ (denoted by $w_i\in \subG(\sigma)$), i.e.,
\begin{align}\label{eq:subg-defn}
	\EE\qth{e^{\langle a,w_i \rangle}}
	\leq e^{\sigma^2\|a\|^2\over 2}, \text{ for all $i\in \sth{1,\dots,n}$ and $a\in \reals^d$}.
\end{align}
In order to apply our main results to the sub-Gaussian clustering problem, we need to derive a decay condition similar to $G_\sigma$. To this end, we note the next result  from Remark 2.2 of \cite{hsu2012tail}: given any $t>0$ and $w\in \subG(\sigma)$ we have
\begin{align}\label{eq:norm-conc}
	\PP\qth{\|w\|^2>\sigma^2\cdot (d+2\sqrt{dt}+2t)}\leq e^{-t}.
\end{align}
Simplifying the above, we get
\begin{align}
	\PP\qth{\|w\|>\sigma\cdot x}\leq \exp\pth{-(\sqrt{x^2-{d/2}}-\sqrt{d/2})^2/2},\quad x\geq \sqrt d.
\end{align}
Hence we apply \prettyref{thm:main-genr} with
\begin{align}
	\begin{gathered}
		G(x) = \exp\pth{-{(\sqrt{x^2-{d/2}}-\sqrt{d/2})^2/ 2}},\quad x\geq \sqrt d,\\
		G^{-1}(y) = \pth{2\log(1/y)+2\sqrt{d\log(1/y)}+d}^{1/2}
		\leq \sqrt{2\log (1/y)}+\sqrt d.
	\end{gathered}
\end{align}
In view of the above,  the next result directly follows.
\begin{corollary}
	\label{cor:subg-1}
	There are absolute constants $c_0$ and $c_1$ such that the following holds true. Fix $\gamma\in (\frac {10}{n\alpha},\frac 12)$ and suppose that the clustering initialization satisfies either one of the following conditions
		$${H_0}\geq \frac 12 +\gamma \quad  \text{ or } \quad \Lambda_0\leq \frac 12-{c_1(d+1/{\alpha})^{1/2}\over {\snr}}.$$
	Then, for sub-Gaussian $w_i$, whenever $\snr\geq c_0(d +1/{(\alpha\gamma)})^{1/2}$, the \codelta algorithm with $\delta=\frac 12-{\gamma\over 4}$ achieves the mislabeling rate
		$$
		\EE\qth{\ell(\hat z^{(s)},z)}\leq \exp\sth{-\frac 12\pth
			{\snr-c_1^2(d+1/{\alpha})^{1/2}}^2-2\log k},
		\quad s\geq 2.
		$$
\end{corollary}

\begin{remark}
	The implications of the above result are the following: whenever $\snr$ is significantly larger than $(d+1/(\alpha\gamma))^{1/2}$ and $\log k$, the mislabeling rate in the sub-Gausssian mixture model is approximately $\exp\pth{-{\Delta^2/(8\sigma^2)}}$. This matches the theoretical limit for mislabeling proportion in the sub-Gaussian mixture model with a constant $d$; see \cite{lu2016statistical} for an example which demonstrates that Lloyd's algorithm achieves a similar error rate. When $d$ is fixed, the initialization conditions stated above in \prettyref{cor:subg-1} are weaker than the conditions required for Lloyd's algorithm. In particular, the initialization condition on $H_0$ for Lloyd's algorithm depends on the relative distance between the closest cluster centroid and the farthest cluster centroids, given by $\lambda={\max_{h\neq g\in [k]}\|\theta_g-\theta_h\|/\Delta}$. As the value of $\lambda$ increases Lloyd's algorithm requires a stronger initialization condition to guarantee the optimal mislabeling. Notably, this dependency of initialization condition on $\lambda$ is necessary for Lloyd's algorithm to converge as the mean based centroid estimate for any cluster can be destabilized via contamination from the farthest away clusters. We believe that the dependency on $d$ in the condition involving $\snr$ can be further improved by first running a spectral method on the dataset and then applying the \codelta algorithm. However, the analysis is beyond the scope of the current paper.
\end{remark}

\subsection{Mixture models with moment constraints on the norm}

\label{sec:model-moment}
In this section, we explore the clustering guarantees of \codelta when the data generating distributions have moment constraints. We say that a random variable $w$ is distributed according to a $p$-th moment constraint on the norm with a scale parameter $\sigma$, denoted by $w\in \calR_p(\sigma)$ for a given $p > 0$, if it satisfies the following condition:
\begin{itemize}[label=(P)]
\item There exists $x_0>0$ such that $\PP\qth{\|w\|>x}<{\sigma^p\over x^p}$ for all $x\geq x_0$. Without a loss of generality we will assume $x_0 \geq \sigma$ as otherwise the bound is trivial.
\end{itemize}
 We observe independent samples $Y_1,\dots,Y_n\in \reals^d$ from a mixture of $k$ many $\calR_p(\sigma)$ distributions
\begin{align}
Y_i=\theta_{z_i}+w_i,\ i=1,\dots,n,\quad w_i\in \calR_p(\sigma),
\quad z_i\in \{1,2,\dots,k\},\quad {\theta_h}\in \reals^d, h\in [k],
\end{align}
where $ z=\sth{z_i}_{i=1}^n\in[k]^n$ denote the underlying label vector. The mislabeling proportion for the estimated label vector $\hat z$ produced by the \codelta algorithm is summarized as follows.

\begin{theorem}
	\label{thm:main-norm}
	Suppose that $\gamma\in (\frac {10}{n\alpha},\frac 12)$. Then there exist absolute constants $c_1,c_2>0$ such that the following holds. If the clustering initialization satisfies  
	$$H_0\geq \frac 12+\gamma\quad \text{or}\quad \Lambda_0\leq \frac 12-{e^{{c_1/p\alpha}}\over \snr},$$
	then whenever $\snr\geq  e^{{c_2/p\alpha\gamma}}$ 
	we have that the \codelta algorithm with $\delta=\frac 12-\frac \gamma4$ achieves the expected mislabeling rate $$\EE\qth{\ell(\hat z^{(s)},z)}\leq k^2 (\snr-e^{{2c_1/p\alpha}})^{-p}
	+ 8ke^{-{n\alpha\over 4}},
		\quad s\geq 2.
		$$
	In addition, this rate is optimal, up to a factor depending on $k,\alpha$.
\end{theorem}
Notably, in the above result, we never assume that the error distributions are centered around zero. As long as there is sufficient decay around the location parameter $\theta_h$, our result states that we should be able to produce good clustering guarantees.  Note that the second term is usually negligible.
%\subsection{Mixture models with coordinate-wise moment constraints} Difficult
%to remove the d-dependency even after using the mean as mean concentration is not strong enough

	\section{Provable initialization methods}

\label{sec:initialization}

In this section, we propose  centroid initialization algorithms which guarantee that the conditions on $\Lambda_0$ required in \prettyref{thm:main-genr} are met with a high probability. We deal with the cases $k=2$ and $k\geq 3$ separately. The algorithm in the case of $k\geq 3$ follows from a recursive structure which calls the algorithm for $k=2$ at the end.

\subsection{Two centroids}

	We first present our initialization algorithm for the setup with $k=2$. Our algorithm revolves around searching for data points with dense neighborhoods. With a high signal-to-noise ratio, such dense neighborhoods are expected to be close to the centroids. Hence, the data points with a high density neighborhoods can be chosen as good estimates of the true centroids. Our algorithm for finding such data points is presented in \prettyref{algo:hdp}: Given a data set with size $n$ and neighborhood size parameter $q$, the algorithm outputs a data point with the tightest neighborhood in the data set with at least $nq$ points from the set. 

\begin{algorithm}[H]
	\caption{{\sf The High Density Point} (${\sf HDP}_q$) - algorithm}\label{algo:hdp}
	\begin{flushleft}
		{\bf Input}:  Set of points $S=\sth{Y_1,\dots,Y_n}$, neighborhood size parameter $q$
	\end{flushleft}
	\begin{algorithmic}[1]
		\State Create distance matrix $D=\sth{D_{ij}: i,j\in [n],D_{i,j}=\|Y_i-Y_j\|}$
		\For{Each $i\in [n]$}
		\State Compute the $R_i$ as the $\ceil{nq}$-th smallest number in $\sth{D_{ij},j\in [n]}$
		\EndFor
		\State Find $i^* = \argmin_{i\in [n]} R_i$. 
	\end{algorithmic}
	\begin{flushleft}
		{\bf Output}: $Y_i^*$
	\end{flushleft}
\end{algorithm}
In view of the above, we present the initialization algorithm for $k=2$ below. 
\begin{algorithm}[H]
	\caption{The {\sf Initialization via Ordered Distances} ({$\iod_{2,m_1,m,\beta}$})- algorithm with 2 centroids}\label{algo:init}
	\begin{flushleft}
		{\bf Input}:  Data $Y_1,\dots,Y_n$, truncation parameter $\beta$, batch size $m$, and initial cluster size $m_1$
	\end{flushleft}
	\begin{algorithmic}[1]
		\State Compute $\mu_1^{(1)}=\hdp_{m_1\over n}(\{Y_1,\dots,Y_n\})$.
		\State Order the rest of the points in increasing Euclidean distance from $\mu^{(1)}_1$. 
		\State Denote the first $m_1$ points in the list as $\calP_1^{(1)}$ and the rest of the points list as $\overline{\calP_1^{(1)}}$ in increasing order of distance form $\mu_1^{(1)}$.
		\State Compute $\dist_1^{(1)}$ as the $(1-\beta)m_1$-th smallest value among the distances from $\mu_1^{(1)}$ to $\calP_1^{(1)}$. 
		\For{$\ell=1,\dots,\ceil{n-m_1\over m}$}
		\State Assign $\mu_1^{(\ell)}=\mu_1^{(1)}$. Compute $\dist_1^{(\ell)}$ as the $(1-\beta)\abs{\calP_1^{(\ell)}}$-th smallest value among the distances from $\mu_1^{(1)}$ to $\calP_1^{(\ell)}$ 
		\State Compute $\mu_2^{(\ell)} = {\sf HDP}_{1-\beta}(\overline{\calP_1^{(\ell)}})$.
		\State Compute $\dist_2^{(\ell)}$ as the $(1-\beta)\abs{\overline{\calP_1^{(\ell)}}}$-th smallest value among the distances from $\mu_2^{(\ell)}$ in the set $\overline{\calP_1^{(\ell)}}$.
		\State Store $\totdist^{(\ell)} = \dist_1^{(\ell)}+\dist^{(\ell)}_2$.
		\State Move the first $m$ points in the list $\overline{\calP_1^{(\ell)}}$ to $\calP_1^{(\ell)}$ to construct $\overline{\calP_1^{(\ell+1)}}, \calP_1^{(\ell+1)}$
		\EndFor
		\State Find $(\mu_1^*,\mu_2^*)=(\mu_1^{(\ell^*)},\mu_2^{(\ell^*)})$ and $\totdist^*=\totdist^{(\ell^*)}$ corresponding to 
		$$\ell^* = \argmin_{\ell\in \{1,\dots, \ceil{n-m_1\over m}-1\}}\totdist^{(\ell)}.$$
	\end{algorithmic}
	\begin{flushleft}
		\noindent {\bf Output}: $(\mu_1^*,\mu_2^*)$ and $\totdist^*$.
	\end{flushleft}
\end{algorithm}

\begin{remark}[Explanation of \prettyref{algo:init}]
	We apply \prettyref{algo:init} to the full data set $\calS=\{Y_1,\dots,Y_n\}$ with the choice $m_1={n\alpha\over 4}$. The first centroid estimate $\hat \theta_1$ is chosen by picking the index $i^*\in \{1,\dots,n\}$ such that the tightest neighborhood in $\calS$ of size $m_1$ around $Y_{i^*}$ has the smallest radius compared to any other such neighborhood around any other point $Y_i$. Using results on concentration of the quantiles of $\decay$ \prettyref{lmm:E-conc-init-2}, it is not too difficult to show that, for some constant $C=\sigma\tilde C_{G,\alpha}$ depending on the decay function $G$ and minimum cluster proportion $\alpha$, with a high probability
	$$
	\hat \theta_1=Y_{i^*}\in \cup_{i=1,2}\calB(\theta_i,C),$$ 
	where $\calB(x,R)$ denotes the Euclidean ball of radius $R$ around the point $x$. Without a loss of generality, suppose that $\calB(\theta_1,C)$ is the set containing $\hat\theta_1$. Denote the first $m_1$ points in the data set closest to $\hat\theta_1$ as $\calP_1$ and denote the complement set of points as $\overline{\calP_1}$. 
	
	In view of the above, it is clear that the inherent challenge in finding a good initialization lies in obtaining a good estimate of $\theta_2$. At this stage, it might seem reasonable to apply the $\sf HDP$ algorithm again on the remaining set of points $\overline{\calP_1}$ to estimate $\theta_2$. Unfortunately, a direct application of the ${\sf HDP}$ on the set $\overline{\calP_1}$ need not guarantee a good estimate of $\theta_2$. This is because as there are significantly many points in $\overline{\calP_1}\cap \sth{Y_i:i\in T^*_1}$ (at least $n^*_1-{n\alpha\over 4}\geq {3n\alpha\over 4}$ points) and the data point chosen via ${\sf HDP}$ can indeed belong to $\{Y_i:i\in T^*_1\}$, which will be closer to $\theta_1$ than $\theta_2$ with a high probability. To remedy this issue, we gradually move $m$ points from $\overline{\calP_1}$ to ${\calP_1}$, prioritizing the points in $\overline{\calP_1}$ that are closer to $\hat \theta_1$. At each transfer step, we can compute the corresponding centroid estimate $\hat \theta_2$, using ${\sf HDP}$ estimator, while keeping $\hat \theta_1$ as it is. To control the stopping point at which we terminate the transfer of points from $\overline{\calP_1}$ to ${\calP_1}$ we use the quantile-based measure $\totdist$. The reason behind using the quantiles of the intra-cluster distances rather than their sum, which is often used in the WCSS metric defined in \eqref{eq:wcss}, is that the quantiles are more robust to outlying observations. Notably, once we transfer a significant number of points from $\overline{\calP_1}$ that belong to $T^*_1$ and keep a substantial number of points in $\overline{\calP_1}$ that belong to $T^*_2$, a second application of the ${\sf HDP}$ algorithm will guarantee a good estimate $\hat \theta_2$ of $\theta_2$
\end{remark}

The following result describes our choice for the parameters $m_1,m,\beta$ in the above algorithm and the corresponding centroid estimation guarantees.
\begin{theorem}\label{thm:proof-algo-2}
	Suppose that out of the $n$ many observed data points, $n^*_i$ many are from cluster $T_i^*, i=1,2$ and assume that for some constant $\alpha>0$ the counts satisfy $n^*_1,n^*_2> n\alpha$. Then there are constants $c_1,c_2>0$ such that if $\Delta\geq c_1\sigma\decay^{-1}\pth{e^{-{c_2\over \alpha^2}}}$ then the $\iod_{2,m_1,m,\beta}$ algorithm with $m_1=\ceil{n\alpha\over 4},m=\max\{1,\floor{n\alpha^2\over 16}\},\beta={\alpha\over 4}$ guarantees, for a permutation $\pi$ on $\{1,2\}$ 
	$$
	\max_{i=1,2}\|\theta_{\pi(i)}-\mu_i^*\|\leq \Delta/3
	$$ 
	with probability at least $1-4e^{-{n\alpha/ 4}}$.
\end{theorem}

\begin{remark}\label{rmk:outlier-init2}
	 Our main result \prettyref{thm:main-genr} states that for a large enough $\snr$, the mislabeling guarantee for the \cod algorithm holds for any initial centroid estimates that satisfy $\Lambda_0\leq {1\over 2+c}$ for some constant $c>0$. In other words, given centroid estimates $\mu_1^*,\mu_2^*$ of $\theta_1,\theta_2$, it is sufficient to satisfy
	 \begin{align}
	 	\max_{i=1,2}\|\theta_{\pi(i)}-\mu_i^*\| 
	 	= \Delta \Lambda_0\leq {\Delta\over 2+c},
	 \end{align}
	 for some $c>0$. In view of \prettyref{thm:proof-algo-2},  our proposed initialization paired with the proposed \cod algorithm leads to the desired mislabeling.  
\end{remark}

The following result resolves the time complexity to run \prettyref{algo:init}.
\begin{theorem}
	The runtime of $\iod_{2,m_1,m,\beta}$ is at most $O\pth{{1\over \beta^2}\pth{n^2d+n^2\log n}}$.
\end{theorem}
\begin{proof}
	We first find the point in the data set with the tightest neighborhood of $m_1$ other points and the corresponding $(1-\beta)$-quantile of the distances from it. We have the following observations.
	\begin{itemize}
		\item Computing the tightest neighborhood of $m_1$ involves computing all the pairwise distances, which has a time complexity of $n^2d$. 
		\item Computing the $m_1/n$-quantile of the distances for all the points has a time complexity $O(n^2\log n)$, and then subsequently computing the minimum of those values takes at most $O(\log n)$ time.
	\end{itemize}
	Once we have found the first centroid, for each $1\leq \ell\leq {2/\beta^2}$ we construct $\calP_1^{(\ell)},\overline{\calP_1^{(\ell)}}$ according to distances from the first centroid. For each values of $\ell$, we repeat the following.
	\begin{itemize} 
		\item We find the $(1-\beta)$ quantiles of the distances from $\mu_1^{(1)}$ in $\calP_1^{(\ell)}$, which takes $n\log n$ time. 
		\item We then find the point in the data set with the tightest neighborhood of $(1-\beta)\abs{\overline{\calP_1^{(\ell)}}}$ other points in $\overline{\calP_1^{(\ell)}}$ and the corresponding $(1-\beta)$-quantile of the distances from it. This will again take at most $O(n^2d+n^2\log n)$ time similar to above considerations.
	\end{itemize} 
	Combining the above, we get that the total runtime is ${O(1)\over \beta^2}\pth{n^2d+n^2\log n}$.
\end{proof}

\subsection{Algorithm with a general $k$}

To extend the above algorithm for a general cluster number $k$ we use a recursive framework that utilizes the structure of \prettyref{algo:init}. We first locate a point from the data set that has the tightest neighborhood of size $m_1$ (denote it by $\calP$). This will serve as the first centroid estimate. Then for the remaining point set (call it $\overline{\calP}$) we recursively 
apply the initialization algorithm to find the best $k-1$ cluster centers. We repeat the process of finding the best $k-1$ cluster centroids from $\overline{\calP}$ after successively removing $m$ points from $\overline{\calP}$ and adding it to $\calP$. In each step, say $\ell$, we compute an appropriate distance measure similar to $\totdist^{(\ell)} = \dist_1^{(\ell)}+\dist_2^{(\ell)}$ in \prettyref{algo:cent}, that quantifies the goodness of the clustering at that step. Finally, the centroids generated in the step with the lowest distance measure chosen to be the final output. Whenever we are left with the task of finding two centroids from $\overline{\calP}$, we resort to ${\sf IOD}_{2,m_1,m,\beta}$. The details are provided in \prettyref{algo:init-k}.

%\begin{remark}
%		Consequences for adversarial outliers using \prettyref{lmm:E-conc-init} and \prettyref{lmm:E-conc-init-k}.
%\end{remark}

\begin{algorithm}[H]
	\caption{The {\sf Initialization via Ordered Distances} ({$\iod_{k,m_1,m,\beta}$})- algorithm}\label{algo:init-k}
	\begin{flushleft}
		{\bf Input}:  Data $\{Y_1,\dots,Y_n\}$, $k$ clusters to be found, truncation parameter $\beta$, batch size $m$, initial cluster size $m_1$.
	\end{flushleft}
	
	\begin{flushleft}
		{\bf Output}: Centroid estimates $\{\mu_i^{*}\}_{i=1}^k$ and Error measure $\totdist_k^*$.
	\end{flushleft}
	
	\begin{algorithmic}[1]
		
		\If{$k\geq 3$}
		\State \label{steps:old1} Compute $\mu^{(k,1)}_k=\hdp_{m_1\over n}(\{Y_1,\dots,Y_n\})$
		\State Denote the first $m_1$ points closest to $\mu_k^{(k,1)}$ as $\calP_k^{(1)}$ and the rest of the points as $\overline{\calP_k^{(1)}}$.
		\For{$\ell_k=1,\dots,\floor{n-m_1\over m}$}
		\State Set $\mu^{(k,\ell_k)}_k = \mu^{(k,1)}_k$ and compute 
		$$\text{$\dist_{k}^{(\ell_k)}$= 
			the distance to the $(1-\beta)|\calP_k^{(\ell_k)}|$-th closest point from $\mu^{(k,\ell_k)}_k$ in $\calP_k^{(\ell_k)}$}.$$
		\State Run the $\iod_{k-1,m_1,m,\beta}$ algorithm on the set $\overline{\calP_{k}^{(\ell_k)}}$ and note the outputs: 
		$$\text{centroid set $\{\mu_{i}^{(k,\ell_k)}\}_{i=1}^{k-1}$ and error measure as $\totdist_{k-1}^{(\ell_k)}$}.$$
		
		\State Store $\totdist_{k}^{(\ell_k)} = \dist_k^{(\ell_k)}+\totdist_{k-1}^{(\ell_k)}$.
		\State Move the first $m$ points in $\overline{\calP_k^{(\ell_k)}}$, that are closer to $\mu_k^{(k,1)}$, to ${\calP_k^{(\ell_k)}}$ to construct $\overline{\calP_k^{(\ell_k+1)}}, \calP_k^{(\ell_k+1)}$.
		\EndFor
		\State  $\ell_k^* = \argmin_{\ell_k}\totdist_k^{(\ell_k)},\ \{\mu_i^{*}\}_{i=1}^k=\{\mu_i^{(k,\ell_k^*)}\}_{i=1}^k,
		\ \totdist_{k} = \totdist_{k}^{(\ell_k^*)}$.
		\ElsIf{k=2}
		\State Run the $\iod_{2,m_1,m,\beta}$ algorithm and note the output as $\{\mu_1^{*},\mu_2^{*}\}$ and $\totdist_k^*$.
		\EndIf
	\end{algorithmic}
\end{algorithm}

The following result describes a choice of the parameters $m_1,m,\beta$ that guarantees a good initialization, sufficient to meet the requirements related to $\Lambda_0$ in \prettyref{thm:main-genr}. Hence, our initialization algorithm, paired with the clustering technique $\sf COD$, produces the desired mislabeling starting from scratch.

\begin{theorem}\label{thm:proof-algo-k}
	Suppose that out of the $n$ many observed data points, $n^*_i$ many are from cluster $T_i^*, i=1,\dots,k$ and for some constant $\alpha>0$ the counts satisfy $n^*_i> {n\alpha}, i=1,\dots,k$. Then there are constants $c_1,c_2$ such that the following is satisfied.  Whenever $\Delta>c_1k\sigma\decay^{-1}\pth{e^{-{c_2/ \beta^2}}}$, there is a permutation $\pi$ of the set $[k]$ that satisfies $\max_{i\in [k]}\|\theta_{\pi(i)}-\mu_i^*\|\leq \Delta/3$ with probability at least $1-2ke^{-n\alpha/4}$, where the $\{\mu_i^*\}_{i=1}^k$ are centroids generated via the $\iod_{k,m_1,m,\beta}$ algorithm with 
	$$m_1=\ceil{n\alpha\over 4},m=\max\sth{1,\floor{n\beta^2\over 2}},\beta={\alpha\over 4k}.$$
\end{theorem}
%\begin{remark}
%	Similar to \prettyref{rmk:outlier-init2}, for a large enough $\snr$, the \cod algorithm obtains the optimal mislabeling for any initial centroid approximates that satisfies $\Lambda_0\leq {1\over 2+c}$ for some constant $c>0$. In other words, given centroid estimates $\mu_1^*,\dots,\mu_k^*$ of $\theta_1,\dots,\theta_k$, in terms of centroid approximation error, it is sufficient to have
%	\begin{align}
	%	\max_{i=1,\dots,k}\|\theta_{\pi(i)}-\mu_i^*\| 
%		= \Delta \Lambda_0\leq {\Delta\over 2+c},
%	\end{align}
	%for some $c>0$. In view of \prettyref{thm:proof-algo-k}, this implies that our initialization paired with the \cod algorithm leads to an optimal mislabeling.  
%\end{remark}

In view of \prettyref{thm:proof-algo-k},  our initialization paired with the \cod algorithm leads to the desired mislabeling.  Notably, the Lloyd algorithm \cite{lu2016statistical} and the hybrid $k$-median algorithm in \cite{JKY2023} also required the initialization condition $\Lambda_0<{ 1/(2+c)}$, for a constant $c>0$, to produce the optimal mislabeling rate in the sub-Gaussian clustering problem. In view of \prettyref{sec:specific_dist}, and the proof of \prettyref{thm:proof-algo-k} in \prettyref{app:proof-init-k}, we note that \prettyref{thm:proof-algo-k} will require $\Delta\geq \sigma c_{\alpha} \sqrt d$ for some constant $c_{\alpha}$ depending on $\alpha$. This implies the following. 
\begin{corollary}
	There is a constant $c_{\alpha}$ depending on $\alpha$ such that the following holds true. When initialized with our initialization scheme \iod, the hybrid $k$-median algorithm in \cite{JKY2023} and Lloyd's algorithm \cite{lu2016statistical} produce the optimal mislabeling rate in the sub-Gaussian error setup, provided $\Delta>\sigma c_{\alpha} \sqrt d$ and $d$ is bounded by a constant.
\end{corollary}
The following result resolves the time complexity to run \prettyref{algo:init-k}.
\begin{theorem}\label{thm:runtime-init-k}
	The runtime of $\iod_{k,m_1,m,\beta}$ is at most $ {(O(1)/\beta^2)^{k-1}(n^2d + n^2\log n)}$.
\end{theorem}
\begin{proof}
	As our method is a recursive process, we construct a recursion that relates the computation time of finding the best $k$ centroids to that of finding $k-1$ best centroids. In the recursion process, when we want to find out the best $k$ centroids from the data, we first find the point in the data set with the tightest neighborhood of $m_1$ other points and the corresponding ${m_1\over n}$-quantile of the distances from it. This involves the computation of all the pairwise distances, which has a time complexity of $O(n^2d)$, computing the $(1-\beta)$-quantile of the distances for all the points, which has a time complexity $O(n^2\log n)$, and finally computing the minimum which takes $\log n$ at most. Once we have found the first centroid, for each $1\leq \ell_k\leq {2/\beta^2}$ we construct $\calP_k^{(\ell_k)},\overline{\calP_k^{(\ell_k)}}$ according to distances from the first centroid, which takes another $n$ unit time and perform the $k-1$ centroid finding algorithm on $\overline{\calP_k^{(\ell_k)}}$ which has at most $n$ points. Let $U_k$ be the time complexity of finding the best $k$-centroids given $n$ data points. Then, given the above reasoning, we have
	\begin{align*}
		U_k \leq (O(1)/\beta^2) \qth{U_{k-1} + n} + O(n^2d + n^2\log n).
	\end{align*}
	Solving the above recursion we get
	\begin{align*}
		U_k \leq (O(1)/\beta^2)^{k-2} U_2 + (O(1)/\beta^2)^{k-2}\qth{n^2d + n^2\log n + 2n/\beta^2}.
	\end{align*}note that via a similar argumentthe 2-centroid finding problem takes ${O(1)\over \beta^2}\pth{n^2d + n^2\log n}$ time. Hence, we simplify the above to get the desired result. 
\end{proof}

	\section{Clustering in the presence of adversarial outliers}

\label{sec:adv_outliers}

In this section we show that the inherent robustness of our algorithms extends to the scenario when a constant fraction of data (for a finite $k$) might be adversarial outliers. We study the setup where an adversary, after accessing the original data set, adds $n^{\out}$ many new points of its choice. Our results in this section primarily address how the previous theoretical guarantees change when outliers might be present and how much adversarial contamination our initialization methods can tolerate without misbehaving significantly. We do not aim to optimize the outlier levels that our algorithms can tolerate; it is left for future works.

We first discuss the extension of \prettyref{thm:main-genr}. To retain the above mislabeling guarantees, we need to apply a higher value of the truncation parameter $\delta$ for running $\sf{TM}_{\delta}$. We have the following guarantees.
\begin{theorem}
	\label{thm:out-genr}
	Suppose that an adversary, after analyzing the data $Y_1,\dots,Y_n$ coming from the general mixture model \eqref{eq:model-genr}, adds $n^{\out}=n\alpha(1-\psi)$ many outliers of its choice for some $\psi\in (0,1]$. Then there exists a constant $c_0>0$ such that the following holds. \begin{itemize}
		\item If we have an initial estimate of the label vector satisfying $H_0\geq \frac 12+\gamma$ with $\gamma\in (\frac {10}{n\alpha},\frac 12)$,
	then whenever $\snr\geq\decay^{-1}\pth{\exp\sth{-\frac {c_0}{\alpha\cdot\min\{\gamma,\psi\}}}}$ 
	we have that the label vector output $\hat z^{(s)}$ obtained from the \codelta algorithm after $s$ iterations, with $\delta=\frac 12-\frac 14 \min\sth{\gamma,\frac \psi6}$, achieves $$\EE\qth{\ell(\hat z^{(s)},z)}\leq k^2{\decay}\pth{\snr- \decay^{-1}\pth{\exp\sth{-\frac {c_0}{\alpha\psi}}}}
	+ 8ke^{-{n\alpha\over 4}},
	\quad s\geq 2.
	$$
	
	\item Instead of initial cluster labels, if we have initial centroid estimates satisfying $\Lambda_0\leq \frac 12-{\decay^{-1}\pth{\exp\sth{-\frac {c_0}{\alpha\psi}}} \over \snr}$, then the above conclusion holds with $\gamma=0.3$.
	\end{itemize}
\end{theorem}
\begin{remark}
	Since the adversarial data are arbitrary, the result also applies to the case where the number of clusters is undetermined. In that case, we can simply regard data beyond the first $k$ class as adversarial attacks, so long as the total number of such data points is not too large. Then, Theorem~\ref{thm:out-genr} gives an accurate bound on the first $k$ clusters.
\end{remark}

\begin{remark}
	The recent work of \cite{JKY2023} studies a similar setup of adversarial outliers, although in the specific case of sub-Gaussian mixture models. The previous work uses the coordinate-wise median for centroid estimation. We expect that the similar theoretical guarantees as in \prettyref{thm:out-genr} for the above coordinate-wise median based clustering algorithm will require $\snr\geq \sqrt d \decay^{-1}{\exp\sth{-{c_0\over \alpha\cdot \min{\gamma,\psi}}}}$, which is worse than the results in this paper by a factor of $\sqrt d$. This is because in the worst case, the outliers can equally impact the estimation guarantees of the coordinate-wise median in each coordinate, which produces this $\sqrt d$ factor. Our algorithm can avoid this extra factor of $\sqrt d$ as it utilizes the underlying Euclidean distance-based structure of the error distributions, in contrast to the algorithm based on the coordinate-wise median, which originates from an $\ell_1$-norm based minimization procedure.
\end{remark}

\begin{remark}
	Our method requires that the proportion of outliers is, at most, the proportion of points in the smallest cluster. It is a requirement to produce a small mislabeling as adversarial setups that allow larger outliers than the above can allow the adversary to substitute the smallest cluster with a different cluster of their choice and mislead any reasonable clustering algorithm. 
\end{remark}

Next, we discuss the results regarding our initialization algorithms in the presence of adversarial outliers.

\begin{theorem}\label{thm:proof-algo-outliers}
Suppose that out of the $n$ many observed data points, $n^*_i$ many are from cluster $T_i^*, i=1,\dots,k$ and $n^\out$ many are adversarial outliers (i.e., $\sum_{i=1}^kn^*_i+n^\out=n$). Also, assume that for some constant $\alpha>0$ the counts satisfy $n^*_i> {n\alpha}, i=1,\dots,k$.  We apply \prettyref{algo:init} for $k=2$ and \prettyref{algo:init-k} for $k\geq 3$. Then the consequences of \prettyref{thm:proof-algo-2} (i.e., for $k=2$) hold if $n^\out\leq {n\alpha^2\over 32}$ and the consequences of \prettyref{thm:proof-algo-k} (i.e., for a general $k\geq 3$) hold if $n^\out\leq {n\alpha^2\over 64k}$.
\end{theorem}

\begin{remark}
	The above result guarantees that for $k=O(1),\alpha>c/k>0$, our clustering algorithm can tolerate $O(n)$ outliers in the data and still retain the initialization guarantees similar to the case without outliers. It is reasonable to comment that with an increase in the number of clusters, i.e., comparatively fewer points in each cluster, the tolerable amount of adversarial outliers should decrease for any algorithm. However, optimizing the number of adversarial outliers our algorithm can tolerate for a general $k$ is beyond the scope of the current work and is left for future directions. 
\end{remark}
	
	\section{Suboptimality of Lloyd's algorithm}

\label{sec:subopt_lloyd}

In this section, we establish that Lloyd's algorithm might produce a suboptimal mislabeling even when the initial labels are reasonably good.
In the case of at least three centroids, even when error distributions have bounded support, if one of the centroids is far away from the rest, then the mislabeled points originating from that centroid can destabilize the cluster means and hence lead to the poorly estimated centroid. In the two-centroid setup,  the suboptimality occurs when error distributions exhibit heavy tails.

\subsection{The case of at least three centroids}
This section assumes that whenever Lloyd's algorithm produces an empty cluster, it randomly picks one of the data points as the missing centroid for the next iteration step. Then, we have the following result.
\begin{lemma}
	Given any $\beta \in (0, 1)$, there exists a system of three centroids and an initialization with mislabeling proportion $\beta$ such that Lloyd's algorithm does not produce better than a constant proportion of mislabeling. 
\end{lemma}
\begin{proof}
	We consider the one dimensional setup with three centroids, located at $-{\Delta\over 2}, {\Delta\over 2}$ and ${c\Delta\over 2\beta}$ for some constant $c>2$ and sufficiently large $\Delta$. Consider the data generating model
	\begin{align}
		\begin{gathered}
			Y_i = \theta_{z_i}+w_i,\ i=1,\dots, n,\\
			w_i\iiddistr~ \Unif(-1,1),\ z_i\in \{1,2,3\},\ \theta_1= -{\Delta\over 2}, \theta_2={\Delta\over 2}, \theta_3={c\Delta\over 2\beta}.
		\end{gathered}
	\end{align}
	Let $T^*_h=\{i\in [n]:z_i=h\}, h\in \{1,2,3\}$ be as before. We assume equal number of points in all three clusters, i.e., $|n_h^*|={n/3}$. To define the initial label estimates, choose any $\ceil{n\beta/3}$ points from $T^*_3$, say $S$ and consider the initialization
	\begin{align}
		\hat z_i^{(0)} = 
		\begin{cases}
			2 \quad &\text{if }i\in S,\\
			z_i\quad &\text{otherwise}.
		\end{cases}
	\end{align}  
    This is a good initialization, except for a fraction of $\beta$ mislabels in class 2.
    
	We now study the iteration steps for Lloyd's algorithm. After the first iteration, assuming $\Delta$ is sufficiently large, the centroid estimates satisfy
	\begin{align*}
		\theta_1^{(0)}
		\in \pth{-{\Delta\over 2}-1, {\Delta\over 2}+1},
		\quad 
		\hat \theta_2^{(0)}\in \pth{{(c+1)\Delta\over 2(1+\beta)}-1,{(c+1)\Delta\over 2(1+\beta)}+1},
		\quad  \theta_3^{(0)}
		\in \pth{{c\Delta\over 2\beta}-1, {c\Delta\over 2\beta}+1}.
	\end{align*}
	Note that the above implies that given any data point, it is either closer to $\hat \theta_1^{(0)}$  or to $\hat \theta_3^{(0)}$, depending on whether the data is from clusters 1 and 2 or from cluster 3. As a result, $T_2^{(1)}$ is empty, and we randomly pick one of the data points as $\hat \theta_2^{(1)}$. With a constant probability, the choice is given by one of the points in $\{Y_i:i\in T^*_3\}$. In that scenario, in all subsequent stages $\hat \theta_2^{(s)},\hat \theta_3^{(s)}$ will continue to be inside the interval $({c\Delta\over 2\beta}-1,{c\Delta\over 2\beta}+1)$. As a result, all the points from $T^*_2$ are mislabeled. This shows that with constant probability we will have a constant proportion of mislabeling even if all possible label permutations are considered.
\end{proof}
\subsection{The case of two centroids}

We produce a counter example where Lloyd's algorithm fails even with a good initialization. Fix $\epsilon\in (0,1)$. Given any $\Delta>0$ we choose a sample size so big that $n^{\epsilon}> 4\Delta$. Next consider the decay function
\begin{align}
	\decay(x) = \frac 1{1+x^{1-\epsilon}}, \quad  x> 0.
\end{align}
The model we use is 
\begin{align}
	\begin{gathered}
		Y_i = \theta_{z_i}+w_i,\ i=1,\dots, n,\\
		w_i\iiddistr~ W,\quad \PP\qth{W>x} = \decay(x), x>0, \quad z_i\in \{1,2\},\quad \theta_1= 0, \theta_2={\Delta},
	\end{gathered}
\end{align}
with equal cluster sizes. Then, given $n$ samples from the above mixture model, we have
\begin{align}
	\PP\qth{\cup_{i=1}^n\sth{w_i>n^{1+\epsilon}}}
	&= 1 - \PP\qth{\cap_{i=1}^n\sth{w_i\leq n^{1+\epsilon}}}
	\nonumber\\
	&= 1 - \Pi_{i=1}^n \PP\qth{w_i\leq n^{1+\epsilon}}
	= 1 - \pth{1 - \frac 1{1 + n^{1-\epsilon^2}}}^n
	\geq 1 - e^{-n^{\epsilon^2}}.
\end{align}
This implies that with probability at least 1/2 there is at least one index $i^*$ such that $w_{i^*}> n^{1+\epsilon}$. Then, whichever cluster contains $Y_{i^*}$, its corresponding centroid estimate will be bigger than $n^{\epsilon}$. Notably, in the next step, when we use the Euclidean distance to cluster estimate, the best estimated clusters will be of the form
\begin{align*}
	{T}_1^{(s+1)} = \{i\in [n]: Y_i\in [0,x]\}, 
	{T}_2^{(s+1)} = \{i\in [n]: Y_i\in (x,\infty)\},\
	x={(\hat \theta_1^{(s)}+\hat \theta_2^{(s)})/ 2}.
\end{align*}
As one of the centroid estimates is bigger than $n^{\epsilon}$ we get that $x\geq n^{\epsilon}/2\geq 2\Delta$. Next, we present the following concentration result.

\begin{lemma}
	\label{lmm:E-conc-genr}
	Fix $\epsilon_0 > 0$. Then there is an event $\econ_{\epsilon_0}$ with probability at least $1-k\cdot e^{-{\min_{g\in [k]}{n_g^*}\over 4}}$ on which 
	$$\sum_{i\in T_g^*}\indc{\epsilon\Delta\leq \|w_i\|}
	\leq \frac {5n_g^*}{4 \log(1/\decay(\epsilon_0\Delta/{\sigma}))}, \quad \epsilon\geq \epsilon_0,\forall g\in [k].$$
\end{lemma}

A proof of the above result is presented at the end of this section. Note that in view of \prettyref{lmm:E-conc-genr}, for all large enough $\Delta$ and $n$ we have
\begin{align*}
	\PP\qth{\sum_{i\in T_h^*}\indc{w_i<\Delta} > \frac {3n}8, \ h\in \{1,2\}}\geq \frac 34.
\end{align*}
In view of $x\geq 2\Delta$, using the above inequality conditioned on the event $\cup_{i=1}^n \{w_i>n^{1+\epsilon}\}$ we have that
\begin{align*}
	\PP\qth{|\hat{T}_1^{(s+1)}\cap T_h^*| \geq {3n\over 8}, \ h\in \{1,2\}}\geq \frac 34.
\end{align*}
Hence, on the event $\cup_{i=1}^n \{w_i>n^{1+\epsilon}\}$, that has a probability at least 1/2, there will be at least ${3n\over 8}$ points that are mislabeled.

\begin{proof}[Proof of \prettyref{lmm:E-conc-genr}]
	We define $B_i=\indc{\epsilon_0\Delta\leq \|w_i\|}$. As $\epsilon\geq \epsilon_0$, it is enough to find an event $\calE_1$ with the said probability on which
	\begin{align}
		\label{eq:km8-genr}
		\PP\qth{\sum_{i\in T_g^*} B_{i} \geq \frac {5n_g^*}{4 \log(1/\decay(\epsilon_0\Delta/\sigma))}}\leq e^{-{n_g^*\over 4}} \text{ for each } g\in [k].
	\end{align}
	Note that
	\begin{align}
		\PP\qth{\epsilon_0\Delta\leq \|w_i\|}
		\leq \decay(\epsilon_0\Delta/\sigma)
		.
	\end{align}
	This implies $\sum_{i\in T_g^*}B_i$ is stochastically smaller than a random variable distributed as $\Binom(n_g^*,\decay(\epsilon_0\Delta/\sigma))$. We continue to analyze \eqref{eq:km8-genr} via Chernoff's inequality in \prettyref{lmm:chernoff} for a random variable with the $\Binom(n_g^*,\decay(\epsilon_0\Delta/\sigma))$ distribution. Define $$a= \frac {5}{4 \log(1/\decay(\epsilon_0\Delta/\sigma))},\quad m=n_g^*,\quad  q=e^{-5/(4a)} = \decay(\epsilon_0\Delta/\sigma).$$
	Then we have $a=\frac 5{4\log(1/q)}>\frac 1{\log(1/q)}\geq q$. Using $a\log a\geq -0.5$ for $a\in (0,1)$ we get
	\begin{align}\label{eq:km7-genr}
		\PP\qth{\sum_{i\in T_g^*} B_{i} \geq \frac {5n_g^*}{4 \log(1/\decay(\epsilon_0\Delta/\sigma))}}
		&\leq \exp\pth{-mh_q(a)}
		\nonumber\\
		&\leq \exp\pth{-m\pth{a\log{a\over q}+(1-a)\log{1-a\over 1-q}}}
		\nonumber\\
		&\leq \exp\pth{-m\sth{a\log{a\over e^{-{5/ (4a)}}}+(1-a)\log(1-a)}}
		\nonumber\\
		&=  \exp\pth{-m\sth{a\log a+(1-a)\log(1-a)+\frac 54}}
		\leq e^{-n_g^*/4}.
	\end{align}
	
\end{proof}

\section{Experiments}

\label{sec:expt}

\subsection{Synthetic datasets}
In this section, we evaluate our proposed algorithm ({\sf IOD} for initialization and {\sf COD} for clustering) on synthetic datasets and compare its performance in terms of the mislabeling proportion with the classical Lloyd's algorithm (e.g., the Lloyd–Forgy algorithm \cite{lloyd1982least}). For initializing Lloyd's algorithm, we consider three methods: 
\begin{itemize}
	\item the  proposed $\iod$ algorithm
	\item the $k$-means++ algorithm \cite{vassilvitskii2006k}
	\item randomly chosen initial centroid estimates from the dataset.
\end{itemize}
We simulate the data points with the errors $\{w_i\}$ independently from the multivariate $t_\nu$-distribution with a scale parameter $\sigma$, i.e., the $w_i$ random variable has a density
\begin{align}\label{eq:t-dist}
	f(x) = {\Gamma((\nu+d)/2)\over \Gamma(\nu/2)\nu^{d/2}\pi^{d/2}\sigma}
	\qth{1+ {\|x\|^2\over \sigma \nu}}^{-(\nu+d)/2}.
\end{align}
We study the effect of different dimensions $d$, degrees of freedom $\nu$ for the $t$ distribution, and the scale parameter $\sigma$. We consider the number of centroids $k=2,3$ for our experiments. The centroids of the cluster components are generated randomly and then scaled to make sure that they are at least 25 units apart. For each of the clusters, we generate 200 data points. When running the ${\sf IOD}$ initialization method in \prettyref{algo:init}, \prettyref{algo:init-k} and the {\sf COD} clustering method in \prettyref{algo:cod}, we use the parameters
$$m_1 = 20,m = 10,\beta = 0.05, \delta = 0.3.$$ 
Our experiments are divided into the following regimes.
\begin{itemize}
	
	\item {\it Different degrees of freedom.} We fix the data dimension $d=5$ and $\sigma=5$. We vary the degrees of freedom $\nu$ in the set $
	\{1, 1.5, 10\}$ to cover the cases of a very heavy tail where the mean does not exist, a moderately heavy tail where the mean exists but variance does not, and finally a very light tail where other higher moments exist.
	
	\item {\it Different scale parameters.} We fix the data dimension $d=10$ and $\nu = 1.5$. We vary the scale parameter $\sigma$ in the set $
	\{1, 5, 10\}$ to cover the cases of large, moderate, and low signal-to-noise ratios, respectively.
	\item {\it Different dimensions.} The true points are generated with $\nu = 1.5, \sigma=5$. We vary the data dimension $d$ in the set $\{2, 10, 30\}$.
\end{itemize}
We repeat all the experiment setups 150 times to estimate the mislabeling proportion and its standard error. The average mislabeling errors are presented in \prettyref{tab:nu}, \prettyref{tab:sigma}, \prettyref{tab:dim} (along with the standard errors within the parenthesis). 

\subsubsection*{Results}
We first present the numerical study describing the effect of $\nu$ \prettyref{tab:nu}. For the large value of $\nu=10$ the data are supposed to be highly concentrated around the centroids, which should guarantee a low mislabeling error. Lloyd's algorithm should work well in such a light tail setup, even though its mislabeling optimality is unknown. Nonetheless, our simulations demonstrate a low mislabeling error for all the algorithms for both $k=2,3$. As we consider heavier tails by decreasing $\nu$ to 1.5, we observe a steep increase in the mislabeling error for all the methods, although our algorithm produces the best performance. Notably, Lloyd's algorithm, when paired with our proposed {\sf IOD} initialization method, improves on the performance of the classical $k$-means++ initialization technique. However, further decreasing $\nu$ to 1, a setup where even the population mean does not exist, all instances of Lloyd's-type methods perform equally poorly, while our algorithm produces significantly lower mislabeling errors. 

\begin{table}[H]
	\centering
	\caption{Effect of degrees of freedom: $n = 200k, \sigma=5 , d=5, \Delta = 25$}\label{tab:nu}
	\begin{tabular}{|c|c|c|c|c|c|}
		\hline
		$k$ & $\nu$ & {\sf COD} + {\sf IOD} & Lloyd + {\sf IOD} & Lloyd + $k$-means++ & Lloyd + random init\\
		\hline
		\multirow{3}{*}{2} & 1 & 0.322
		 (0.011)	&	0.495
		  (0.002)	& 0.498
		   (0.000) & 0.497
		    (0.000)
		\\
		& 1.5 & 0.128
		 (0.001)	& 0.322 (0.014)	& 0.48 (0.006) & 0.366 (0.014)\\
		& 10 & 0.014 (0.000)	& 0.013 (0.000) & 0.014 (0.000) & 0.014 (0.000)
		\\
		\hline
		\multirow{3}{*}{3} &1 & 0.422
		 (0.005)	&	0.652 (0.004)	& 0.664 (0.000) & 0.65 (0.005)
		\\
		& 1.5 & 0.364
		 (0.007)	& 0.411 (0.009) & 0.576 (0.011) & 0.403 (0.013)\\
		& 10 & 0.043 (0.008) & 0.034 (0.007) & 0.014 (0.000)
		& 0.081 (0.013)
		\\
		\hline
	\end{tabular}
\end{table}

Next, we demonstrate the effect of the scale parameter $\sigma$ in \prettyref{tab:sigma}. For a fixed $\nu,\Delta$ this amounts to studying the effect of $\snr={\Delta\over 2\sigma}$ on the mislabeling error. The proportion of mislabeling should decay with large $\snr$, or equivalently with low $\sigma$ values, and this is supported by our demonstrations. Additionally, in all the setups, our algorithm performs significantly better than its competitors.

\begin{table}[H]
	\centering
	\caption{Effect of scale: $n = 200k, \nu=1.5 , d=10, \Delta = 25$}
	\label{tab:sigma}
	\begin{tabular}{|c|c|c|c|c|c|}
		\hline
		$k$ & $\sigma$ & {\sf COD} + {\sf IOD} & Lloyd + {\sf IOD} & Lloyd + $k$-means++ & Lloyd + random init\\
		\hline
		\multirow{3}{*}{2} 
		& 1 & 0.014 (0.000)& 0.029 (0.006)& 0.274 (0.018)& 0.1 (0.014)\\
		& 5  &0.173 (0.003)&	0.424 (0.006)& 0.496 (0.001)&	0.451 (0.005)
		\\
		& 10 &0.352 (0.005)&	0.492 (0.001)&	0.497 (0.000)&	0.495 (0.001)
		\\
		\hline
		\multirow{3}{*}{3} & 1 & 0.161
		 (0.012)	&	0.169 (0.012)	& 0.27 (0.017) & 0.169 (0.013)
		\\
		& 5 & 0.412
		 (0.001)	& 0.485
		  (0.006) & 0.654
		   (0.003) & 0.53
		    (0.008) \\
		& 10 & 0.509
		 (0.003)	& 0.628
		  (0.005) & 0.664
		   (0.000) & 0.647
		    (0.004) \\
		\hline
	\end{tabular}
\end{table}

In \prettyref{tab:dim}, we demonstrate how the data dimensions affect the performance of our algorithm. As the data dimension increases while keeping the centroid separation fixed, the performance of the clustering algorithm deteriorates. This is because the norm of the error random variables increases proportionally to the square root of the dimension, multiplied with variability in each coordinate. Nonetheless, our proposed clustering algorithm performs more robustly compared to the other methods in the simulation studies. It might be possible to improve all the clustering techniques by applying some dimension reduction, for example, feature screening approaches \cite{fan2008high} and the spectral methods in \cite{loffler2021optimality}, to the data set before applying the clustering methods. However, such analysis is beyond the scope of the current work.

\begin{table}[H]
	\centering
	\caption{Effect of dimension: $n = 200k, \nu=1.5 , \sigma = 5, \Delta = 25$}
	\label{tab:dim}
	\begin{tabular}{|c|c|c|c|c|c|}
		\hline
		$k$ & $d$ & {\sf COD} + {\sf IOD} & Lloyd + {\sf IOD} & Lloyd + $k$-means++ & $k$-means + random init\\
		\hline
		\multirow{3}{*}{2} & 2 & 0.099 (0.001) &	0.154 (0.007) &	0.398 (0.01) &	0.231 (0.01)\\
		& 10 & 0.174 (0.004) & 0.414 (0.008) & 0.495 (0.001) & 0.445 (0.006)
		\\
		&30 & 0.309 (0.01) & 0.492 (0.002) & 0.497 (0.000) & 0.494 (0.002)
		\\
		\hline
		\multirow{3}{*}{3} & 2 & 0.156 (0.008)	& 0.2 (0.009)	& 0.38 (0.009) &	0.236 (0.009)
		\\
		& 10 & 0.41 (0.002) &	0.479 (0.009) &	0.655 (0.004) &	0.528 (0.011)
		\\
		&30 & 0.467 (0.002) &	0.64 (0.002) &	0.664 (0.000) &	0.653 (0.004)
		\\
		\hline
	\end{tabular}
\end{table}

\subsection{Real data experiments}
Furthermore, we evaluated our proposed algorithm on the publicly available Letter Recognition dataset \cite{misc_letter_recognition_59}. The data set contains 16 primitive numerical attributes (statistical moments and edge counts) of black-and-white rectangular pixel displays of the 26 capital letters in the English alphabet. The character images were based on 20 different fonts, and each letter within these 20 fonts was randomly distorted to produce a file of 20,000 unique stimuli. We apply our proposed algorithm to this data with the aim of clustering data points corresponding to the same letters together. Additionally, we explore the robustness guarantees of the algorithms when some small number of data points corresponding to other contaminating letter classes are also present. In that setup, the goal is to minimize the mislabeling error corresponding to the letter classes with larger sample sizes.

\subsubsection*{Experiment setup}

For our experiment, we consider two and three-cluster setups. In the two-cluster setup, we pick the data points corresponding to the letters ``W" and ``V" as the clusters, and in the three-cluster setup, we pick the data points corresponding to the letters ``X", ``M", and ``A". We randomly sample 100 points from each cluster in both setups to simulate a contamination-free setup. To introduce outliers, in each scenario, we add 20 randomly chosen data points corresponding to the letter ``R". Once the data set is prepared, we apply the following clustering algorithms
\begin{itemize}
	\item the proposed $\iod$ initialization algorithm and {\sf COD} clustering algorithm
	\item the hybrid $k$-median algorithm in \cite{JKY2023} for clustering in the presence of adversarial outliers, initialized with the ${\sf IOD}$ algorithm
	\item Lloyd's algorithm initialized with the ${\sf IOD}$ algorithm
	\item Lloyd's algorithm initialized with the $k$-means++ algorithm
	\item Lloyd's algorithm initializations from the dataset.
\end{itemize}
The relevant parameters for the clustering are the same as those in the simulation studies section, with the only modification being for the value of $\delta$. This is in accordance with \prettyref{thm:out-genr}, which proposes that in the presence of outliers, it is meaningful to choose a more robust clustering algorithm, which corresponds to a higher value of $\delta$. For our studies, we fix $\delta = 0.48$. The entire process, starting from data generation to applying the algorithms, is independently repeated 150 times to measure the average mislabeling proportion and the corresponding standard errors. The results are presented in \prettyref{tab:letter2} (the standard errors are presented within the parentheses beside the average mislabeling values).

\subsubsection*{Results}
All the results show that our method consistently yields the lowest proportion of mislabeling, outperforming the other algorithms. Remarkably, our method produces a better mislabeling rate even in the absence of outliers. This indicates a heavy tail structure in the data set. Interestingly, in the two cluster setup, the mislabeling proportion reduces in the presence of data points from the letter class ``R". This is possible, as we did not aim to pick the outlier class that distorts the clustering process. We instead study the effect of a particular outlier class. This indicates a similarity of the data points from the outlier class with one of the clusters, resulting in more points being observed in the neighborhood of the corresponding cluster. When we observe more points in the clusters, separating the clusters becomes much more accessible, resulting in lower mislabeling. 

\begin{table}[H]
	\centering
	\caption{Results for clustering letters: $n = 100k, \delta = 0.48$, outlier proportion = 20\%, outlier class = R}
	\label{tab:letter2}
	\begin{tabular}{|c|c|c|c|c|c|c|}
		\hline
		Classes & Outliers & {\sf COD} + {\sf IOD} & $k$-median + {\sf IOD}& Lloyd + {\sf IOD} & Lloyd + $k$-means++ & Lloyd + random\\
		\hline
		\multirow{2}{*}{ W, V} & without & 0.276 (0.008) & 0.32 (0.005) &	0.391 (0.005) &	0.355 (0.004) &	0.402 (0.004)\\
		& with & 0.269 (0.008) & 0.317 (0.004) & 0.381 (0.004) & 0.352 (0.004) & 0.398 (0.004)
		\\
		\hline
		\multirow{2}{*}{X, M, A} & without & 0.194 (0.010) & 0.245 (0.007)	& 0.374 (0.004)	& 0.342 (0.006) &	0.357 (0.007)
		\\
		& with & 0.264 (0.009) & 0.275 (0.006) & 0.388 (0.003) &	0.354 (0.004) &	0.379 (0.005)
		\\
		\hline
	\end{tabular}
\end{table}

	\noindent \textbf{Acknowledgments:} Soham Jana thanks Debarghya Mukherjee and Sohom Bhattacharya for helpful discussions at the initial stage of the paper.

	\bibliographystyle{IEEEtran}
	
	\bibliography{tex/References}

% Generated by IEEEtran.bst, version: 1.14 (2015/08/26)
\begin{thebibliography}{10}
\providecommand{\url}[1]{#1}
\csname url@samestyle\endcsname
\providecommand{\newblock}{\relax}
\providecommand{\bibinfo}[2]{#2}
\providecommand{\BIBentrySTDinterwordspacing}{\spaceskip=0pt\relax}
\providecommand{\BIBentryALTinterwordstretchfactor}{4}
\providecommand{\BIBentryALTinterwordspacing}{\spaceskip=\fontdimen2\font plus
\BIBentryALTinterwordstretchfactor\fontdimen3\font minus
  \fontdimen4\font\relax}
\providecommand{\BIBforeignlanguage}[2]{{%
\expandafter\ifx\csname l@#1\endcsname\relax
\typeout{** WARNING: IEEEtran.bst: No hyphenation pattern has been}%
\typeout{** loaded for the language `#1'. Using the pattern for}%
\typeout{** the default language instead.}%
\else
\language=\csname l@#1\endcsname
\fi
#2}}
\providecommand{\BIBdecl}{\relax}
\BIBdecl

\bibitem{hastie2009elements}
T.~Hastie, R.~Tibshirani, J.~H. Friedman, and J.~H. Friedman, \emph{The
  elements of statistical learning: data mining, inference, and
  prediction}.\hskip 1em plus 0.5em minus 0.4em\relax Springer, 2009, vol.~2.

\bibitem{xu2005survey}
R.~Xu and D.~Wunsch, ``Survey of clustering algorithms,'' \emph{IEEE
  Transactions on neural networks}, vol.~16, no.~3, pp. 645--678, 2005.

\bibitem{ABBASI20072826}
\BIBentryALTinterwordspacing
A.~A. Abbasi and M.~Younis, ``A survey on clustering algorithms for wireless
  sensor networks,'' \emph{Computer Communications}, vol.~30, no.~14, pp.
  2826--2841, 2007, network Coverage and Routing Schemes for Wireless Sensor
  Networks. [Online]. Available:
  \url{https://www.sciencedirect.com/science/article/pii/S0140366407002162}
\BIBentrySTDinterwordspacing

\bibitem{sasikumar2012k}
P.~Sasikumar and S.~Khara, ``K-means clustering in wireless sensor networks,''
  in \emph{2012 Fourth international conference on computational intelligence
  and communication networks}.\hskip 1em plus 0.5em minus 0.4em\relax IEEE,
  2012, pp. 140--144.

\bibitem{maravelias1999habitat}
C.~D. Maravelias, ``Habitat selection and clustering of a pelagic fish: effects
  of topography and bathymetry on species dynamics,'' \emph{Canadian Journal of
  Fisheries and Aquatic Sciences}, vol.~56, no.~3, pp. 437--450, 1999.

\bibitem{pigolotti2007species}
S.~Pigolotti, C.~L{\'o}pez, and E.~Hern{\'a}ndez-Garc{\'\i}a, ``Species
  clustering in competitive lotka-volterra models,'' \emph{Physical review
  letters}, vol.~98, no.~25, p. 258101, 2007.

\bibitem{ng2006medical}
H.~Ng, S.~Ong, K.~Foong, P.-S. Goh, and W.~Nowinski, ``Medical image
  segmentation using k-means clustering and improved watershed algorithm,'' in
  \emph{2006 IEEE southwest symposium on image analysis and
  interpretation}.\hskip 1em plus 0.5em minus 0.4em\relax IEEE, 2006, pp.
  61--65.

\bibitem{ajala2012fuzzy}
A.~Ajala~Funmilola, O.~Oke, T.~Adedeji, O.~Alade, and E.~Adewusi, ``Fuzzy
  kc-means clustering algorithm for medical image segmentation,'' \emph{Journal
  of information Engineering and Applications, ISSN}, vol. 22245782, pp.
  2225--0506, 2012.

\bibitem{beatty2013spillover}
A.~Beatty, S.~Liao, and J.~J. Yu, ``The spillover effect of fraudulent
  financial reporting on peer firms' investments,'' \emph{Journal of Accounting
  and Economics}, vol.~55, no. 2-3, pp. 183--205, 2013.

\bibitem{fan2023unearthing}
J.~Fan, Q.~Liu, B.~Wang, and K.~Zheng, ``Unearthing financial statement fraud:
  Insights from news coverage analysis.'' \emph{Available at SSRN 4338277},
  2023.

\bibitem{lu2016statistical}
Y.~Lu and H.~H. Zhou, ``Statistical and computational guarantees of lloyd's
  algorithm and its variants,'' \emph{arXiv preprint arXiv:1612.02099}, 2016.

\bibitem{loffler2021optimality}
M.~L{\"o}ffler, A.~Y. Zhang, and H.~H. Zhou, ``Optimality of spectral
  clustering in the gaussian mixture model,'' \emph{The Annals of Statistics},
  vol.~49, no.~5, pp. 2506--2530, 2021.

\bibitem{abbe2022}
E.~Abbe, J.~Fan, and K.~Wang, ``An $\ell_p$ theory of pca and spectral
  clustering,'' \emph{The Annals of Statistics}, vol.~50, no.~4, pp.
  2359--2385, 2022.

\bibitem{ndaoud2022sharp}
M.~Ndaoud, ``Sharp optimal recovery in the two component gaussian mixture
  model,'' \emph{The Annals of Statistics}, vol.~50, no.~4, pp. 2096--2126,
  2022.

\bibitem{dreveton2024universal}
M.~Dreveton, A.~G{\"o}zeten, M.~Grossglauser, and P.~Thiran, ``Universal lower
  bounds and optimal rates: Achieving minimax clustering error in
  sub-exponential mixture models,'' \emph{arXiv preprint arXiv:2402.15432},
  2024.

\bibitem{chen2021cutoff}
X.~Chen and Y.~Yang, ``Cutoff for exact recovery of gaussian mixture models,''
  \emph{IEEE Transactions on Information Theory}, vol.~67, no.~6, pp.
  4223--4238, 2021.

\bibitem{kumar2004simple}
A.~Kumar, Y.~Sabharwal, and S.~Sen, ``A simple linear time (1+/spl
  epsiv/)-approximation algorithm for k-means clustering in any dimensions,''
  in \emph{45th Annual IEEE Symposium on Foundations of Computer
  Science}.\hskip 1em plus 0.5em minus 0.4em\relax IEEE, 2004, pp. 454--462.

\bibitem{chen2021optimal}
X.~Chen and A.~Y. Zhang, ``Optimal clustering in anisotropic gaussian mixture
  models,'' \emph{arXiv preprint arXiv:2101.05402}, 2021.

\bibitem{vassilvitskii2006k}
S.~Vassilvitskii and D.~Arthur, ``k-means++: The advantages of careful
  seeding,'' in \emph{Proceedings of the eighteenth annual ACM-SIAM symposium
  on Discrete algorithms}, 2006, pp. 1027--1035.

\bibitem{patel2023consistency}
D.~Patel, H.~Shen, S.~Bhamidi, Y.~Liu, and V.~Pipiras, ``Consistency of lloyd's
  algorithm under perturbations,'' \emph{arXiv preprint arXiv:2309.00578},
  2023.

\bibitem{deshpande2020robust}
A.~Deshpande, P.~Kacham, and R.~Pratap, ``Robust $ k $-means++,'' in
  \emph{Conference on Uncertainty in Artificial Intelligence}.\hskip 1em plus
  0.5em minus 0.4em\relax PMLR, 2020, pp. 799--808.

\bibitem{bhaskara2019greedy}
A.~Bhaskara, S.~Vadgama, and H.~Xu, ``Greedy sampling for approximate
  clustering in the presence of outliers,'' \emph{Advances in Neural
  Information Processing Systems}, vol.~32, 2019.

\bibitem{kannan2009spectral}
R.~Kannan, S.~Vempala \emph{et~al.}, ``Spectral algorithms,'' \emph{Foundations
  and Trends{\textregistered} in Theoretical Computer Science}, vol.~4, no.
  3--4, pp. 157--288, 2009.

\bibitem{bojchevski2017robust}
A.~Bojchevski, Y.~Matkovic, and S.~G{\"u}nnemann, ``Robust spectral clustering
  for noisy data: Modeling sparse corruptions improves latent embeddings,'' in
  \emph{Proceedings of the 23rd ACM SIGKDD international conference on
  knowledge discovery and data mining}, 2017, pp. 737--746.

\bibitem{zhang2018understanding}
Y.~Zhang and K.~Rohe, ``Understanding regularized spectral clustering via graph
  conductance,'' \emph{Advances in Neural Information Processing Systems},
  vol.~31, 2018.

\bibitem{srivastava2023robust}
P.~R. Srivastava, P.~Sarkar, and G.~A. Hanasusanto, ``A robust spectral
  clustering algorithm for sub-gaussian mixture models with outliers,''
  \emph{Operations Research}, vol.~71, no.~1, pp. 224--244, 2023.

\bibitem{JKY2023}
S.~Jana, K.~Yang, and S.~Kulkarni, ``Adversarially robust clustering with
  optimality guarantees,'' \emph{arXiv preprint arXiv:2306.09977}, 2023.

\bibitem{lloyd1982least}
S.~Lloyd, ``Least squares quantization in pcm,'' \emph{IEEE transactions on
  information theory}, vol.~28, no.~2, pp. 129--137, 1982.

\bibitem{bradley1996clustering}
P.~Bradley, O.~Mangasarian, and W.~Street, ``Clustering via concave
  minimization,'' \emph{Advances in neural information processing systems},
  vol.~9, 1996.

\bibitem{minsker2024geometric}
S.~Minsker and N.~Strawn, ``The geometric median and applications to robust
  mean estimation,'' \emph{SIAM Journal on Mathematics of Data Science},
  vol.~6, no.~2, pp. 504--533, 2024.

\bibitem{chen2018robust}
M.~Chen, C.~Gao, and Z.~Ren, ``Robust covariance and scatter matrix estimation
  under huber’s contamination model,'' \emph{The Annals of Statistics},
  vol.~46, no.~5, pp. 1932--1960, 2018.

\bibitem{diakonikolas2022clustering}
I.~Diakonikolas, D.~M. Kane, D.~Kongsgaard, J.~Li, and K.~Tian, ``Clustering
  mixture models in almost-linear time via list-decodable mean estimation,'' in
  \emph{Proceedings of the 54th Annual ACM SIGACT Symposium on Theory of
  Computing}, 2022, pp. 1262--1275.

\bibitem{depersin2022robust}
J.~Depersin and G.~Lecu{\'e}, ``Robust sub-gaussian estimation of a mean vector
  in nearly linear time,'' \emph{The Annals of Statistics}, vol.~50, no.~1, pp.
  511--536, 2022.

\bibitem{devroye2016sub}
L.~Devroye, M.~Lerasle, G.~Lugosi, and R.~I. Oliveira, ``Sub-gaussian mean
  estimators,'' 2016.

\bibitem{oliveira2023trimmed}
R.~I. Oliveira and L.~Resende, ``Trimmed sample means for robust uniform mean
  estimation and regression,'' \emph{arXiv preprint arXiv:2302.06710}, 2023.

\bibitem{lee2022optimal}
J.~C. Lee and P.~Valiant, ``Optimal sub-gaussian mean estimation in very high
  dimensions,'' in \emph{13th Innovations in Theoretical Computer Science
  Conference (ITCS 2022)}.\hskip 1em plus 0.5em minus 0.4em\relax
  Schloss-Dagstuhl-Leibniz Zentrum f{\"u}r Informatik, 2022.

\bibitem{lugosi2021robust}
G.~Lugosi and S.~Mendelson, ``Robust multivariate mean estimation: the
  optimality of trimmed mean,'' 2021.

\bibitem{lugosi2019near}
------, ``Near-optimal mean estimators with respect to general norms,''
  \emph{Probability theory and related fields}, vol. 175, no.~3, pp. 957--973,
  2019.

\bibitem{lugosi2019mean}
------, ``Mean estimation and regression under heavy-tailed distributions: A
  survey,'' \emph{Foundations of Computational Mathematics}, vol.~19, no.~5,
  pp. 1145--1190, 2019.

\bibitem{audibert2011robust}
J.-Y. Audibert and O.~Catoni, ``Robust linear least squares regression,'' 2011.

\bibitem{kaufman2009finding}
L.~Kaufman and P.~J. Rousseeuw, \emph{Finding groups in data: an introduction
  to cluster analysis}.\hskip 1em plus 0.5em minus 0.4em\relax John Wiley \&
  Sons, 2009.

\bibitem{rousseeuw1987clustering}
P.~Rousseeuw and P.~Kaufman, ``Clustering by means of medoids,'' in
  \emph{Proceedings of the statistical data analysis based on the L1 norm
  conference, neuchatel, switzerland}, vol.~31, 1987.

\bibitem{huber1965robust}
P.~J. Huber, ``A robust version of the probability ratio test,'' \emph{The
  Annals of Mathematical Statistics}, pp. 1753--1758, 1965.

\bibitem{huber1992robust}
------, ``Robust estimation of a location parameter,'' \emph{Breakthroughs in
  statistics: Methodology and distribution}, pp. 492--518, 1992.

\bibitem{diakonikolas2019robust}
I.~Diakonikolas, G.~Kamath, D.~Kane, J.~Li, A.~Moitra, and A.~Stewart, ``Robust
  estimators in high-dimensions without the computational intractability,''
  \emph{SIAM Journal on Computing}, vol.~48, no.~2, pp. 742--864, 2019.

\bibitem{liu2023robustly}
A.~Liu and A.~Moitra, ``Robustly learning general mixtures of gaussians,''
  \emph{Journal of the ACM}, 2023.

\bibitem{minsker2021minimax}
S.~Minsker, M.~Ndaoud, and Y.~Shen, ``Minimax supervised clustering in the
  anisotropic gaussian mixture model: a new take on robust interpolation,''
  \emph{arXiv preprint arXiv:2111.07041}, 2021.

\bibitem{diakonikolas2020robustly}
I.~Diakonikolas, S.~B. Hopkins, D.~Kane, and S.~Karmalkar, ``Robustly learning
  any clusterable mixture of gaussians,'' \emph{arXiv preprint
  arXiv:2005.06417}, 2020.

\bibitem{bakshi2022robustly}
A.~Bakshi, I.~Diakonikolas, H.~Jia, D.~M. Kane, P.~K. Kothari, and S.~S.
  Vempala, ``Robustly learning mixtures of k arbitrary gaussians,'' in
  \emph{Proceedings of the 54th Annual ACM SIGACT Symposium on Theory of
  Computing}, 2022, pp. 1234--1247.

\bibitem{bakshi2020outlier2}
A.~Bakshi and P.~Kothari, ``Outlier-robust clustering of non-spherical
  mixtures,'' \emph{arXiv preprint arXiv:2005.02970}, 2020.

\bibitem{omran2007overview}
M.~G. Omran, A.~P. Engelbrecht, and A.~Salman, ``An overview of clustering
  methods,'' \emph{Intelligent Data Analysis}, vol.~11, no.~6, pp. 583--605,
  2007.

\bibitem{hsu2012tail}
D.~Hsu, S.~Kakade, and T.~Zhang, ``A tail inequality for quadratic forms of
  subgaussian random vectors,'' \emph{Electronic communications in
  Probability}, vol.~17, pp. 1--6, 2012.

\bibitem{fan2008high}
J.~Fan and Y.~Fan, ``High dimensional classification using features annealed
  independence rules,'' \emph{Annals of statistics}, vol.~36, no.~6, p. 2605,
  2008.

\bibitem{misc_letter_recognition_59}
D.~Slate, ``{Letter Recognition},'' UCI Machine Learning Repository, 1991,
  {DOI}: https://doi.org/10.24432/C5ZP40.

\bibitem{yu1997assouad}
B.~Yu, ``Assouad, fano, and le cam,'' in \emph{Festschrift for Lucien Le Cam:
  research papers in probability and statistics}.\hskip 1em plus 0.5em minus
  0.4em\relax Springer, 1997, pp. 423--435.

\bibitem{roch2024modern}
S.~Roch, \emph{Modern discrete probability: An essential toolkit}.\hskip 1em
  plus 0.5em minus 0.4em\relax Cambridge University Press, 2024.

\bibitem{boucheron2013concentration}
S.~Boucheron, G.~Lugosi, and P.~Massart, \emph{Concentration inequalities: A
  nonasymptotic theory of independence}.\hskip 1em plus 0.5em minus 0.4em\relax
  Oxford university press, 2013.

\end{thebibliography}
	
	\newpage

\section{Mislabeling upper bound in \prettyref{thm:main-genr}}
\subsection{Preparation}
The proof of \prettyref{thm:main-genr} is primarily based on the following two stages: 
\begin{enumerate}[label=(\alph*)]
	\item analyzing accuracy of the clustering method based on current centroid estimates,
	\item analyzing the next centroid updates based on current labels.
\end{enumerate} 
We obtain results on these steps separately and then combine them to prove \prettyref{thm:main-genr}. Our analysis depends on high-probability events $\enorm_{\tau},\calE_{\gamma_0,\epsilon_0}$ given in the following lemmas.

\begin{lemma}\label{lmm:E-norm-genr}
	Suppose that $w_i$-s are independent random variables satisfying the $\decay_\sigma$-decay condition \ref{cond:G-decay} and $\beta\in (0,1)$ is fixed. Then given any $\tau>0$ there is an event $\enorm_{\tau}$ with probability at least $1-e^{-0.3n}$ on which the following holds. For any $S\subseteq [n]$ with $|S|\geq {n\beta}$, the cardinality of the set 
	$$\sth{i\in S: \norm{w_i}\leq \sigma \decay^{-1}\pth{\exp\sth{-\frac {1+1/\beta}{\tau}}}}$$ 
	is at least $(1-\tau)|S|$.
\end{lemma}

The following lemma provides a lower bound on $H_{s+1}$ based on $\Lambda_s$ and establishes an upper bound on $\Lambda_{s}$ in terms of $H_s$. 

\begin{lemma}\label{lmm:Hs-Ls-bound-genr}
	Fix any $\epsilon_0\in (0,\frac 12),\gamma_0\in (\frac {10}{n\alpha},\frac 12)$. Then whenever $\Delta,\sigma>0$ satisfies $\frac {5}{2\alpha\log(1/\decay(\epsilon_0\Delta/\sigma))}< \frac 12$ the following holds true. There is an event $\calE_{\gamma_0,\epsilon_0}$ with a probability of at least $1-2ke^{-{n\alpha/ 4}}-e^{-0.3n}$ on which for all $s\geq 0$, the \codelta algorithm with $\delta\in(\frac 12 - {\gamma_0\over 4},\frac 12)$ ensures:
	\begin{enumerate}[label=(\roman*)]
		\item if $\Lambda_s\leq \frac 12-\epsilon_0$ then $H_{s+1} \geq 1- \frac {5}{2 \alpha \log(1/\decay(\epsilon_0\Delta/\sigma))}$,
		\item if $H_{s+1}\geq \frac 12 + \gamma_0$ then $\Lambda_{s+1}\leq {8\sigma\over \Delta}\decay^{-1}\pth{\exp\sth{-\frac {1+2/\alpha}{\tau}}}$, where $\tau = {\gamma_0\over 1+2\gamma_0}$.
	\end{enumerate}
\end{lemma}

\begin{proof}[Proof of \prettyref{lmm:Hs-Ls-bound-genr}]
	
	We first prove part (i). For any $g\neq h\in[k]\times [k]$,
	\begin{align}\label{eq:ht4}
		\indc{z_i=g,\hat z_i^{(s+1)}=h}
		&\leq \indc{\|Y_i-\hat\theta_h^{(s)}\|^2
			\leq \|Y_i-\hat\theta_g^{(s)}\|^2,\ i\in T^*_g}
		\nonumber\\
		&=\indc{\|\theta_g+w_i-\hat\theta_h^{(s)}\|^2
			\leq \|\theta_g+w_i-\hat\theta_g^{(s)}\|^2}
		=\indc{\|\theta_g-\hat\theta_h^{(s)}\|^2-\|\theta_g-\hat\theta_g^{(s)}\|^2\leq 2\langle w_i,\hat\theta_h^{(s)}-\hat\theta_g^{(s)}\rangle}.
	\end{align}
	The triangle inequality and the fact $\|\theta_h-\hat\theta_h^{(s)}\|\leq \Lambda_s\Delta
	\leq \Lambda_s\|\theta_g-\theta_h\|$ implies
	\begin{align*}
		\|\theta_g-\hat\theta_h^{(s)}\|^2\geq \pth{\|\theta_g-\theta_h\|-\|\theta_h-\hat\theta_h^{(s)}\|}^2
		\geq (1-\Lambda_s)^2\|\theta_g-\theta_h\|^2.
	\end{align*}
	This implies
	\begin{align}
		\label{eq:km1}
		&\|\theta_g-\hat\theta_h^{(s)}\|^2-\|\theta_g-\hat\theta_g^{(s)}\|^2
		\nonumber\\
		&\geq (1-\Lambda_s)^2\|\theta_g-\theta_h\|^2
		-\Lambda_s^2\|\theta_g-\theta_h\|^2
		= (1-2\Lambda_s)\|\theta_g-\theta_h\|^2.
	\end{align}
	In view of the last inequality, using the fact
	\begin{align*}
		|\langle w_i,\hat\theta^{(s)}_h-\hat\theta^{(s)}_g\rangle|
		&\leq \|w_i\|\cdot \|\hat\theta^{(s)}_h-\hat\theta^{(s)}_g\|
		\nonumber\\
		&\leq \|w_i\|\cdot(\|\hat\theta^{(s)}_g-\theta_g\|+\|\hat\theta^{(s)}_h-\theta_h\|+\|\theta_g-\theta_h\|)
		\leq \|w_i\|(2\Lambda_s+1)\|\theta_g-\theta_h\|,
	\end{align*}
	we continue \eqref{eq:ht4} to get
	\begin{align}\label{eq:bound1}
		\indc{z_i=g,\hat z_i^{(s+1)}=h}
		&\leq \indc{{1-2\Lambda_s\over 2(1+2\Lambda_s)}\|\theta_g-\theta_h\|\leq \|w_i\|}.
	\end{align}
	Simplifying above with $\Lambda_s\leq \frac 12 - \epsilon_0$ 
	\begin{align}
		\indc{z_i=g,\hat z_i^{(s+1)}=h} \leq \indc{{\epsilon_0}\|\theta_g-\theta_h\|\leq \|w_i\|}
	\end{align}
	Summing $\indc{z_i=g,\hat z_i^{(s+1)}=h}$ over $\sth{i\in T_g^*}$, in view of \prettyref{lmm:E-conc-genr}, we get on the event $\econ_{\epsilon_0}$
	\begin{align}
		\label{eq:km2-genr}
		n_{gh}^{(s+1)}&\leq \sum_{i\in T_g^*}\indc{{\epsilon_0}\Delta\leq \|w_i\|}
		\leq \frac {5n_g^*}{4\log(1/\decay(\epsilon_0\Delta/\sigma))},\ \forall h\in [k],h\neq g.
	\end{align}
	Using the last display and noting that $k\leq \frac 1{\alpha}$ and $n_g^*\geq n\alpha$ we get
	\begin{align}\label{eq:km20-genr}
		{\sum_{h\in[k]\atop h\neq g}n_{gh}^{(s+1)}\over n_g^*}
		&\leq 
		\frac {5k}{4\log(1/\decay(\epsilon_0\Delta/\sigma))}
		\leq
		\frac {5}{4\alpha \log(1/\decay(\epsilon_0\Delta/\sigma))}.
	\end{align}
	Next, we switch $g,h$ in \eqref{eq:km2-genr} and sum over $h\in [k], h\neq g$. We get on the event $\econ_{\epsilon_0}$
	\begin{align*}
		\sum_{h\in[k]\atop h\neq g}n_{hg}^{(s+1)}
		&\leq \frac {5\sum_{h\in [k],h\neq g}n_h^*}{4\log(1/\decay(\epsilon_0\Delta/\sigma))} 
		\leq \frac {5n}{4\log(1/\decay(\epsilon_0\Delta/\sigma))}. 
	\end{align*}
	Using the relation between $\epsilon_0,\Delta,\sigma,\alpha$ in the lemma statement
	we get 
	$${\sum_{h\in[k]\atop h\neq g}n_{gh}^{(s+1)}\over n_g^*}\leq \frac 12,$$ 
	which implies
	$$n_g^{(s+1)}\geq n_{gg}^{(s+1)}=
	n_g^*-\sum_{h\in[k]\atop h\neq g}n_{gh}^{(s+1)}
	\geq \frac 12 n_g^*\geq \frac 12 n\alpha.$$ 
	This gives us
	\begin{align*}
		{\sum_{h\neq g}n_{hg}^{(s+1)}\over n_g^{(s+1)}}
		\leq \frac {5}{2 \alpha\log(1/\decay(\epsilon_0\Delta/\sigma))}.
	\end{align*} 
	Using the above with \eqref{eq:km20-genr} we get with probability $1-2kn^{-c/4}$
	\begin{align*}
		H_{s+1}=1-\max\sth{{\sum_{h\neq g}n_{gh}^{(s+1)}\over n_g^{*}},{\sum_{h\neq g}n_{hg}^{(s+1)}\over n_g^{(s+1)}}}
		\geq 
		1-\frac {5}{2 \alpha\log(1/\decay(\epsilon_0\Delta/\sigma))}.
	\end{align*}	
	
	Next we present below the proof of \prettyref{lmm:Hs-Ls-bound-genr}(ii). We use  \prettyref{prop:Lambda-bound} provided below.
	\begin{proposition}\label{prop:Lambda-bound}
		Suppose that given any $\tau\in (0,1)$, there is an event $\calF_{\tau}$ and a number $D_{\tau}>0$ such that, on $\calF_{\tau}$, for any $S\subseteq [n]$ with $|S|>{n\alpha\over 2}$, the cardinality of the set 
		$\sth{i\in S: \norm{w_i}\leq D_{\tau}}$ is at least $(1-\tau)|S|$. Then for any $\gamma\in (10/{n\alpha},1/2)$, using $\tau={\gamma\over 1+2\gamma}$, we get that on the event $\calF_\tau$, if $H_{s}\geq \frac 12+\gamma$, then the \codelta algorithm with any $\delta\in(\frac 12 - {\gamma\over 4},\frac 12)$ returns $\Lambda_{s}\leq {8D_{\tau}/\Delta}$.
	\end{proposition}
	
	A proof of the above result is provided at the end of this section. We choose $$\gamma=\gamma_0,\quad \tau={\gamma_0\over 1+2\gamma_0},
	\quad D_{\tau} = \sigma \decay^{-1}\pth{\exp\sth{-\frac {1+2/\alpha}{\tau}}},
	\quad \calF_{\tau}=\enorm_{\tau}.$$
	In view of \prettyref{lmm:E-norm-genr} note that the event $\calF_\tau$ has probability at least $1-e^{-0.3n}$ and satisfies the requirement in \prettyref{prop:Lambda-bound}. This implies that we get the required bound on $\Lambda_{s+1}$. 
	
	Combining the proof of part (i) we conclude that both the claims hold with probability at least $1-8ke^{-{n\alpha\over 4}}$. 
\end{proof}

\begin{proof}[Proof of \prettyref{prop:Lambda-bound}]
	We prove the above result using a contradiction. Fix $\tau={\gamma\over 1+2\gamma}$ as specified. Our entire analysis will be on the event $\calF_\tau$. Let us assume $\Lambda_{s} > {8D_{\tau}/\Delta}$. This implies that there exists a cluster $h$ such that the centroid estimation error satisfies 
	$$\|\hat \theta_h^{(s)}-\theta_h\|> {8D_{\tau}}.$$ As $H_{s}\geq \frac 12+\gamma$, we know that $n_{hh}^{(s)}\geq \pth{\frac 12+\gamma} n_h^*$. As we are on the set $\calF_{\tau}$ and
	$$
	S=T_{h}^{(s)}\cap T_{h}^*, \quad |S|=n_{hh}^{(s)} \geq {\pth{\frac 12+\gamma}n_h^*}\geq {\pth{\frac 12+\gamma}n \alpha}
	$$ 
	we get $\abs{\sth{j\in S: \norm{w_j}\leq D_{\tau}}}
	\geq (1-\tau) |S|$. In view of $n_{hh}^{(s)}\geq \pth{\frac 12+\gamma} n_h^{(s)}$ from $H_{s}\geq \frac 12+\gamma$ the above implies
	$$
	\abs{\sth{j\in S: \norm{w_j}\leq D_{\tau}}}
	\geq {1+\gamma\over 1+2\gamma}|S|
	\geq \pth{{1+\gamma\over 1+2\gamma}}\pth{\frac 12 + \gamma}n_h^{(s)}\geq \pth{\frac 12 + \frac \gamma2} n_h^{(s)}.
	$$
	This gives us 
	\begin{align}
		\abs{\sth{j\in T_h^{(s)}: \|{Y_j-\theta_h}\|\leq D_{\tau}}}
		\geq \pth{\frac 12 + \frac \gamma2} n_h^{(s)}
		\label{eq:temp0}
	\end{align}
	Next we will show 
	\begin{align}
		\abs{\sth{j\in T_h^{(s)}:\|Y_j-\hat\theta_h^{(s)}\|\leq 4D_\tau}}\geq (1-\delta)n_h^{(s)} \geq \pth{\frac 12 +\frac \gamma4}n_h^{(s)}.
		\label{eq:temp2}
	\end{align} 
	To prove the above, we first set $
	W = \sth{j\in T_h^{(s)}: \|Y_j-\theta_h\|\leq D_{\tau}}. 
	$
	Then given any $j_0\in W$, all the points in $\sth{Y_j:j\in W}$ are within $2D_{\tau}$ distance of $Y_{j_0}$. This implies
	\begin{align}
		\abs{\sth{j\in T_h^{(s)}:\|Y_j-Y_{j_0}\|\leq 2D_\tau}}\geq \pth{\frac 12 +\frac \gamma2}n_h^{(s)}.
		\label{eq:temp1}
	\end{align}
	Now, remember the computation of ${\sf TM}_{\delta}(\{Y_j:j\in T_h^{(s)}\})$ in \prettyref{algo:cod} according to \prettyref{algo:cent}. In view of \eqref{eq:temp1} we then have $R_{i^*}\leq 2D_\tau$. Hence, for $\delta=\frac 12-\frac \gamma4$
	$$
	\abs{\sth{j\in T_h^{(s)}:\|Y_j-Y_{i^*}\|\leq 2D_\tau}}
	\geq \abs{\sth{j\in T_h^{(s)}:\|Y_j-Y_{i^*}\|\leq R_{i^*}}}
	\geq \pth{\frac 12 + \frac \gamma4} n_{h}^{(s)}.
	$$ 
	Then, the steps in \prettyref{algo:cent} imply for some $V\subset T_h^{(s)}$ with $|V| = \pth{\frac 12+\frac \gamma4}n_h^{(s)}$
	\begin{align*}
		\|\hat\theta_h^{(s)}-Y_j\|
		\leq \|\hat\theta_h^{(s)}-Y_{i^*}\|+\|Y_j-Y_{i^*}\|
		\leq \frac {\sum_{j\in V}\| Y_j -Y_{i^*}\|}{\floor{(1-\delta)n_h^{(s)}}+1}
		+R_{i^*}
		\leq 2R_{i^*}
		\leq 4D_{\tau},\quad j\in V.
	\end{align*}
	This completes the proof of \eqref{eq:temp2}. 
	
	Finally, combining \eqref{eq:temp0} and \eqref{eq:temp2}
	we get a contradiction as
	\begin{itemize}
		\item 
		$
		\abs{\sth{j\in T_h^{(s)}:\|Y_j-\hat\theta_h^{(s)}\|\leq 4D_\tau} 
		\cup \sth{j\in T_h^{(s)}: \|{Y_j-\theta_h}\|\leq D_{\tau}}}
		\leq \abs{\{j\in T_h^{(s)}\}} = n_h^{(s)}
		$
		\item 
		$\abs{\sth{j\in T_h^{(s)}:\|Y_j-\hat\theta_h^{(s)}\|\leq 4D_\tau} 
		\cap \sth{j\in T_h^{(s)}: \|{Y_j-\theta_h}\|\leq D_{\tau}}}=0$
		
		\item $\abs{\sth{j\in T_h^{(s)}:\|Y_j-\hat\theta_h^{(s)}\|\leq 4D_\tau}}+  \abs{\sth{j\in T_h^{(s)}: \|{Y_j-\theta_h}\|\leq D_{\tau}}}\geq \pth{1 +\frac {3\gamma}4}n_h^{(s)}.$
	\end{itemize}
\end{proof}

\subsection{Proof of \prettyref{thm:main-genr}}

In view of the above results we provide the proof of \prettyref{thm:main-genr} below. Note that it suffices to bound the larger quantity $\EE\qth{\frac 1n\sum_{i=1}^n\indc{z_i\neq \hat z_i}}$ to obtain a bound on $\EE\qth{\ell(\hat z,z)}$, as the later contains a minimum over all possible label permutations. For an ease of notations, denote
\begin{align}
	A_s=\frac 1n \sum_{i=1}^n\indc{\hat z^{(s)}_i\neq z_i}
	=\frac 1n \sum_{h\neq g \in [k]} n_{hg}^{(s)}.
\end{align} 
For $c_1>0$ to be chosen later, we define
$$\epsilon_0={\decay^{-1}(e^{-c_1/ \alpha})\over {\Delta/\sigma}},\quad \gamma_0=\gamma.$$
Then from \prettyref{lmm:Hs-Ls-bound-genr} it follows that we can choose $c_1,c_2>0$ such that if $$\Delta\geq c_2\sigma\decay^{-1}\pth{\exp\sth{-\frac {1+2/\alpha}{\tau}}}, \tau={\gamma_0\over 1+2\gamma_0},$$ then on the set $\calE_{\epsilon_0,\gamma_0}$, for a large enough $c_1$, we have 
\begin{itemize}
	\item if $\Lambda_0\leq \frac 12-\epsilon_0$ then $H_1\geq 0.8$,
	\item if $H_0\geq \frac 12+\gamma_0$ then $\Lambda_0\leq 0.2$
\end{itemize}
A second application of \prettyref{lmm:Hs-Ls-bound-genr}, with $\epsilon_1=0.3,\gamma_1=0.3$, guarantees that if 
${\Delta}\geq \decay^{-1}\pth{\exp\sth{-\frac {c_3}\alpha}}$ for a large enough $c_3$,
then on the set $\calE_{\epsilon_1,\gamma_1}$ we have for all $s\geq 1$,
\begin{enumerate}[label=(P\arabic*)]
	\item \label{eq:P1-genr} If $\Lambda_s\leq \frac 12 - \epsilon_1$ then $H_{s+1}\geq 1- \frac {5}{2 \alpha \log(1/\decay(\epsilon_0\Delta/\sigma))}\geq 0.8$,
	\item \label{eq:P2-genr} If $H_s\geq \frac 12 + \gamma_1$ then $\Lambda_s\leq {8\sigma\over \Delta}\decay^{-1}\pth{\exp\sth{-\frac {(1+2/\alpha)(1+2\gamma_1)}{\gamma_1}}}\leq { \decay^{-1}(e^{-c_4/\alpha})\over \snr}\leq 0.2$,
\end{enumerate}
where $c_4$ is an absolute constant. Note that from \prettyref{lmm:Hs-Ls-bound-genr} the probabilities of each of the sets $\calE_{\epsilon_0,\gamma_0},\calE_{\epsilon_1,\gamma_1}$ are at least $1-2k^{n\alpha\over 4}-e^{-0.3 n}$, and hence
\begin{align}
	\label{eq:Prob-E-new-genr}
	\calE&=\calE_{\epsilon_0,\gamma_0}\cap \calE_{\epsilon_1,\gamma_1}\text{ with }
	\PP\qth{\calE}
	\geq 1-8ke^{-{n\alpha\over 4}}.
\end{align} 
In view of the above arguments, on the event $\calE$ we have
\begin{align}
	\label{eq:km14-genr}
	\Lambda_s\leq{ \decay^{-1}(e^{-c_4/\alpha})\over \snr}\leq 0.2,\ H_{s}\geq 0.8 \text{ for all } s\ge 1.
\end{align}
Next, we will show that $\PP\qth{\left.z_i\neq \hat z_i^{(s+1)}\right|\calE}$ is small for each $i\in [n]$, and then sum over $i$ to achieve the required result. Fix a choice for $z_i$, say equal to $g\in [k]$. Remember \eqref{eq:bound1}
\begin{align}
	\indc{z_i=g,\hat z_i^{(s+1)}=h}
	&\leq \indc{{1-2\Lambda_s\over 2(1+2\Lambda_s)}\|\theta_g-\theta_h\|\leq \|w_i\|}.
\end{align}
Then in view of the inequalities
\begin{itemize}
	\item $(1+x)^{-1}\geq 1-x$ with $x=2\Lambda_s< 1$
	\item $(1-x)^2\geq 1-2x$ with the above choices of $x$,
\end{itemize}
and $\|\theta_g-\theta_h\|\geq \Delta$ we continue the last display to get
\begin{align*}
	\indc{z_i=g,\hat z_i^{(s+1)}=h}
	\leq \indc{\frac 12(1-4\Lambda_s)\Delta\leq \|w_i\|}
	\leq \indc{\sigma (\snr-4\decay^{-1}(e^{-c_4/\alpha}))\leq \|w_i\|},
\end{align*}
where the last inequality followed using the bound on  $\Lambda_s\leq{ \decay^{-1}(e^{-c_4/\alpha})\over \snr}$ in \eqref{eq:km14-genr}.
Taking expectation conditioned on the event $\calE$ in \eqref{eq:Prob-E-new-genr} and using the inequality
$$
\indc{z_i\neq\hat z_i^{(s+1)}}
\leq\sum_{g,h\in [k]\atop g\neq h} \indc{\hat z_i^{(s+1)}=h,z_i=g}
$$
we get
\begin{align*}
	\PP\qth{\left.z_i\neq \hat z_i^{(s+1)}\right|\calE}
	&\leq
	k^2 \max_{g,h\in [k]\atop g\neq h}\PP\qth{\left.z_i=g, \hat z_i^{(s+1)}=h \right|\calE}
	\leq k^2\decay\pth{\snr-4\decay^{-1}(e^{-c_4/\alpha})}
\end{align*}
This implies
\begin{align*}
	\EE\qth{A_{s+1}|\calE}
	=\frac 1n\sum_{i=1}^n\PP\qth{\left.z_i\neq \hat z_i^{(s+1)}\right| \calE}\leq k^2\decay\pth{\snr- 4\decay^{-1}(e^{-{c_4/\alpha}})}.
\end{align*}
Combining the above with \eqref{eq:Prob-E-new-genr} we get
\begin{align*}
	\EE\qth{A_{s+1}}&\leq 8ke^{-{n\alpha\over 4}}+k^2\decay\pth{\snr- 4\decay^{-1}(e^{-{c_4/\alpha}})}.
\end{align*}

\section{Proof of results with outlier (\prettyref{thm:out-genr})}

\subsection{Preparation}

The following lemma provides results that show a lower bound on $H_{s+1}$ based on $\Lambda_s$ and establish upper bound on $\Lambda_{s}$ in terms of $H_s$, when $n\alpha(1-\psi)$ outliers are present. 

\begin{lemma}\label{lmm:Hs-Ls-bound-genr-out}
	Fix any $\epsilon_0\in (0,\frac 12),\gamma_0\in (\frac {10}{n\alpha},\frac 12)$. Then, whenever $\Delta,\sigma>0$ satisfies $\frac {5}{2\alpha\log(1/\decay(\epsilon_0\Delta/\sigma))}< \frac 12$ the following holds true. There is an event $\calE_{\gamma_0,\epsilon_0}$, which has a probability at least $1-4ke^{-{n\alpha\over 4}}$, on which we have for all $s\geq 0$, the \codelta algorithm with $\delta\in(\frac 12 - {\gamma_0\over 4},\frac 12)$ ensures:
	\begin{enumerate}[label=(\roman*)]
		\item if $\Lambda_s\leq \frac 12-\epsilon_0$ then $H_{s+1} \geq \frac 12 +{{\psi-2\xi\over 2(2-\psi)}}$ where $\xi={5\over 4\alpha \log(1/\decay(\epsilon_0\Delta/\sigma))}$,
		\item if $H_{s+1}\geq \frac 12 + \gamma_0$ then $\Lambda_{s+1}\leq {8\sigma\over \Delta}\decay^{-1}\pth{\exp\sth{-\frac {1+2/\alpha}{\tau}}}$, where $\tau = {\gamma_0\over 1+2\gamma_0}$.
	\end{enumerate}
\end{lemma}

\begin{proof}[Proof of \prettyref{lmm:Hs-Ls-bound-genr-out}]
	
	Repeating the argument in \eqref{eq:bound1} in the proof of \prettyref{lmm:Hs-Ls-bound-genr} we have
	\begin{align*}
		\indc{z_i=g,\hat z_i^{(s+1)}=h}
		&\leq \indc{{\epsilon_0\Delta/\sigma}\leq \|w_i\|}.
	\end{align*}
	Summing $\indc{z_i=g,\hat z_i^{(s+1)}=h}$ over $\sth{i\in T_g^*}$, in view of \prettyref{lmm:E-conc-genr}, we get on the event $\econ_{\epsilon_0}$
	\begin{align}
		\label{eq:km2-genr-out}
		n_{gh}^{(s+1)}&\leq \sum_{i\in T_g^*}\indc{\epsilon_0\Delta\leq \|w_i\|}
		\leq \frac {5n_g^*}{4\log(1/\decay(\epsilon_0\Delta/\sigma))},\ \forall h\in [k],h\neq g.
	\end{align}
	Using the last display and noting that $k\leq \frac 1{\alpha}$ and $n_g^*\geq n\alpha$ we get
	\begin{align}\label{eq:km20-genr-out}
		{\sum_{h\in[k]\atop h\neq g}n_{gh}^{(s+1)}\over n_g^*}
		&\leq 
		\frac {5k}{4\log(1/\decay(\epsilon_0\Delta/\sigma))}
		\leq
		\frac {5}{4\alpha \log(1/\decay(\epsilon_0\Delta/\sigma))}.
	\end{align}
	Next we switch $g,h$ in \eqref{eq:km2-genr-out} and sum over $h\in [k], h\neq g$. We get with probability similar to $\PP\qth{\econ_{\epsilon_0}}$
	\begin{align*}
		\sum_{h\in[k]\atop h\neq g}n_{hg}^{(s+1)}
		&\leq \frac {5\sum_{h\in [k],h\neq g}n_h^*}{4\log(1/\decay(\epsilon_0\Delta/\sigma))} 
		\leq \frac {5n}{4\log(1/\decay(\epsilon_0\Delta/\sigma))}. 
	\end{align*}
	We define 
	$\xi={5\over 4\alpha \log(1/\decay(\epsilon_0\Delta/\sigma))}.$
	In view of \eqref{eq:km20-genr-out} this implies 
	\begin{align}\label{eq:bound2}
		n_{gg}^{(s+1)}=n_g^*-\sum_{h\in[k]\atop h\neq g}n_{gh}^{(s+1)}\geq n_g^*-n\alpha\xi\geq n_g^*(1-\xi).
	\end{align} 
	Using the above and noticing that in addition to the points in $\cup_{h\in[k]}\sth{T_{h}^*\cap T_g^{(s+1)}}$, $T_g^{(s+1)}$ can at most have $n\alpha(1-\psi)$ many extra points, accounting for the outliers, we get
	\begin{align*}
		{n_{gg}^{(s+1)}\over n_g^{(s+1)}}\geq {n_{gg}^{(s+1)}\over n_{gg}^{(s+1)}+n\alpha\xi+n\alpha(1-\psi)}
		\geq \frac 1{1+{n_g^*(1-\psi+\xi)\over n_{gg}^{(s+1)}}}
		\geq \frac 1{1+{1-\psi+\xi\over 1-\xi}}
		=\frac 12 + {\psi-2\xi\over 2(2-\psi)}.
	\end{align*}
	Combining the last display with \eqref{eq:bound2} we get
	\begin{align*}
		H_{s+1}=\min_{g\in [k]}\sth{\min\sth{{n_{gg}^{(s)}\over n_g^*},{n_{gg}^{(s)}\over n_{g}^{(s)}}}}
		\geq \frac 12+\min\sth{{\psi-2\xi\over 2(2-\psi)},\frac12-\xi}.
	\end{align*}
	As $\psi<1$, we get ${\psi-2\xi\over 2(2-\psi)}\leq \frac \psi2-\xi\leq \frac 12-\xi$. This finishes the proof.
	
	The proof of \prettyref{lmm:Hs-Ls-bound-genr-out}(ii) is similar to the proof of \prettyref{lmm:Hs-Ls-bound-genr}(ii).
\end{proof}

\subsection{Proof of \prettyref{thm:out-genr}}

For $c_1>0$ to be chosen later, we define
$$\epsilon_0={\decay^{-1}\pth{e^{-{c_1\over \alpha\psi}}}\over {\Delta/\sigma}},\quad \gamma_0=\gamma.$$
Then from \prettyref{lmm:Hs-Ls-bound-genr-out} it follows that we can choose $c_2>0$ such that if 
\begin{align*}
	\Delta\geq c_2\sigma\max\sth{\decay^{-1}\pth{\exp\sth{-\frac {1+2/\alpha}{\tau}}},\decay^{-1}\pth{e^{-{c_1\over \alpha\psi}}}}, \quad \tau={\gamma_0\over 1+2\gamma_0},
\end{align*}
then on the set $\calE_{\epsilon_0,\gamma_0}$, as $\delta= \frac 12 - \min\sth{\frac {\gamma_0}4,{\delta/24}}$, we have that the \codelta algorithm guarantees
\begin{itemize}
	\item if $\Lambda_0\leq \frac 12-\epsilon_0$ then $H_1\geq \frac 12 + {\psi\over 6}$.
	\item if $H_0\geq \frac 12+\gamma_0$ then $\Lambda_0\leq 0  .3$,
\end{itemize}
A second application of \prettyref{lmm:Hs-Ls-bound-genr-out}, with $\epsilon_1=0.2,\gamma_1={\psi\over 6}$ and the above lower bound on $\Delta$ for large enough $c_1,c_2$, implies that the \codelta algorithm guarantees on the event $\calE_{\epsilon_1,\gamma_1}$ for all $s\geq 1$,
\begin{enumerate}[label=(P\arabic*)]
	\item \label{eq:P1-genr-out} if $\Lambda_s\leq \frac 12 - \epsilon_1$ then $H_{s+1}\geq \frac 12+\frac \psi6 $,
	\item \label{eq:P2-genr-out} if $H_s\geq \frac 12 + \gamma_1$ then $\Lambda_s\leq {8\sigma\over \Delta}\decay^{-1}\pth{\exp\sth{-\frac {(1+2/\alpha)(1+2\gamma_1)}{\gamma_1}}}\leq { \decay^{-1}(e^{-c_4/(\alpha\psi)})\over \snr}\leq 0.2$,
\end{enumerate}
where $c_4$ is an absolute constant. Combining the above displays we get that on the event
\begin{align}
	\calE&=\calE_{\epsilon_0,\gamma_0}\cap \calE_{\epsilon_1,\gamma_1}\text{ with }
	\PP\qth{\calE}
	\geq 1-8ke^{-{n\alpha\over 4}}
\end{align} 
we have
\begin{align}
	\Lambda_s\leq { \decay^{-1}(e^{-c_4/(\alpha\psi)})\over \snr}\leq 0.2,\ H_{s}\geq \frac 12+\frac \psi6 \quad \text{for all } s\ge 2.
\end{align}
Now, it is sufficient to show that $\PP\qth{\left.z_i\neq \hat z_i^{(s+1)}\right|\calE}$ is small for each $i\in [n]$. This will imply that on the event $\calE$, $\ell(\hat z^{(s+1)},z)=\frac 1n\sum_{i=1}^n\indc{z_i\neq \hat z_i^{(s+1)}}$ is also small in probability. In view of the above, using a Markov inequality we will conclude the result. From this point onward the proof is again similar to the proof in \prettyref{thm:main-genr} for showing that $\PP\qth{\left.z_i\neq \hat z_i^{(s+1)}\right|\calE}$ is small. The only difference is that we replace the term $\decay^{-1}(e^{-c_4/\alpha})$ by $\decay^{-1}(e^{-c_4/(\alpha\psi)})$. This finishes the proof.

\section{Proof of mislabeling lower bound (\prettyref{thm:lb-general})}

\label{app:lower_bound}

We will consider a smaller set of labels to perform the analysis. For a simplicity of notations, let $m=k\ceil{n\alpha}\leq k(n\alpha+1)$ and $n-m$ is divisible by 3. Define
\begin{align}
	\begin{gathered}
		\calZ^*=\bar z \times \{1,2\}^{n-m\over 3}\subseteq [k]^n,\\
		m= k\ceil{n\alpha}, \quad \bar z\in [k]^{m+{2(n-m)\over 3}},\quad 
		\bar z_i = u , u\in \{1,\dots,k\}, \ i\in (u-1){m\over k}+1,\dots,u{m\over k},\\
		\bar z_{m+1}=\dots=\bar z_{m+{n-m\over 3}} = 1,\quad
		\bar z_{m+\ceil{n-m\over 3}+1}=\dots=\bar z_{m+{2(n-m)\over 3}} = 2.
	\end{gathered}
\end{align}
In other words, for each $z\in \calZ^*$, we already know the labels corresponding to the first $m+2{n-m\over 3}$ entries. For the rest of the entries the labels can either be 1 or 2, and we will put a prior on the distribution of those labels. Note that for each label vector $z\in \calZ^*$ we have $|\{i:z_i=g\}|\geq \ceil{n\alpha}$ for each $g=1,\dots,k$. Under the above choice of label set, it suffices to show 
$$\inf_{\hat z\in \calZ^*}\sup_{(z,\{\theta_i\}_{i=1}^k)\in \calP_0:z\in \calZ^* }\EE\qth{\ell(\hat z,z)}
\geq {1-k\alpha-k/n \over 12}\cdot \decay(\snr + C_{\decay}),\quad \Delta\geq \sigma C_{\decay},$$
for a suitable choice of $C_{\decay}$. 

To show the above, we first note that for any $z,z'\in \calZ^*, z\neq z'$, we have
$$
\sum_{i=1}^n\indc{z_i\neq z_i'}
=\sum_{i=m+{2(n-m)\over 3}+1}^n\indc{z_i\neq z_i'}
\leq {n-m\over 3},
$$ 
and for any $\pi\in \calS_k$ that is not an identity map 
\begin{align*}
\sum_{i=1}^n\indc{\pi(z_i)\neq z_i'}
\geq
\begin{cases}
	2\ceil{n\alpha}+\sum_{i=m+{2(n-m)\over 3}+1}^n\indc{z_i\neq z_i'}, &\text{ if $\pi(1)=1,\pi(2)=2$}\\
	{n-m\over 3},
	&\text{ if $\pi(1)\neq 1$ or $\pi(2)\neq 2$}.
\end{cases}
\end{align*}
This implies the loss $\EE\qth{\ell(\hat z,z)}=\EE\qth{\inf_{\pi\in \calS_k}\qth{\frac 1n\sum_{i=1}^n\indc{\pi(z_i)\neq \hat z_i}}}$ reduces to $\EE\qth{\sum_{i=1}^n\frac 1n\indc{z_i\neq z_i'}}$.
In view of the above deductions, with the parameter (label) set $\calZ^*$ we will apply Assouad's Lemma \cite{yu1997assouad} to bound 
$$
\inf_{\hat z\in \calZ^*}\sup_{\pth{z,\{\theta_i\}_{i=1}^k}\in \calP_0:z\in \calZ^* }
\frac 1n\EE\qth{\sum_{i=1}^n\indc{z_i\neq z_i'}}.
$$
\begin{lemma}[Assouad]\label{lmm:assouad}
	Let $r\geq 1$ be an integer and  $\calF_{r}=\{\calQ_z: z\in \calZ\}$ contains $2^r$ probability measures. Let $\hat f$ be an estimator of $f(\calQ_z)$ taking values in a metric space $(\calD,d)$, for some $\calQ_z\in \calF_r$. Write $v\sim v'$ if $v$ and $v'$ differ in only one coordinate, and write $v\sim_j v'$ when the coordinate is the $j$-th. Suppose that there are $r$ pseudo-distances $d_1,\dots,d_r$ on $\calD$ such that for any $x,y\in \calD$
	\begin{align*}
		d(x,y)=\sum_{j=1}^r d_j(x,y),
	\end{align*}
	and further assume that, if $v\sim_jv'$ then 
		$d_j(f(\calQ_z),f(\calQ_z'))\geq \delta$.
	Then 
	\begin{align*}
		\max_{\calQ_z\in \calF_r} \EE_z\qth{d(\hat f,f(\calQ_z))}
		\geq r\cdot {\delta\over 2}\cdot \min\{1-\TV(\calQ_z,\calQ_{z'}):z\sim z'\}.
	\end{align*}
\end{lemma}

To apply the above lemma, define the data distribution $\calQ_z$ given any label vector $z\in \calZ^*$
\begin{align*}
	\begin{gathered}
		z\in \calZ^*,\quad \calQ_{z_i}=\text{Distribution of $\theta_{z_i}+w_i$, } i=1,\dots,n.\\
	\overline \calQ =\calQ_{\bar z_1}\times\dots\times \calQ_{\bar z_{m+{2(n-m)\over 3}}},\quad 
	Q_z= \bar \calQ\times \calQ_{z_{m+{2(n-m)\over 3}+1}}\times\dots\times \calQ_{z_n}.
	\end{gathered}
\end{align*}
In view of the above definition, to apply \prettyref{lmm:assouad}, we choose
\begin{align*}
	\calZ=\calZ^*,\quad r={n-m\over 3}, \quad f(\calQ_z)=z,\quad d_j(z,z')=\indc{z_{m+{2(n-m)\over 3}+j}\neq z'_{m+{2(n-m)\over 3}+j}},\quad 
	\delta= 1.
\end{align*}
Hence, using \prettyref{lmm:assouad} we get that given any estimator $\hat z$ (which we can choose to be in $\calZ^*$) it satisfies
\begin{align}
	\label{eq:HT_1}
	\max_{\calQ_z\in \calF_r} \EE_z\qth{\sum_{i=1}^n\indc{\hat z_i\neq z_i}}
	&\geq {n-m\over 6}\min\{1-\TV(\calQ_z,\calQ_{z'}):z\sim z'\in \calZ^*\}
	\nonumber\\
	&= {n-m\over 6}(1-\TV(P_1,P_2)),
\end{align}
We further specify the error distributions corresponding to the labels $\{z_{m+1},\dots,z_n\}$ based on the decay function $G$. As $G$ is already differentiable on $(\sigma c_G,\infty)$ we can extend $G$ on $(0, \sigma c_G]$ such that it is differentiable throughout with $G(0)=1,G'(0)=0$. Then $1-G(\cdot)$ is a distribution function with a density $-G'$. We define 
\begin{align}\label{eq:HT_3}
	\begin{gathered}
		\{w_i\}_{i=m+{2(n-m)\over 3}+1}^n\iiddistr ~R\cdot V,\\
	\PP\qth{R\geq x} = 1-\decay(x/\sigma),\quad \PP\qth{V={{\theta_1-\theta_2\over \|\theta_1-\theta_2\|}}}=\frac 12 = \PP\qth{V={{\theta_2-\theta_1\over \|\theta_1-\theta_2\|}}}.
	\end{gathered}
\end{align}
In view of the above we can simplify \eqref{eq:HT_1} as
\begin{align}
	\label{eq:HT_2}
	\max_{\calQ_z\in \calF_r} \EE_z\qth{\sum_{i=1}^n\indc{\hat z_i\neq z}}\geq  {n-m\over 6}(1-\TV(P_1,P_2)),
\end{align}
where $P_i$ denotes the distribution of $\theta_i+R\cdot V$ for $i=1,2$. To analyze the total variation term in the above formula, we first note that without a loss of generality we can assume that $\theta_1,\theta_2$ lie on the real line with $\theta_1=-{\Delta\over 2},\theta_2={\Delta\over 2}$. This is because the total variation distance is invariant under location shifts and rotational transformations. Then the distributions in \eqref{eq:HT_3} can be simplified in terms of the density of $w_i$-s given by
\begin{align*}
	f_{w_i}(x) = -\frac 1{2\sigma}G'(|x|/\sigma),\quad i=m+{2(n-m)\over 3}+1,\dots,n.
\end{align*}
Hence, using a location shift argument, we get the densities of $P_1,P_2$ on $(-\infty,\infty)$ as
\begin{align}
	\label{eq:HT_0}
	dP_1(y) = -\frac 1{2\sigma} G'\pth{|y+{\Delta/2}|\over \sigma} dy,\quad 
	dP_2(y) = -\frac 1{2\sigma} G'\pth{|y-{\Delta/2}|\over \sigma} dy.
\end{align}
Then the total variation distance between $P_1,P_2$ can be bounded as
\begin{align*}
	\TV(P_1,P_2)
	&=\frac 12 \int_{-\infty}^\infty |dP_1(y) - dP_2(y)|
	\nonumber\\
	&=\frac 1{4\sigma} \int_{-\infty}^\infty \abs{G'\pth{\abs{y+{\Delta\over 2}}\over \sigma} - G'\pth{\abs{y-{\Delta\over 2}}\over \sigma}} dy
	\nonumber\\
	&\stepa{=}\frac 1{2\sigma} \int_{-\infty}^0 \abs{G'\pth{\abs{y+{\Delta\over 2}}\over \sigma} - G'\pth{\abs{y-{\Delta\over 2}}\over \sigma}} dy
	\nonumber\\
	&=\frac 1{2\sigma} \pth{\int_{-\infty}^{-{\Delta\over 2}-\sigma c_G}+
	\int_{-{\Delta\over 2}-\sigma c_G}^{-{\Delta\over 2}+\sigma c_G}
	+\int_{-{\Delta\over 2}+\sigma c_G}^{0}} 
	\nonumber\\ 
	& \quad \abs{-G'\pth{\abs{y+{\Delta\over 2}}\over \sigma} - \pth{-G'\pth{\abs{y-{\Delta\over 2}}\over \sigma}}} dy
	\nonumber\\
	&\stepb{\leq} -\frac 1{2\sigma} \pth{\int_{-\infty}^{-{\Delta\over 2}-\sigma c_G}
	+\int_{-{\Delta\over 2}+\sigma c_G}^{0}} {G'\pth{\abs{y+{\Delta\over 2}}\over \sigma}} dy
	\nonumber\\
	&~\quad -
	\frac 1{2\sigma} \int_{-{\Delta\over 2}-\sigma c_G}^{-{\Delta\over 2}+\sigma c_G}
	\pth{G'\pth{\abs{y+{\Delta\over 2}}\over \sigma} + G'\pth{\abs{y-{\Delta\over 2}}\over \sigma}} dy
\end{align*}	
where $c_G>0$ is as prescribed in \ref{pt:decay-prop} and
\begin{itemize}
	\item (a) followed as $\abs{G'\pth{\abs{y+{\Delta\over 2}}\over \sigma} - G'\pth{\abs{y-{\Delta\over 2}}\over \sigma}}$,  as a function of $y$, is symmetric about 0
	\item (b) followed as $G'(y)$ is negative for all $y>0$ and we allow ${\Delta}\geq 2\sigma c_G$ implies
	$$-G'\pth{\abs{y+{\Delta\over 2}}\over \sigma} \geq {-G'\pth{\abs{y-{\Delta\over 2}}\over \sigma}}\geq 0, \quad y\in (-\infty,-{\Delta\over 2}-\sigma c_G)\cup (-{\Delta\over 2}+\sigma c_G,0).$$
\end{itemize}
We continue the last inequality on $\TV(P_1,P_2)$ to get
\begin{align*}
	\TV(P_1,P_2)
	&\leq -\int_{-\infty}^{0} \frac 1{2\sigma} {G'\pth{\abs{y+{\Delta\over 2}}\over \sigma}} dy
	-  \int_{-{\Delta\over 2}-\sigma c_G}^{-{\Delta\over 2}+\sigma c_G}
		\frac 1{2\sigma} {G'\pth{\abs{y-{\Delta\over 2}}\over \sigma}} dy
	\nonumber\\
	&\stepa{=} 1+ \int_{0}^{\infty} \frac 1{2\sigma} {G'\pth{{y+{\Delta\over 2}}\over \sigma}} dy
	- \int_{-{\Delta\over 2}-\sigma c_G}^{-{\Delta\over 2}+\sigma c_G}
	\frac 1{2\sigma} {G'\pth{-\pth{y-{\Delta\over 2}}\over \sigma}} dy
	\nonumber\\
	&\stepb{=} 1 + \int_{\Delta/2}^{\infty} \frac 1{2\sigma} {G'\pth{z\over \sigma}} dz
	- \int_{{\Delta}-\sigma c_G}^{{\Delta}+\sigma c_G}
	\frac 1{2\sigma} {G'\pth{z\over \sigma}} dz
	\nonumber\\
	&
	\stepc{\leq} 1 + \int_{\Delta/2}^{\infty} \frac 1{2\sigma} {G'\pth{z\over \sigma}} dz
	- \int_{{\Delta/2}}^{{\Delta/2}+2\sigma c_G}
	\frac 1{2\sigma} {G'\pth{z\over \sigma}} dz
	= 1 - \frac 12 G\pth{{\Delta\over 2\sigma}+2c_G},
\end{align*}
where 
\begin{itemize}
	\item (a) followed as $-\frac 1{2\sigma} {G'\pth{\abs{y+{\Delta\over 2}}\over \sigma}}$ is a density on $(-\infty,\infty)$ from \eqref{eq:HT_0}
	\item (b) followed by change of variables
	\item (c) followed as $-G'(y)=|G'(y)|$ is decreasing over $(c_G,\infty)$ and $\Delta \geq 2\sigma c_G$.
\end{itemize} 
Plugging the last display in \eqref{eq:HT_2} completes the proof.

\iffalse

\section{Proofs for the sub-Gaussian mixture model}

\label{app:subG}

Note that the size of the set $S$ satisfies $|S|=(1-\delta)n_h^*\geq (1-\delta){n\alpha\over 2}\geq {n\alpha\over 4}$. For any fixed subset $S\subset [n]$ with $|S|\geq {n\alpha\over 4}$ we note that 
$$
\EE\qth{e^{\langle a,\frac 1{|S|}\sum_{i\in S}w_i \rangle}}
\leq e^{2\sigma^2\|a\|^2\over n\alpha},\quad  a\in \reals^d.
$$
Hence using \prettyref{eq:norm-conc} with $2\sqrt{dt}\leq t+d$ and $t=n$ we get that
\begin{align}
	\PP\qth{\|\frac 1{|S|}\sum_{i\in S}w_i\|^2>\sigma^2\cdot {4(2d+3n)\over n\alpha}}\leq e^{-n}.
\end{align}
Using an union bound over the at most $2^n$ possibilities we get for all $s\geq 1$
$$
\PP\qth{\max_{h\in [k]}\|\hat \theta_h^{(s)}-\theta_h\|^2>{4\sigma^2\over \alpha}(3+2d/n)}
\leq ke^{-0.3n}.
$$
As a consequence, we get $\Lambda^{(s)}=\frac 1{\Delta}\max_{h\in [k]}{\|\hat \theta_h^{(s)}-\theta_h\|}=\|\frac 1{|S|}\sum_{i\in S} w_i\|$. 

\fi

\section{Technical results}

\begin{proof}[Proof of \prettyref{lmm:E-norm-genr}]
	It suffices to show that for any set $S\subseteq [n]$
	$$
	\PP\qth{\sum_{i\in S}\indc{\norm{w_i} > \sigma\decay^{-1}\pth{\exp\sth{-\frac {1+1/\beta}{\tau}}}}\geq \tau |S|}
	\leq e^{-n},
	$$
	and then use the union bound over different choices of $S$ to get the result. We define 
	$$
	V_i=\indc{\norm{w_i} > \sigma\decay^{-1}\pth{\exp\sth{-\frac {1+1/\beta}{\tau}}}}.
	$$
	In view of the above definitions it is enough to show that
	\begin{align}
		\label{eq:ht9-genr}
		\PP\qth{\sum_{i\in S} V_{i} \geq \tau |S|}\leq e^{-n} \text{ for all } S\subseteq [n].
	\end{align}
	Note that
	\begin{align}
		\PP\qth{\norm{w_i} > \sigma\decay^{-1}\pth{\exp\sth{-\frac {1+1/\beta}{\tau}}}}
		 \leq {\exp\sth{-\frac {1+1/\beta}{\tau}}}\leq {\tau\over 1+1/\beta}<\tau.
	\end{align}
	We note the following result on stochastic dominance.
	\begin{lemma}\cite[Chapter 4.2]{roch2024modern}
		Given two random variables $X,Y\in \reals$, suppose that we call $X$ to be stochastically smaller than $Y$ if
		$
			\PP[X\geq a]\leq \PP[Y\geq a],\quad \forall a\in \reals.
		$ 
		Then, for $X\sim \Binom(n,p),Y\sim\Binom(n,q)$ for $0<p<q<1$, we get that $X$ is stochastically smaller than $Y$.
	\end{lemma}
	The above result implies that $\sum_{i\in S}V_i$ is stochastically smaller than any $Z\sim \Binom(|S|, {\exp\{-\frac {1+1/\beta}{\tau}}\})$. We continue to analyze \eqref{eq:ht9-genr} using Chernoff's inequality for the Binomial random variable:
	\begin{lemma}{\cite[Section 2.2]{boucheron2013concentration}}
		\label{lmm:chernoff}
		For a random variable $Z\sim \Binom(m,q)$, we have
		\begin{align*}
			\begin{gathered}
				\PP\qth{Z\geq ma}
				\leq \exp\pth{-mh_q(a)};\quad q<a<1,\
				h_q(a)=a\log{a\over q}+(1-a)\log{1-a\over 1-q}.
			\end{gathered}
		\end{align*}
	\end{lemma}
	Hence, using 
	$$
	Z\sim \Binom(|S|,q),\quad 
	q={\exp\sth{-\frac {1+1/\beta}{\tau}}},
	\quad a=\tau,
	\quad 
	m=|S|
	$$ 
	in the above lemma and $x\log x\geq -0.5$ for $x\in (0,1)$ we get
	\begin{align}\label{eq:ht11}
		\PP\qth{\sum_{i\in S} V_{i} \geq \tau|S|}
		&\leq \PP\qth{Z\geq \tau|S|}
		\nonumber \\
		&\leq \exp\pth{-m\pth{\tau\log{\tau\over q}+(1-\tau)\log{1-\tau\over 1-q}}}
		\nonumber\\
		&\leq  \exp\pth{-m\sth{{(1+ 1/\beta)}+\tau\log \tau + (1-\tau)\log(1-\tau)}}
		\leq e^{-n}.
	\end{align}
	Finally taking an union bound over all choices of $S$ we get the desired bound.
\end{proof}

	\section{Proof of the two cluster initialization result (\prettyref{thm:proof-algo-2})}
\label{sec:outlier-proof-init-2}

\subsection{Preparation}

\label{sec:n_i-conc_results}

For the proofs involving the initialization results, let $\calB(x,R)$ denote the Euclidean ball of radius $R$ around the point $x$. Our proofs rely on the following high probability guarantees.
\begin{lemma}\label{lmm:E-conc-init-2}
	There is an event $\tilde \calE$ with 
	$
	\PP[\tilde \calE]\geq 1-4e^{-{\min_{g=1,2} n_g^*\over 4}}
	$
	on which the following statements hold for the $2$-cluster problem, for any given $\beta\in (0,1)$:
	\begin{enumerate}[label=(\roman*)]
		\item $\abs{\calB(\theta_i,\sigma \decay^{-1}(e^{-{5\over 4\beta^2}}))\cap \sth{Y_i:i\in T^*_i}}\geq n^*_i(1-{\beta^2})$ for each $i=1,2$,
		\item $\abs{\calB(\theta_i,{\Delta\over 32})\cap \sth{Y_i:i\in T^*_i}}\geq n^*_i\pth{1-{5\over 4\log\pth{1/G\pth{\Delta/ (32\sigma)}}}}$ for each $i=1,2$.
	\end{enumerate} 
\end{lemma}

\begin{proof}
	The proof of part (i) follows by choosing $\epsilon_0={\sigma\over \Delta} \decay^{-1}(e^{-{5\over 4\beta^2}})$ in \prettyref{lmm:E-conc-genr}. The proof of part (ii) follows by choosing by choosing $\epsilon_0=\frac 1{32}$ in \prettyref{lmm:E-conc-genr}.
\end{proof}
%
%\begin{theorem}\label{thm:outlier-proof-algo-2}
%	Suppose that we have the data points $\{Y_1,\dots,Y_n\}$ out of which $n^*_i$ points came from the mixture component $i=1,2$ and $n^{\out}$ is the total number of outliers. In addition, the above counts satisfy the constraints
%	$$
%		n^*_1+n^*_2+n^\out = n,\quad n^*_1,n^*_2\geq n\alpha,\quad n^\out\leq {n\alpha^2\over 32},
%	$$ 
%	for constants $\alpha \in (0,1)$. Then there is a constant $C_{\decay,\alpha}>0$ such that if $\Delta\geq C_{\decay,\alpha}$ then the $\iod_{2,m_1,m,\beta}$ algorithm with $m_1=\ceil{n\alpha\over 4},m=\max\{1,\floor{n\alpha^2\over 16}\},\beta={\alpha\over 4}$ guarantees, for a permutation $\pi$ on $\{1,2\}$ 
%	$$
%	\max_{i=1,2}\|\theta_{\pi(i)}-\mu_i^*\|\leq \Delta/3
%	$$ 
%	with probability at least $1-4e^{-{n\alpha\over 4}}$.  
%\end{theorem}
%

\subsection{Proof of \prettyref{thm:proof-algo-2}}

In the proof below, we assume that all the mentioned constants depend on $\decay,\alpha,\sigma$, unless otherwise specified. We will provide a proof of the result involving the outliers, as in \prettyref{thm:proof-algo-outliers} as the proof in absence of the outliers is very similar. In summary, we show the following:

{\it 	Suppose that out of the $n$ many observed data points, $n^*_i$ many are from cluster $T_i^*, i=1,2$ and $n^\out$ many are adversarial outliers (i.e., $n_1^*+n_2^*+n^\out=n$). Also, assume that for some constant $\alpha>0$ the counts satisfy $n^*_1,n^*_2> n\alpha$, $n^\out\leq {n\alpha^2\over 32}$. Then the following holds with probability at least $1-4e^{-{n\alpha/ 4}}$. There are constants $c_1,c_2>0$ such that if $\Delta\geq c_1\sigma\decay^{-1}\pth{e^{-{c_2\over \alpha^2}}}$ then the centroid outputs from $\iod_{2,m_1,m,\beta}$ algorithm satisfy $
	\max_{i=1,2}\|\theta_{\pi(i)}-\mu_i^*\|\leq \Delta/3$ for a permutation $\pi$ on $\{1,2\}$ with 
	$$m_1=\ceil{n\alpha\over 4},m=\max\{1,\floor{n\alpha^2\over 16}\},\beta={\alpha\over 4}.$$
}

We prove the above statement here. For our entire analysis we will assume that the event $\tilde \calE$ in \prettyref{lmm:E-conc-init-2} holds, which has a high probability guarantee. We will extensively use the following definition of order statistics: Given any set $V$ of real numbers and fraction $0<p<1$, define $V^{\sth{p}}$ as the $\ceil{p|V|}$-th smallest number in $V$. Then, the proof is a combination of the following results.

\begin{lemma}\label{lmm:outlier-centroid1}
	There is one $\theta_i$, such that $\|\theta_i-\mu_1^{(1)}\|\leq 3\sigma \decay^{-1}(e^{-{5\over 4\beta^2}})$.
\end{lemma}

\begin{lemma}\label{lmm:outlier-first-jump}
	There is a stage $\ell+1$, with $\ell\geq 1$,  such that $\dist_1^{(\ell+1)}> {\Delta\over 16}$.	
\end{lemma}

\begin{lemma}\label{lmm:outlier-small-totdist}
	Suppose that $\ell=\min\sth{r\geq 1:\dist_1^{(r+1)}> {\Delta\over 16}}$. Then $\totdist_{\ell}\leq \Delta/8$.
\end{lemma}	

\begin{lemma}\label{lmm:outlier-apprx-centroid}
	If $\totdist_\ell\leq {\Delta\over 8}$, then there is a permutation $\pi$ of $\{1,2\}$ such that $$\max_{i=1,2} {\|\mu_i^{(\ell)}-\theta_{\pi(i)}\|}\leq {\Delta\over 3}.$$
\end{lemma}

\prettyref{lmm:outlier-centroid1}, \prettyref{lmm:outlier-first-jump} and \prettyref{lmm:outlier-small-totdist} together implies, provided $\Delta$ is large enough, that among all of the iterations of our algorithm there is an instance on which the $\totdist_\ell$ measure becomes smaller than ${\Delta\over 8}$. As our algorithm finally picks the iteration step $\ell=\ell^*$ with the lowest $\totdist_\ell$ measure, it ensures that $\totdist_{\ell^*}\leq {\Delta\over 8}$. In view of \prettyref{lmm:outlier-apprx-centroid} this implies $\max_{i=1,2}\|\theta_{\pi(i)}-\mu_i^*\|\leq \Delta/3$ as required. Below we provide the proofs.

\begin{proof}[Proof of \prettyref{lmm:outlier-centroid1}]
	In view of \prettyref{lmm:E-conc-init-2}, there is a constant $c_1=\sigma \decay^{-1}(e^{-{5\over 4\beta^2}})$ such that
	\begin{align}\label{eq:outlier-init1}
		|\sth{j\in [n]:Y_j\in \calB(\theta_i,c_1)}|\geq n^*_i\pth{1-\beta^2},\quad i\in \{1,2\}.
	\end{align}
	As we have $n^*_1,n^*_2>n\alpha$ by assumption, it follows that there is a point $Y_i$ such that 
	$$|\sth{j\in [n]:Y_j\in \calB(Y_i,2c_{1})}|\geq m_1 \geq {n\alpha\over 4}.$$
	Hence, the tightest neighborhood around any point $Y_i,i\in [n]$, that contains at least $n\alpha/4$ points from $Y_1,\dots,Y_n$, has a radius of at most $2c_{1}$ around that $Y_i$. Define
	\begin{align}\label{eq:outlier-distance}
		D({x},S)=\sth{\|x-Y_{i}\|:i\in S},\quad x\in \reals^d,S\subseteq [n].
	\end{align} 
	Let $i^*$ be one such index in $[n]$ that satisfies 
	\begin{align}
		\sth{D(Y_{i^*},[n])}^{\sth{m_1\over n}} = \min_{j\in [n]} \sth{D(Y_{j},[n])}^{\sth{m_1\over n}}.
	\end{align} 
	Then $\calB(Y_{i^*},2c_{1})$ and $\cup_{i=1,2}\calB(\theta_i,c_{1})$ can not be disjoint, as in view of \eqref{eq:outlier-init1} the disjointedness will imply that their union will contain more than $n$ points from $Y_1,\dots,Y_n$
	\begin{align*}
		&\abs{\sth{i\in [n]:Y_i\in \calB(Y_{i^*},2c_1)}\cup\qth{\cup_{j=1,2}\sth{i\in [n]:Y_i\in \calB(\theta_j,c_1)}}}
		\nonumber\\
		&\geq m_1 + \sum_{i=1,2}n^*_i\pth{1-\beta^2}
		= {n\alpha\over 4} + (n - n^\out)(1-\beta^2) 
		\geq n +{n\alpha\over 4} - n\beta^2 - n^\out
		\geq n + {n\alpha^2\over 8},
	\end{align*}	
	where we use the fact that $\sth{i\in [n]:Y_i\in \calB(\theta_i,c_1)},i=1,2$ are disjoint sets as $\|\theta_1-\theta_2\|\geq \Delta, \beta=\frac \alpha4, n^*_1+n^*_2+n^{\out}=n$ and $n^\out\leq {n\alpha^2\over 16}$. Hence, $Y_{i^*}$ is at a distance at most $3c_{1}$ from one of the true centroids $\theta_1,\theta_2$. Without a loss of generality we can pick $\mu_1^{(1)} = Y_{i^*}$ and we assume that $\mu_1^{(1)}$ is closer to  $\theta_1$ than $\theta_2$.	
\end{proof}

\begin{proof}[Proof of \prettyref{lmm:outlier-first-jump}]
	In view of the proof of \prettyref{lmm:outlier-centroid1}, let us assume that $\theta_1$ is the closest centroid to $\mu_1^{(1)}$ and define $c_1=\sigma \decay^{-1}(e^{-{5\over 4\beta^2}})$ as before to have
	\begin{align}\label{eq:outlier-init0}
		\mu_1^{(1)}\in \calB\pth{\theta_1,3c_1},\quad 
		\abs{\{Y_j:j\in T^*_i\}\cap\calB\pth{\theta_i,c_1}}\geq n^*_i\pth{1-\beta^2},i\in \{1,2\}.
	\end{align}
	We observe the following: 
	\begin{itemize}
		\item In view of $\calB(\mu_1^{(1)},4c_1)\supset\calB(\theta_1,c_1)$ we get
		\begin{align*}
			\abs{\sth{Y_i:i\in [n]}\cap \calB(\mu_1^{(1)},4c_1)}
			\geq \abs{\sth{Y_i:i\in [n]}\cap\calB(\theta_1,c_1)}
			\geq n\alpha(1-\beta^2)\geq {n\alpha\over 2}.
		\end{align*}
		As the size of $\calP_1^{(1)}$ is at most $m_1 = \ceil{n\alpha\over 4}$ the distance of $\mu_1^{(1)}$ to any point in $\calP_1^{(1)}$ is less than $4c_1$. As $\Delta\geq 64 c_1$ the last statement implies $\dist_1^{(1)}\leq {\Delta\over 16}$.
		\item At the last step, say $\tilde \ell$, in our algorithm, $\calP_1^{(\tilde \ell)}$ will have at least $n-m = n-{n\alpha^2\over 16}$ many points. In view of \eqref{eq:outlier-init0} we also have
		\begin{align*}
			|\sth{Y_j:j\in [n],Y_j\in \calB(\theta_2,c_1)}\cap \calP_1^{(\tilde\ell)}|
			&\geq
			|\sth{Y_j:j\in [n],Y_j\in \calB(\theta_2,c_1)}|-  |\overline{\calP_1^{(\tilde\ell)}}|
			\nonumber\\
			&\geq n^*_2\pth{1-{\beta^2}}-{n\alpha^2\over 16}
			\geq n\alpha-{n\alpha^2\over 16}-{n\alpha\beta^2}.
		\end{align*}
	\end{itemize} 
	As we have
	\begin{itemize}
		\item the tightest neighborhood in the data set around $\mu_1^{(1)}$ with a size at least $(1-\beta)|\calP_1^{(\tilde\ell)}|$, say, $N$, will include at least $(1-\beta)(n-{n\alpha^2\over 16})\geq n-{n\beta}-{n\alpha^2\over 16}$ points, and
		\item \eqref{eq:outlier-init0} implies that $\{Y_j:j\in T^*_2\}\cap\calB\pth{\theta_2,c_1}$ will contain at least ${n\alpha\over 2}$ points
	\end{itemize}
	we get that
	$$
	|N|+\abs{\{Y_j:j\in T^*_2\}\cap\calB\pth{\theta_2,c_1}} \geq n + {n\alpha\over 4}.
	$$ 
	This implies $N\cap \{Y_j: j\in T^*_2,\calB\pth{\theta_2,c_1}\}$ is nonempty. Suppose that $y$ is an element in the above set. Then we have that the distance of $y$ from $\mu_1^{(\tilde\ell)}$ is at least $\Delta-4c_1$, 
	\begin{align*}
		\|\mu_1^{(\tilde\ell)}-y\|
		\geq \|\theta_1-\theta_2\|-\|\mu_1^{(\tilde\ell)}-\theta_1\|-\|\theta_2-y\|
		\geq \Delta - 4c_1.
	\end{align*}
	Hence we get $\dist_1^{(\tilde \ell)}\geq \Delta-4c_1$. As $\Delta>64c_1$ we get $\dist_1^{(\tilde \ell)}\geq {\Delta\over 2}$ as required.
\end{proof}

\begin{proof}[Proof of \prettyref{lmm:outlier-small-totdist}]
	In view of \prettyref{lmm:outlier-centroid1}, without a loss of generality we assume that $\theta_1$ is the closest centroid to $\mu_1^{(1)}$ and \eqref{eq:outlier-init0}. We first prove the following claims:
	\begin{align}
		\abs{\calP_1^{(\ell)}\cap \{Y_i:i\in T^*_1\}}&\geq n^*_1-m-{5n\over 4\log(1/(G(\Delta/32\sigma)))}
		\label{eq:outlier-init-claim1}\\
		\abs{\calP_1^{(\ell)}\cap \{Y_i:i\in T^*_2\}}&\leq n\beta + {5n\over 4\log(1/(G(\Delta/32\sigma)))}.
		\label{eq:outlier-init-claim2}
	\end{align}
	The first claim \eqref{eq:outlier-init-claim1} follows from the following sequence of arguments
	\begin{itemize}
		\item Note that $\mu_1^{(\ell)}=\mu_1^{(\ell+1)}$ from the description in \prettyref{algo:init}. This implies $$\dist_1^{(\ell+1)}=\sth{D(\mu_1^{(\ell)},\calP_1^{(\ell+1)})}^{\sth{1-\beta}} > {\Delta\over 16}.$$ 
		As $\calP_1^{(\ell+1)}$ is constructed by including the data points according to increasing Euclidean distances from $\mu_1^{(\ell)}=\mu_1^{(\ell+1)}$ we get
		\begin{align}\label{eq:outlier-init6}
			\calP_1^{(\ell+1)}\supseteq \{Y_i:i\in [n]\}\cap \calB\pth{\mu_1^{(\ell)},{\Delta\over 16}}.
		\end{align}
		As we have $\calP_1^{(\ell)}\subset \calP_1^{(\ell+1)}$ and $|\calP_1^{(\ell)}|\leq |\calP_1^{(\ell+1)}|= |\calP_1^{(\ell)}|+m$, we get that there is a set $A\subseteq \{Y_i:i\in [n]\}$ that satisfies
		$$
		\calP_1^{(\ell)}\supseteq\calP_1^{(\ell+1)}/A,\quad |A| \leq m \leq  {n\alpha^2\over 16},
		$$
		for large enough $n$ such that ${n\alpha^2\over 16}>1$. In view of \eqref{eq:outlier-init6} the last display implies
		$$
		\calP_1^{(\ell)}\cap\{Y_i:i\in T^*_1\}\supseteq\calP_1^{(\ell+1)}\cap\{Y_i:i\in T^*_1\}/A
		\supseteq \calB\pth{\mu_1^{(\ell)},{\Delta\over 16}}\cap\{Y_i:i\in T^*_1\}/A,
		$$
		and hence
		\begin{align}\label{eq:outlier-init7}
			\abs{\calP_1^{(\ell)}\cap\{Y_i:i\in T^*_1\}}&\geq \abs{\calB\pth{\mu_1^{(\ell)},{\Delta\over 16}}\cap\{Y_i:i\in T^*_1\}}-|A|
			\nonumber\\
			& \geq \abs{\calB\pth{\mu_1^{(\ell)},{\Delta\over 16}}\cap\{Y_i:i\in T^*_1\}}-m.
		\end{align}
		
		\item As we have from \eqref{eq:outlier-init0}, with $\Delta\geq 96c_1$:
		\begin{align}
			\begin{gathered}
				\calB\pth{\mu_1^{(\ell)},{\Delta\over 16}}\supseteq \calB\pth{\theta_1,{\Delta\over 16}-3c_1}
				\supseteq \calB\pth{\theta_1,{\Delta\over 32}},
			\end{gathered}	
		\end{align}
		in view of \prettyref{lmm:E-conc-init-2}(ii) we get
		\begin{align}\label{eq:outlier-init8}
			&\abs{\calB\pth{\mu_1^{(\ell)},{\Delta\over 16}}\cap \{Y_i:i\in T^*_1\}}
			\geq \abs{\calB\pth{\theta_1,{\Delta\over 32}}\cap \{Y_i:i\in T^*_1\}}
			\nonumber\\
			&\geq n^*_1\pth{1-{5\over 4\log\pth{1/G\pth{\Delta\over 32\sigma}}}}
			\geq n^*_1-{5n\over 4\log(1/(G(\Delta/32\sigma)))}. 
		\end{align}
	\end{itemize}
	Combining \eqref{eq:outlier-init7} and \eqref{eq:outlier-init8} we get \eqref{eq:outlier-init-claim1}.
	Next, to prove the claim \eqref{eq:outlier-init-claim2}, we note that:
	\begin{itemize}
		\item In view of \prettyref{lmm:E-conc-init-2} there are at most ${5n\over 4\log(1/(G(\Delta/32\sigma)))}$ many points from $\sth{Y_i:i\in T^*_2}$ outside $\calB(\theta_2,{\Delta\over 32})$. As \eqref{eq:outlier-init0} implies $\calB\pth{\mu_1^{(\ell)},{\Delta\over 16}}\subseteq \calB\pth{\theta_1,{\Delta\over 16}+3c_1}$, and, $\calB\pth{\theta_1,{\Delta\over 16}+3c_1}$ and  $\calB(\theta_2,{\Delta\over 32})$ are disjoint, we have
		\begin{align}\label{eq:outlier-init9}
			\abs{\calP_1^{(\ell)}\cap \{Y_i:i\in T^*_2\}\cap \calB\pth{\mu_1^{(\ell)},{\Delta\over 16}}}
			\leq \abs{\{Y_i:i\in T^*_2\}/ \calB\pth{\theta_2,{\Delta\over 32}}}
			\leq {5n\over 4\log(1/(G(\Delta/32\sigma)))} .
		\end{align}
		\item On the other hand,
		$\dist_1^{(\ell)}=\sth{D(\mu_1^{(\ell)},\calP_1^{(\ell)})}^{\sth{1-\beta}}\leq {\Delta\over 16}$
		implies that
		\begin{align}\label{eq:outlier-init10}
			\abs{\calP_1^{(\ell)}\cap \{Y_i:i\in T^*_2\}/ \calB\pth{\mu_1^{(\ell)},{\Delta\over 16}}}\leq n\beta.
		\end{align}
	\end{itemize}
	Combining \eqref{eq:outlier-init9} and \eqref{eq:outlier-init10} we get \eqref{eq:outlier-init-claim2}.
	
	Hence, we have proven the inequalities \eqref{eq:outlier-init-claim1} and \eqref{eq:outlier-init-claim2}. These inequalities together imply
	\begin{align}
		\abs{\overline{\calP_1^{(\ell)}}\cap \{Y_i:i\in T^*_2\}}&\geq n^*_2- n\beta-{5n\over 4\log(1/(G(\Delta/32\sigma)))}
		\label{eq:outlier-init-claim3}\\
		\abs{\overline{\calP_1^{(\ell)}}\cap \{Y_i:i\in T^*_1\}}&\leq m +{5n\over 4\log(1/(G(\Delta/32\sigma)))}.
		\label{eq:outlier-init-claim4}
	\end{align}
	In view of $\abs{\sth{Y_i:i\in T^*_2}/ \calB(\theta_2,{\Delta/32})}\leq {5n\over 4\log(1/(G(\Delta/32\sigma)))}$ from \prettyref{lmm:E-conc-init-2}, we have  
	\begin{align}\label{eq:outlier-init2-1}
		&\abs{\overline{\calP_1^{(\ell)}}/\calB(\theta_2,{\Delta/ 32})}
		\nonumber\\
		&\leq \abs{\sth{\overline{\calP_1^{(\ell)}}\cap \{Y_i:i\in T^*_2\}/{\calB(\theta_2,{\Delta/ 32})}}
			\cup \sth{\overline{\calP_1^{(\ell)}}\cap \{Y_i:i\in T^*_1\}/ {\calB(\theta_2,{\Delta/ 32})}}}
		+n^\out
		\nonumber\\
		&\leq \abs{\{Y_i:i\in T^*_2\}/ \sth{\calB(\theta_2,{\Delta/ 32})}}
		+ \abs{\overline{\calP_1^{(\ell)}}\cap \{Y_i:i\in T^*_1\}}
		+n^\out
		\nonumber\\
		&\leq {5n\over 4\log(1/(G(\Delta/32\sigma)))} + m +{5n\over 4\log(1/(G(\Delta/32\sigma)))}
		+{n\alpha^2\over 32}
		\nonumber\\
		&\leq  {3n\alpha^2\over 32}+{5n\over 2\log(1/(G(\Delta/32\sigma)))}
		\leq {5n\alpha^2\over 32},
	\end{align}
	where the last inequality followed from ${5n\over 2\log(1/(G(\Delta/32\sigma)))}\leq {n\alpha^2\over 16}$ as $\Delta\geq 32\sigma \decay^{-1}(e^{-{40\over \alpha^2}})$. Then we make the following observations.
	\begin{itemize}
		\item As $\abs{\overline{\calP_1^{(\ell)}}}\geq{n\alpha}-{n\beta}-{n\alpha\over 16}\geq {11n\alpha\over 16}$ from \eqref{eq:outlier-init-claim3}, any subset of $\overline{\calP_1^{(\ell)}}$ with size ${\pth{1-\beta}|\overline{\calP_1^{(\ell)}}|}$, discards a set of size at least ${11n\alpha\beta\over 16}\geq {n\alpha^2\over 6}$ (note that $\beta={\alpha\over 4}$).
		\item From \eqref{eq:outlier-init2-1} we get $\abs{\overline{\calP_1^{(\ell)}}/\calB(\theta_2,{\Delta/ 32})}\leq {n\alpha^2\over 6.4}$. In view of the last argument this implies that the set $\overline{\calP_1^{(\ell)}}\cap\calB(\theta_2,{\Delta/ 32})$, which has a diameter at most $\Delta\over 16$, contains more points than any subset of $\overline{\calP_1^{(\ell)}}$ with size ${\pth{1-\beta}|\overline{\calP_1^{(\ell)}}|}$.
		\item Hence, the diameter of the tightest subset of $\overline{\calP_1^{(\ell)}}$ with size ${\pth{1-\beta}|\overline{\calP_1^{(\ell)}}|}$ is at most $\Delta\over 16$.
	\end{itemize} 
	This implies $\dist_2^{(\ell)}\leq \Delta/16$ and concludes our proof.
\end{proof}

\begin{proof}[Proof of \prettyref{lmm:outlier-apprx-centroid}]
	As $\totdist_\ell\leq {\Delta\over 8}$, we have $\dist_i^{(\ell)}\leq {\Delta\over 8}$ for $i\in \{1,2\}$. First we show that both $\mu_1^{(\ell)}$ and $\mu_2^{(\ell)}$ lie in $\cup_{i=1,2}\calB(\theta_i,{\Delta/ 3})$. If not, without a loss of generality let $\mu_2^{(\ell)}$ lie outside  $\cup_{i=1,2}\calB(\theta_i,\Delta/3)$. Then we have
	\begin{align}\label{eq:outlier-init3}
		\calB\pth{\mu_2^{(\ell)},{\Delta\over 8}}\cap \sth{\cup_{i\in \{1,2\}}\calB\pth{\theta_i,{\Delta\over 8}}}=\phi,.
	\end{align} 
	As we have $\dist_2^{(\ell)}\leq {\Delta\over 8}$, we get that 
	\begin{align}\label{eq:outlier-init4}
		\abs{\overline{\calP_1^{(\ell)}} \cap \calB\pth{\mu_2^{(\ell)},{\Delta\over 8}}}\geq \pth{1-\beta}\abs{\overline{\calP_1^{(\ell)}}}.
	\end{align}
	Using \prettyref{lmm:E-conc-init-2}, we get that 
	\begin{align}\label{eq:outlier-init5}
		\abs{\{Y_i:i\in [n]\}/ \sth{\cup_{i\in \{1,2\}}\calB\pth{\theta_i,{\Delta\over 8}}}}\leq {5n\over 4\log(1/(G(\Delta/32\sigma)))}+n^\out.
	\end{align}
	In view of the last display, using \eqref{eq:outlier-init3} and\eqref{eq:outlier-init4} we get
	\begin{align*}
		\abs{\overline{\calP_1^{(\ell)}}}
		&\leq \frac 1{1-\beta}\abs{\overline{\calP_1^{(\ell)}} \cap \calB\pth{\mu_2^{(\ell)},{\Delta\over 8}}}
		\nonumber\\
		&\leq \frac 1{1-\beta}
		\abs{\{Y_i:i\in [n]\}/ \sth{\cup_{i\in \{1,2\}}\calB\pth{\theta_i,{\Delta\over 8}}}}
		\leq {5n\over 2\log(1/(G(\Delta/32\sigma)))}
		+2n^\out.
	\end{align*}
	The last display implies 
	\begin{align}\label{eq:outlier-init2-2}
		\abs{\calP_1^{(\ell)}}\geq n-{5n\over 2\log(1/(G(\Delta/32\sigma)))} - 2n^\out
		\geq n-{n\alpha^2\over 8},
	\end{align}
	where the last inequality followed using $n^\out\leq {n\alpha^2\over 32}$, provided 
	$$
	{5n\over 2\log(1/(G(\Delta/32\sigma)))}\leq {n\alpha^2\over 16}, \quad \text{i.e.},\ \Delta\geq 32\sigma \decay^{-1}\pth{e^{-{40\over \alpha^2}}}.
	$$ 
	In view of \eqref{eq:outlier-init5} with $\cap_{i\in \{1,2\}}\calB\pth{\theta_i,{\Delta\over 8}}=\phi$, $n^\out\leq {n\alpha^2\over 32}$ and $n^*_1,n^*_2\geq n\alpha$, we get that given any set $\calS\subseteq \{Y_i:i\in [n]\}$ of size at least $n-{n\alpha\over 2}$, there will be at least two points in $\calS$ that are at least $\Delta-{\Delta\over 4}$ distance away. Choose $\calS = \{i\in [n]:Y_i\in \calB(\mu_1^{(\ell)},\dist_1^{(\ell)})\}$. Using \eqref{eq:outlier-init2-2} we get
	\begin{align*}
		|\calS|&\geq (1-\beta)\abs{\calP_1^{(\ell)}}
		\geq n(1-{\alpha/ 4})(1-\alpha^2/8)	\geq n-{n\alpha\over 2},
	\end{align*}
	Suppose $x,y$ are the farthest away points in $\calS$. This leads to a contradiction as 
	$$
	\Delta-{\Delta\over 4}\leq \|x-y\|\leq \|x-\mu_1^{(\ell)}\|+\|y-\mu_1^{(\ell)}\|
	\leq 2\cdot \dist_1^{(\ell)}\leq {\Delta\over 4}.
	$$
	Now it remains to show that $\mu_1^{(\ell)},\mu_2^{(\ell)}$ lie in different balls among $\calB(\theta_1,\Delta/3)$ and $\calB(\theta_2,\Delta/3)$. If not, then suppose that both lie in $\calB(\theta_1,\Delta/3)$. Note that either $\calP_1^{(\ell)}$ or $\overline{\calP_1^{(\ell)}}$ will contain more than half the points from $\{Y_i:i\in T^*_2\}\cap \calB\pth{\theta_2,{\Delta\over 8}}$. We deal with the case where $\overline{\calP_1^{(\ell)}}$ is the partition with more than half of the points in $\{Y_i:i\in T^*_2\}\cap \calB\pth{\theta_2,{\Delta\over 8}}$. The case with $\calP_1^{(\ell)}$ can be worked out similarly. Then we have the following
	\begin{itemize}
		\item As $\dist_2^{(\ell)}\leq {\Delta\over 8}$ and $\mu_2^{(\ell)}\in \calB(\theta_1,\Delta/3)$ we get
		$$
		\calB\pth{\mu_2^{(\ell)},\dist_2^{(\ell)}}
		\subseteq \calB\pth{\mu_2^{(\ell)},{\Delta\over 8}}\subseteq \calB\pth{\theta_1,{\Delta\over 3}+{\Delta\over 8}} = \calB\pth{\theta_1,{11\Delta\over 24}}
		$$
		\item In view of the last argument we have 
		\begin{align}\label{eq:init2-outlier-0}
			\calB\pth{\theta_2,{\Delta\over 8}}\cap \calB\pth{\mu_2^{(\ell)},\dist_2^{(\ell)}}\subseteq \calB\pth{\theta_2,{\Delta\over 8}}\cap \calB\pth{\theta_1,{11\Delta\over 24}}=\phi,
		\end{align}
		\item From \prettyref{lmm:E-conc-init-2}, whenever ${5n\over 8\log(1/(G(\Delta/32\sigma)))}\leq {n\alpha\over 6}$, i.e., $\Delta\geq 32\sigma \decay^{-1}\pth{e^{-{3\over 4\alpha}}}$, we get
		\begin{align}\label{eq:init2-outlier-1}
			\frac 12 \abs{\{Y_i:i\in T^*_2\}\cap \calB\pth{\theta_2,{\Delta\over 8}}}\geq {n\alpha\over 2}-{5n\over 8\log(1/(G(\Delta/32\sigma)))}\geq {n\alpha\over 3}.
		\end{align}
	\end{itemize} 
	However, this leads to a contradiction, as in view of \eqref{eq:init2-outlier-0} we have
	$$
	\abs{\overline{\calP_1^{(\ell)}}\cap \calB\pth{\theta_2,{\Delta\over 8}}}
	\leq \abs{\overline{\calP_1^{(\ell)}}/ \sth{Y_i: i\in [n]Y_i\in  \calB\pth{\mu_2^{(\ell)},\dist_2^{(\ell)}}}}
	\leq n\beta \leq {n\alpha\over 4},
	$$
	but on the other hand, using the fact that $\overline{\calP_1^{(\ell)}}$ contains more than half of the points in $\{Y_i:i\in T^*_2\}\cap \calB\pth{\theta_2,{\Delta\over 8}}$, we get from \eqref{eq:init2-outlier-1}
	$$
	\abs{\overline{\calP_1^{(\ell)}}\cap \calB\pth{\theta_2,{\Delta\over 8}}}
	\geq \frac 12 \abs{\{Y_i:i\in T^*_2\}\cap \calB\pth{\theta_2,{\Delta\over 8}}}\geq {n\alpha\over 3}.
	$$
\end{proof}

	\section{Proof of initialization results with $k$-clusters (\prettyref{thm:proof-algo-k})}

\label{app:proof-init-k}

We will use the following high probability guarantee for proving our initialization result. The proof is identical to that of \prettyref{lmm:E-conc-init-2} and is omitted.

\begin{lemma}\label{lmm:E-conc-init-k}
	Given any $\beta\in (0,1)$, there is an event $\tilde \calE_k$ with 
	$
	\PP\qth{\tilde \calE_k}\geq 1-2ke^{-{\min_{g\in [k]} n_g^*\over 4}}
	$
	on which the following holds for the $k$-cluster problem:
	\begin{enumerate}[label=(\roman*)]
		\item $\abs{\calB(\theta_i,\sigma \decay^{-1}(e^{-{5\over 4\beta^2}}))\cap \sth{Y_i:i\in T^*_i}}\geq n^*_i(1-{\beta^2})$ for each $i\in [k]$,
		\item $\abs{\calB(\theta_i,{\Delta\over 16 k})\cap \sth{Y_i:i\in T^*_i}}\geq n^*_i\pth{1-{5\over 4\log\pth{1/G\pth{\Delta/ (16\sigma k)}}}}$ for each $i\in [k]$.
	\end{enumerate} 
\end{lemma}

In the proof below, we assume that all the mentioned constants depend on $\decay,\alpha,\sigma,k$, unless otherwise specified. In addition, for our entire analysis we will assume that the event $\tilde \calE_k$ mentioned in \prettyref{lmm:E-conc-init-k} holds, which has a high probability. We will provide a proof of the result involving outliers, as in \prettyref{thm:proof-algo-outliers} as the proof in absence of outliers is very similar. For the sake of simplifying our analysis we will also assume that $\min_{h\in [k]}n_h^*={n\alpha\over k}$. In summary, the above modifications are equivalent to showing the following:

{\it Suppose that out of the $n$ many observed data points, $n^*_i$ many are from cluster $T_i^*, i=1,\dots,k$ and $n^\out$ many are adversarial outliers (i.e., $\sum_{i=1}^kn^*_i+n^\out=n$). Also assume for some constant $1>\alpha>0$ the counts satisfy $n^*_i> {n\alpha\over k}, i=1,\dots,k$, $n^\out\leq {n\alpha^2\over 64 k^3}$. Then there are constants $c_1,c_2$ such that the following is satisfied.  Whenever $\Delta>c_1k\sigma\decay^{-1}\pth{e^{-{c_2/ \beta^2}}}$, there is a permutation $\pi$ of the set $[k]$ that satisfies $\max_{i\in [k]}\|\theta_{\pi(i)}-\mu_i^*\|\leq \Delta/3$ with probability at least $1-2ke^{-n\alpha/4k}$, where the $\{\mu_i^*\}$ are centroid estimates generated via the $\iod_{k,m_1,m,\beta}$ algorithm with 
	$$m_1=\ceil{n\alpha\over 4k},m=\max\sth{1,\floor{n\beta^2\over 2}},\beta={\alpha\over 4k^2}.$$}
Similar to before, we will extensively use the following definition of order statistics: Given any set $V$ of real numbers and fraction $0<p<1$, let $V^{\sth{p}}$ define the $\ceil{p|V|}$-th smallest number in $V$. We make the following observations for simplifying the notation. Whenever we call the \iod algorithm to find $j$ centroids from the remaining data set, it contains a {\it for-loop} with the loop counter denoted by $\ell_j$. As a result, whenever we find a set of centroids $\hat \mu_k,\dots,\hat\mu_2, \hat\mu_2$ it corresponds to a set of loop counts $\tilde \ell_k,\dots,\tilde\ell_2$
\begin{align*}
	(\hat \mu_k,\dots,\hat\mu_2,\hat\mu_1)
	= (\mu_k^{(k,\tilde \ell_k)},\dots,\mu_2^{(2,\tilde \ell_2)},\mu_1^{(1,\tilde \ell_2)}),
\end{align*}
and vice-versa. In view of this relation, in the proofs below we will interchangeably use the centroids and the indices.

The proof is a combination of the following results.

\begin{lemma}\label{lmm:outlier-centroid1-k}
	There is one $\theta_i$, such that $\|\theta_i-\mu_k^{(k,1)}\|\leq 3\sigma G^{-1}\pth{e^{-{5\over 4\beta^2}}}$.
\end{lemma}

\begin{lemma}\label{lmm:outlier-first-jump-k}
	There is a stage $\bar \ell_k+1$, with $\bar \ell_k\geq 1$,  such that $\dist_k^{(\bar \ell_k+1)}> {\Delta\over 8k},\dist_k^{(\bar \ell_k)}\leq  {\Delta\over 8k}$.	
\end{lemma}

\begin{lemma} \label{lmm:outlier-small-totdist-k}
	There exists steps $\ell_{k},\dots,\ell_2$ such that for each $i=2,\dots,k$, at the $\ell_i$-th step the distance to the $(1-\beta)|\calP_i^{(\ell_i)}|$-th closest point from $\mu_i^{(i,\ell_i)}$, within the set $\calP_i^{(\ell_i)}$ will all be smaller than $\Delta\over 8k$ and the $(1-\beta)|\overline{\calP_2^{(\ell_2)}}|$-th closest point from $\mu_1^{(1,1)}$, within the set $\overline{\calP_2^{(\ell_2)}}$ will be smaller than $\Delta\over 8k$.
\end{lemma} 

\begin{lemma}\label{lmm:outlier-apprx-centroid-k}
	If $\totdist_k^{(\ell_k)}\leq {\Delta\over 8}$ for some $\ell_k$, then for the loop-index $\ell_k,\dots,\ell_2$ achieving the above we get that there is a permutation $\pi$ of $[k]$ such that (with $\ell_1$ being set as $\ell_2$) $$\max_{i\in [k]} {\|\mu_i^{(i,\ell_i)}-\theta_{\pi(i)}\|}\leq {\Delta\over 3}.$$
\end{lemma}

\prettyref{lmm:outlier-centroid1-k}, \prettyref{lmm:outlier-first-jump-k} and \prettyref{lmm:outlier-small-totdist-k} together implies, provided $\Delta$ is large enough, that among all of the iterations of our algorithm there is an instance on which the $\totdist^{(\ell)}_k$ measure becomes smaller than ${\Delta\over 8}$. As our algorithm finally picks the iteration step $\ell=\ell^*$ with the lowest $\totdist^{(\ell)}_k$ measure, it ensures that $\totdist^{(\ell^*_k)}_k\leq {\Delta\over 8}$. In view of \prettyref{lmm:outlier-apprx-centroid-k} this implies $\max_{i\in [k]}\|\theta_{\pi(i)}-\mu_i\|\leq \Delta/3$, for the centroid estimates $\mu_k,\dots,\mu_1$ generated at that iteration stage, as required. Now we prove below \prettyref{lmm:outlier-centroid1-k}, \prettyref{lmm:outlier-first-jump-k}, \prettyref{lmm:outlier-small-totdist-k} and \prettyref{lmm:outlier-apprx-centroid-k}.

\begin{proof}[Proof of \prettyref{lmm:outlier-centroid1-k}]
	
	In view of \prettyref{lmm:E-conc-init-k}, there is a constant $c_1>0$ such that
	\begin{align}\label{eq:outlier-init1-k}
		|\sth{j\in [n]:Y_j\in \calB(\theta_i,c_1)}|\geq n^*_i\pth{1-\beta^2},\quad i\in [k].
	\end{align}
	As we have $n^*_i\geq {n\alpha\over k}$ by assumption, it follows that there is a point $Y_i$ such that 
	$$|\sth{j\in [n]:Y_j\in \calB(Y_i,2c_{1})}|\geq {n\alpha\over 4k} (=m_1).$$
	Hence, the tightest neighborhood around any point $Y_i,i\in [n]$, that contains at least ${n\alpha\over 4k}$ points from $Y_1,\dots,Y_n$, has a radius of at most $2c_{1}$ around that $Y_i$. Using the definition \eqref{eq:outlier-distance}
	\begin{align*}
		D({x},S)=\sth{\|x-Y_{i}\|:i\in S},\quad x\in \reals^d,S\subseteq [n],
	\end{align*} 
	pick $i^*\in [n]$ that satisfies 
	\begin{align*}
		\sth{D(Y_{i^*},[n])}^{\sth{\frac {m_1}n}} = \min_{j\in [n]} \sth{D(Y_{j},[n])}^{\sth{\frac {m_1}n}}.
	\end{align*} 
	Then $\calB(Y_{i^*},2c_{1})$ and $\cup_{j\in [k]}\calB(\theta_j,c_{1})$ can not be disjoint, as in view of \eqref{eq:outlier-init1-k} it will imply that their union will contain more than $n$ points from $Y_1,\dots,Y_n$
	\begin{align*}
		&\abs{\sth{i\in [n]:Y_i\in \calB(Y_{i^*},2c_1)}\cup\qth{\cup_{j\in [k]}\sth{i\in [n]:Y_i\in \calB(\theta_j,c_1)}}}
		\nonumber\\
		&\geq {n\alpha\over 4k} + \sum_{j\in [k]}n^*_j\pth{1-{\beta^2}}
		\geq 
		(n-n^\out)(1-\beta^2) + {n\alpha\over 4k} > n + {n\alpha\over 8k},
	\end{align*}	
	where we assume that 
	\begin{align}\label{eq:out-req-1}
		n^\out\leq {n\alpha\over 16k},
	\end{align}
	and we use the fact that $\sth{i\in [n]:Y_i\in \calB(\theta_j,c_1)},j\in [k]$ are disjoint sets as $\min_{g\neq h\in [k]}\|\theta_g-\theta_h\|\geq \Delta$.
	Hence, $Y_{i^*}$ is at a distance at most $3c_{1}$ from one of the centroids. Without a loss of generality we can pick $\mu_k^{(k,1)} = Y_{i^*}$ and we assume that $\theta_k$ is the closest true centroid to $\mu_k^{(k,1)}$ than any of the other centroids.	
\end{proof}

\begin{proof}[Proof of \prettyref{lmm:outlier-first-jump-k}]
	In view of \prettyref{lmm:outlier-centroid1-k} we have for $c_1=\sigma\decay^{-1}\pth{e^{-{5\over 4\beta^2}}}$
	\begin{align}\label{eq:outlier-init0-k}
		\mu_k^{(k,1)}\in \calB\pth{\theta_k,3c_1},\quad 
		\abs{\{Y_i:i\in T^*_j\}\cap\calB\pth{\theta_j,c_1}}\geq n^*_j\pth{1-\beta^2},j\in [k].
	\end{align}
	We observe the following: 
	\begin{itemize}
		\item The set $\calB(\mu_k^{(k,1)},4c_1)$ contains $\calB(\theta_k,c_1)$, which contains at least ${n\alpha\over k}(1-\beta^2)$ points from $\{Y_i:i\in T^*_k\}$. As the size of $\calP_k^{(1)}$ is at most $\ceil{n\alpha\over 4k}$ the distance of $\mu_k^{(k,1)}$ to any point in $\calP_k^{(1)}$ is less than $4c_1$. As we have $\Delta\geq 32 k c_1$ the last statement implies $\dist_k^{(1)}\leq {\Delta\over 8k}$.
		\item 
		At the last step, say $\tilde \ell_k$, in the for loop indexed by $\ell_k$, $\calP_k^{(\tilde \ell_k)}$ will have at least $n-m$ many points (recall that $m=\floor{n\alpha^2\over 32k^4}\geq 1$ for large $n$). This implies: 
		\begin{enumerate}[label=(\alph*)]
			\item The tightest neighborhood (say $N$) around $\mu_k^{(k,1)}$ with a size at least $(1-\beta)|\calP_k^{(\tilde\ell_k)}|$ will include at least $(1-\beta)(n-m)\geq n-n\beta-m$ points,
			\item \eqref{eq:outlier-init0-k} implies that $\cup_{j\in [k-1]}\sth{\{Y_i:i\in T^*_j\}\cap\calB\pth{\theta_j,c_1}}$ will contain at least ${n\alpha\over 2k}$ points.
		\end{enumerate}
		Hence we get that the sets $N$ and $\cup_{j\in [k-1]}\sth{Y_i:i\in [n],Y_i\in \calB(\theta_j,c_1)}$ can not be disjoint as their union will contain at least $n$ points. The above implies that the neighborhood $N$ will contain at least one point $y$ from the set 
		$\cup_{j\in [k-1]}\sth{Y_i:i\in [n],Y_i\in \calB(\theta_j,c_1)}.$
		Suppose that $y\in N$ is such that
		$$
		y\in \sth{Y_i:i\in [n],Y_i\in \calB(\theta_j,c_1)}\
		\text{for some $j\in [k-1]$}.
		$$Then the distance of $y$ from $\mu_k^{(\tilde\ell_k)}$ is at least $\Delta-4c_1$, 
		\begin{align}\label{eq:outlier-init2-k}
			\|\mu_k^{(\tilde\ell_k)}-y\|
			\geq \|\theta_k-\theta_j\|-\|\mu_k^{(\tilde\ell_k)}-\theta_k\|-\|\theta_j-y\|
			\geq \Delta - 4c_1.
		\end{align}
	\end{itemize} 
	As $\Delta-4c_1\geq {\Delta\over 8k}$ we have that there exist some $1\leq \ell_k\leq n-1$ such that $\dist_k^{(\ell_k+1)} > {\Delta\over 8k}$. Then the following choice of $\bar \ell_k$ finishes the proof
	$$
	\bar \ell_k=\min\sth{r\geq 1:\dist_k^{(r+1)}> {\Delta\over 8k}}.
	$$
\end{proof}

\begin{proof}[Proof of \prettyref{lmm:outlier-small-totdist-k}]
	We will verify the result using an induction argument: The following is satisfied for each $i=k,k-1,\dots,2$ (induction variable).
	There exists an index value $\bar \ell_j$ corresponding the the $j$-th loop count $\ell_j$, for $j=k,\dots,2$ such that the corresponding centroids $\mu_k^{(k,{\bar \ell_k})},\dots,\mu_i^{(i,\bar \ell_i)}$ satisfy
	\begin{enumerate}[label=(Q\arabic*)]
		
		\item \label{prop:outlier-good-centroid} For each $g=k,\dots,2$, there is one $\theta_g$, such that $\|\theta_g-\mu_g^{(g,\ell_g)}\|\leq 3\sigma\decay^{-1}\pth{e^{-{5\over 4\beta^2}}}$.
		
		\item \label{prop:outlier-low-dist} At the $\ell_i$-th step the distance to the $(1-\beta)|\calP_i^{(\ell_i)}|$-th closest point from $\mu_i^{(i,1)}$, within the set $\calP_i^{(\ell_i)}$ will all be smaller than $\Delta\over 8k$.
		
		\item \label{prop:outlier-more-in-complement} For $h=1,\dots,i-1$ 
		$$\abs{\overline{\calP_i^{(\bar \ell_i)}}\cap {\{Y_j:j\in T^*_h\}}}
		\geq n^*_h- (k-i+1)n\beta -{5(k-i+1)n^*_h\over 4\log(1/(G(\Delta/16\sigma k)))}.$$
		\item \label{prop:outlier-less-in-complement}  $\abs{\overline{\calP_{i}^{(\bar \ell_i)}}\cap \sth{\cup_{g=i}^k\{Y_j:j\in T^*_g\}}}\leq (k-i+1)m + {5\sum_{g=i}^k n^*_g\over 4\log(1/G(\Delta/(16k\sigma)))}.$
		
		%			\item \label{prop:outlier-good-centroid-in-complement} There is a constant $\tilde C$ such that $\mu_{i-1}^{(i-1,1)}$ lies inside $\calB(\theta_{i-1},\tilde C)$.
	\end{enumerate}
	
	\paragraph*{\textbf{Base case $i=k$}} 
	Note that our algorithm starts by picking the tightest neighborhood with $m_1$ points, and then we keep adding $m$ points from $\overline{\calP_k^{(\ell_k)}}$ to $\calP_k^{(\ell_k)}$ at each step. In view of \prettyref{lmm:outlier-centroid1-k} we get that the estimate $\mu_k^{(k,1)}$ lies within a radius $3c_1$ of $\theta_k$, with $c_1 = \sigma\decay^{-1}\pth{e^{-{5\over 4\beta^2}}}$. Hence \ref{prop:outlier-good-centroid} is satisfied. In view of \prettyref{lmm:outlier-first-jump-k}, when we run the $k$-th for loop at the iteration $\bar \ell_k$, we get $\dist_k^{(\bar \ell_k)}\leq {\Delta\over 8k}$, and hence \ref{prop:outlier-low-dist} is satisfied.
	
	Let $\bar \ell_k$ be as in \prettyref{lmm:outlier-first-jump-k}. Without a loss of generality we assume that $\theta_k$ is the closest centroid to $\mu_k^{(k,1)}$ and \eqref{eq:outlier-init0-k} holds.
	We first prove the following claims:
	\begin{align}
		\abs{\calP_k^{(\bar \ell_k)}\cap \{Y_i:i\in T^*_k\}}&\geq n^*_k-m-{5n^*_k\over 4\log(1/(G(\Delta/16\sigma k)))}
		\label{eq:outlier-init-claim1-k}\\
		\abs{\calP_k^{(\bar \ell_k)}\cap {\{Y_i:i\in T^*_j\}}}&\leq n\beta + {5n^*_j\over 4\log(1/(G(\Delta/16\sigma k)))},\quad j\in [k-1].
		\label{eq:outlier-init-claim2-k}
	\end{align}
	The first claim \eqref{eq:outlier-init-claim1-k} follows from the following sequence of arguments (note the definition in \eqref{eq:outlier-distance})
	\begin{itemize}
		\item Using $\mu_k^{(k,\bar \ell_k)}=\mu_k^{(k,\bar \ell_k+1)}$ and from \prettyref{lmm:outlier-first-jump-k} we get $$\sth{D(\mu_k^{(k,\bar \ell_k)},\calP_k^{(\bar \ell_k+1)})}^{\sth{1-\beta}} =\sth{D(\mu_k^{(k,\bar \ell_k+1)},\calP_k^{(\bar \ell_k+1)})}^{\sth{1-\beta}} = \dist_k^{(\bar \ell_k+1)} >  {\Delta\over 8k},$$
		which implies 
		\begin{align}\label{eq:outlier-init6-k}
			\calP_k^{(\bar \ell_k+1)}\supseteq \{Y_i:i\in [n]\}\cap \calB\pth{\mu_k^{(k,\bar \ell_k)},{\Delta\over 8k}}.
		\end{align}
		As we have $\calP_k^{(\bar \ell_k)}\subset \calP_k^{(\bar \ell_k+1)}$ and $|\calP_k^{(\bar \ell_k)}|\leq |\calP_k^{(\bar \ell_k+1)}|\leq |\calP_k^{(\bar \ell_k)}|+m$, we get that there is a set $A\subseteq \{Y_i:i\in [n]\}$ that satisfies
		$$
		\calP_k^{(\bar \ell_k)}\supseteq\calP_k^{(\bar \ell_k+1)}/A,\quad |A|\leq m.
		$$
		In view of \eqref{eq:outlier-init6-k} the last display implies
		\begin{align}
			\sth{\calP_k^{(\bar \ell_k)}\cap\{Y_i:i\in T^*_k\}}
			&\supseteq \sth{\calP_k^{(\bar \ell_k+1)}\cap\{Y_i:i\in T^*_k\}/A}
			\nonumber\\
			&\supseteq \sth{\calB\pth{\mu_k^{(k,\bar \ell_k)},{\Delta\over 8k}}\cap\{Y_i:i\in T^*_k\}/A},
		\end{align}
		and hence
		\begin{align}\label{eq:outlier-init7-k}
			\abs{\calP_k^{(\bar \ell_k)}\cap\{Y_i:i\in T^*_k\}}&\geq \abs{\calB\pth{\mu_k^{(k,\bar \ell_k)},{\Delta\over 8k}}\cap\{Y_i:i\in T^*_k\}}-|A|
			\nonumber\\
			& \geq \abs{\calB\pth{\mu_k^{(k,\bar \ell_k)},{\Delta\over 8k}}\cap\{Y_i:i\in T^*_k\}}-{n\beta^2\over 2}.
		\end{align}
		
		\item Note that we have from \eqref{eq:outlier-init0-k} and $\Delta\geq 48\sigma k$:
		\begin{align}
			\begin{gathered}
				\calB\pth{\mu_k^{(k,\bar \ell_k)},{\Delta\over 8k}}\supseteq \calB\pth{\theta_k,{\Delta\over 8k}-3c_1}
				\supseteq \calB\pth{\theta_k,{\Delta\over 16k}}.
			\end{gathered}	
		\end{align}
	 	In view of \prettyref{lmm:E-conc-init-k}(ii) this implies
		\begin{align}\label{eq:outlier-init8-k}
			&\abs{\calB\pth{\mu_k^{(k,\bar \ell_k)},{\Delta\over 8k}}\cap \{Y_i:i\in T^*_k\}}
			\geq \abs{\calB\pth{\theta_k,{\Delta\over 16k}}\cap \{Y_i:i\in T^*_k\}}
			\nonumber\\
			&\geq n^*_k\pth{1-{5\over 4\log\pth{1/G\pth{\Delta\over 16\sigma k}}}}
			\geq n^*_k-{5n^*_k\over 4\log(1/(G(\Delta/16\sigma k)))}. 
		\end{align}
	\end{itemize}
	Combining \eqref{eq:outlier-init7-k} and \eqref{eq:outlier-init8-k} we get \eqref{eq:outlier-init-claim1-k}.
	
	Next, to prove the claim \eqref{eq:outlier-init-claim2-k}, we note that:
	\begin{itemize}
		\item In view of \prettyref{lmm:E-conc-init-k}, for each $j=1,\dots,k-1$, there are at most ${5n^*_j\over 4\log(1/(G(\Delta/16\sigma k)))}$ many points from $\sth{Y_i:i\in T^*_j}$ outside $\calB(\theta_j,{\Delta\over 16k})$. In view of \eqref{eq:outlier-init0-k} and $\mu_k^{(k,\bar \ell_k)}=\mu_k^{(k,1)}$ we get $\calB\pth{\mu_k^{(k,\bar \ell_k)},{\Delta\over 8k}}$ is a subset of $\calB\pth{\theta_k,{\Delta\over 8k}+3c_1}$ as $\Delta\geq 24c_1 k$. Hence we get $\calB\pth{\mu_k^{(k,\bar \ell_k)},{\Delta\over 8k}}$ and $\cup_{j=1}^{k-1}\calB(\theta_j,{\Delta\over 16k})$ are disjoint, and hence we have for each $j=1,\dots,k-1$
		\begin{align}\label{eq:outlier-init9-k}
			&\abs{\calP_k^{(\bar \ell_k)}\cap {\{Y_i:i\in T^*_j\}}\cap \calB\pth{\mu_k^{(k,\bar \ell_k)},{\Delta\over 8k}}}
			\nonumber\\
			&\leq \abs{{\{Y_i:i\in T^*_j\}}/ {\cup_{j=1}^{k-1}\calB\pth{\theta_j,{\Delta\over 16k}}}}
			\leq {5n^*_j\over 4\log(1/(G(\Delta/16\sigma k)))} .
		\end{align}
		\item On the other hand, from \prettyref{lmm:outlier-first-jump-k} we have
		$\dist_k^{(\bar \ell_k)}=\sth{D(\mu_k^{(k,\bar \ell_k)},\calP_k^{(\bar \ell_k)})}^{\sth{1-\beta}}\leq {\Delta\over 8k}$. This
		implies that for each $j\in [k-1]$
		\begin{align}\label{eq:outlier-init10-k}
			\abs{\calP_k^{(\bar \ell_k)}\cap {\{Y_i:i\in T^*_j\}}/ \calB\pth{\mu_k^{(k,\bar \ell_k)},{\Delta\over 8k}}}
			\leq {n\beta}.
		\end{align}
	\end{itemize}
	Combining \eqref{eq:outlier-init9-k} and \eqref{eq:outlier-init10-k} we get \eqref{eq:outlier-init-claim2-k}.
	
	Hence, we have proven the inequalities \eqref{eq:outlier-init-claim1-k} and \eqref{eq:outlier-init-claim2-k}. These inequalities together imply
	\begin{align}
		\abs{\overline{\calP_k^{(\bar \ell_k)}}\cap {\{Y_i:i\in T^*_j\}}}
		&\geq n^*_j- n\beta -{5n^*_j\over 4\log(1/(G(\Delta/16\sigma k)))},\quad j\in [k-1]
		\label{eq:outlier-init-claim3-k}\\
		\abs{\overline{\calP_k^{(\bar \ell_k)}}\cap \{Y_i:i\in T^*_k\}}&\leq m+{5n^*_k\over 4\log(1/(G(\Delta/16\sigma k)))}.
		\label{eq:outlier-init-claim4-k}
	\end{align}
	The first inequality above verifies \ref{prop:outlier-more-in-complement} and the second inequality verifies \ref{prop:outlier-less-in-complement}.

	\paragraph*{\textbf{Induction step from $i$ to $i-1$}} To complete the induction argument, let us suppose that the statement holds for some $3\leq i\leq k$ and we intend to prove the case of $i-1$. The proof of \ref{prop:outlier-good-centroid} follows from the following general result. The proof is essentially a repetition of argument as in the proof of \prettyref{lmm:outlier-centroid1-k}, and is presented at the end of this section.
	\begin{lemma}\label{lmm:outlier-good-centroid-in-complement}
		Suppose that we have for $i\geq 3$ and $h=1,\dots,i-1$
		\begin{align*}
			\abs{\overline{\calP_i^{(\bar \ell_i)}}\cap {\{Y_j:j\in T^*_h\}}}
			&\geq {3n^*_h\over 5}, \quad
			\abs{\overline{\calP_i^{(\bar \ell_i)}}\cap \sth{\cup_{g=i}^k\{Y_j:j\in T^*_g\}}}
			\leq {n\alpha\beta\over 5k}.
		\end{align*} 
		Then there is a centroid $\theta_{i-1}$ such that $\|\mu_{i-1}^{(i-1,1)}-\theta_{i-1}\|\leq 3\sigma\decay^{-1}\pth{e^{-{5\over 4\beta^2}}}$ if $\Delta\geq 16\sigma k\decay^{-1}\pth{e^{-{5k\over \alpha}}}$.
	\end{lemma}

	Next we prove \ref{prop:outlier-low-dist} for the loop indexed by $\ell_{i-1}$. In view of the \prettyref{lmm:outlier-good-centroid-in-complement} and \prettyref{lmm:E-conc-init-k} we note that
	\begin{align}\label{eq:outlier-init0-k-induction}
		\mu_{i-1}^{(i-1,1)}\in \calB\pth{\theta_{i-1}, 3c_1},\quad 
		\abs{\{Y_j:j\in T^*_h\}\cap\calB\pth{\theta_h,c_1}}\geq n^*_h\pth{1-\beta^2},h\in [i-1].
	\end{align}
	where $c_1 = \sigma\decay^{-1}\pth{e^{-{5\over 4\beta^2}}}$. In view of a reasoning similar as in the proof of \prettyref{lmm:outlier-first-jump-k} we note the following.
	As we keep adding $m$ points from $\overline{\calP_{i-1}^{(\ell_{i-1})}}$ to $\calP_{i-1}^{(\ell_{i-1})}$ at each step $\ell_{i-1}=2,\dots,\floor{n'- m_1\over m}$, note that at some stage $\ell_{i-1}$, before we exhaust all the points, the distance to the $(1-{\beta})|\calP_{i-1}^{(\ell_{i-1})}|$-th closest point from $\mu_{i-1}^{(i-1,1)}$, within the set $\calP_{i-1}^{(\ell_{i-1})}$ will exceed ${\Delta\over 8k}$. 
	Hence, to prove our claim \ref{prop:outlier-low-dist} we observe the following: 
	\begin{itemize}
		\item In view of \eqref{eq:outlier-init0-k-induction} we have $\calB(\mu_{i-1}^{(i-1,1)},4c_1)\supseteq\calB(\theta_{i-1},c_1)$, which implies
		\begin{align}\label{eq:outlier-initk-1}
			\abs{\{Y_j:j\in T^*_{i-1}\}/\calB(\mu_{i-1}^{(i-1,1)},4c_1)}
			\leq \abs{\{Y_j:j\in T^*_{i-1}\}/\calB(\theta_{i-1},c_1)}
			\leq n^*_{i-1}\beta^2.
		\end{align}
		In view of the assumption \ref{prop:outlier-more-in-complement} at the induction step $i$ the last display implies 
		\begin{align}
			&\abs{\overline{\calP_i^{(\bar \ell_i)}}\cap \{Y_j:j\in T^*_{i-1}\}\cap \calB(\mu_{i-1}^{(i-1,1)},4c_1)}
			\nonumber\\
			&=\abs{\overline{\calP_i^{(\bar \ell_i)}}\cap \{Y_j:j\in T^*_{i-1}\}}-\abs{\overline{\calP_i^{(\bar \ell_i)}}\cap \{Y_j:j\in T^*_{i-1}\}/ \calB(\mu_{i-1}^{(i-1,1)},4c_1)}
			\nonumber\\
			&\geq \abs{\overline{\calP_i^{(\bar \ell_i)}}\cap \{Y_j:j\in T^*_{i-1}\}}-\abs{\{Y_j:j\in T^*_{i-1}\}/ \calB(\mu_{i-1}^{(i-1,1)},4c_1)}
			\nonumber\\
			&\stepa{\geq} n^*_{i-1}- (k-i+1)n\beta - n^*_{i-1}\beta^2 -{5(k-i+1)n^*_{i-1}\over 4\log(1/(G\pth{{\Delta\over 16\sigma k}}))}
			\stepb{\geq} {n\alpha\over 2k},
		\end{align}
		where (a) used the inequality \eqref{eq:outlier-initk-1} and (b) holds whenever $\Delta\geq 16\sigma k\decay^{-1}\pth{e^{-{10k\over \alpha}}}$ as $(k-i+1)n\beta\leq {n\alpha\over 4k}, n\beta^2\leq {n\alpha\over 16k^4}$. As the size of $\calP_{i-1}^{(1)}$ is at most $\ceil{n\alpha\over 4k}$ the distance of $\mu_{i-1}^{({i-1},1)}$ to any point in $\calP_{i-1}^{(1)}$ is less than $4c_1$, which implies $\dist_{i-1}^{(1)}\leq {\Delta\over 8k}$.
		\item We first note that 
		at the last step, say $\tilde \ell_{i-1}$, in the for loop indexed by $\ell_{i-1}$, $\overline{\calP_{i-1}^{(\tilde \ell_{i-1})}}$ will have at most $m$ many points and
		\begin{align}\label{eq:temp3}
			{\overline{\calP_i^{(\tilde \ell_i)}}} = \overline{\calP_{i-1}^{(\tilde \ell_{i-1})}}\cup {\calP_{i-1}^{(\tilde \ell_{i-1})}},
			\quad
			\abs{\overline{\calP_{i-1}^{(\tilde \ell_{i-1})}}}\leq m.
		\end{align} 
		Hence, in view of \ref{prop:outlier-more-in-complement} and $n^*_1\geq 4n\beta$ we get
		\begin{align}\label{eq:outlier-init11-k}
			&\abs{\calP_{i-1}^{(\tilde \ell_{i-1})}\cap \{Y_j:j\in T^*_{1}\}}
			\nonumber\\
			&\geq \abs{\overline{\calP_i^{(\tilde \ell_i)}}\cap \{Y_j:j\in T^*_{1}\}} - m
			\stepa{\geq} n^*_1 - kn\beta - {n\beta^2\over 2} - {5kn^*_1\over 4\log(1/\decay({\Delta\over 16\sigma k}))}
			\stepb{\geq} 2n\beta, 
		\end{align} 
		where (a) followed from \eqref{eq:temp3}
		and (b) followed as $\Delta\geq 16\sigma k\decay^{-1}\pth{e^{-{5k^2\over \alpha}}}$ and $kn\beta\leq {n\alpha\over 4k}, n\beta^2\leq {n\alpha\over 16k^4}$. As the tightest neighborhood (say $N$) around $\mu_{i-1}^{(i-1,1)}$ with a size at least $(1-\beta)|\calP_{i-1}^{(\tilde\ell_{i-1})}|$ will exclude at most $n\beta$ points from $\calP_{i-1}^{(\tilde\ell_{i-1})}$, in view of \eqref{eq:outlier-init11-k} we get that the neighborhood $N$ will include at least $n\beta$ points from $\{Y_j:j\in T^*_{1}\}$. Now, \eqref{eq:outlier-init0-k} implies that $\{Y_i:i\in T^*_1\}\cap\calB\pth{\theta_1,c_1}$ will contain at lest $n^*_1(1-\beta^2)$ points from $\{Y_j:j\in T^*_{1}\}$, hence we get that the above neighborhood $N$ will contain at least one point $y\in \sth{Y_j:j\in [n],Y_i\in \calB(\theta_1,c_1)}$. Then the distance of $y$ from $\mu_{i-1}^{(i-1,\tilde\ell_{i-1})}$ is at least $\Delta-4c_1$, 
		\begin{align}
			\|\mu_{i-1}^{(i-1,\tilde\ell_{i-1})}-y\|
			\geq \|\theta_1-\theta_{i-1}\|-\|\mu_{i-1}^{(i-1,\tilde\ell_{i-1})}-\theta_{i-1}\|-\|\theta_1-y\|
			\geq \Delta - 4c_1.
		\end{align}
	\end{itemize} 
	Hence we have that there exist some $1\leq \ell_{i-1}\leq n-1$ such that $\dist_{i-1}^{(\ell_{i-1}+1)} > {\Delta\over 8k}$. Choose $\bar \ell_{i-1}$ as
	\begin{align}\label{eq:temp4}
	\bar \ell_{i-1}=\min\sth{r\geq 1:\dist_{i-1}^{(r+1)}> {\Delta\over 8k}}.
	\end{align}
	to satisfy the condition \ref{prop:outlier-low-dist}.
	%The set $\calP_{i-1}^{(\floor{n'-m_1\over m})}$ will contain at least $n'-m$ points and in the data set there are at least $32k^2(i-2)m$ points coming from the clusters with centroids $\theta_1,\dots,\theta_{i-2}$. As a result, the set of points that are closer to $\mu_{i-1}^{(i-1,\floor{n'-m_1\over m})}$ than the $(1-{\beta})|\calP_{i-1}^{(\floor{n'-m_1\over m})}|$-th closest point in $\calP_{i-1}^{(\floor{n'-m_1\over m})}$ will contain at least $32k^2(i-2)m-m$ points from those other clusters. In addition we note that $\mu_{i-1}^{({i-1},\floor{n'-m_1\over m})}$ is within $c_{k,\alpha,G}$ distance of $\theta_{i-1}$ and $$\abs{\cup_{j=1}^{i-2}\calB(\theta_j,{\Delta\over 4k})}\geq 32k^2(i-2)m-{nk\over \log(1/G(\Delta/4k))}\geq 16k^2(i-2)m.$$
	%	Hence the $(1-{\beta})|\calP_{i-1}^{(\floor{n'-m_1\over m})}|$-th closest distance from a $\mu_{i-1}^{(k,\floor{n'-m_1\over m})}$ will be at least $\Delta - c_{k,\alpha,G}-{\Delta\over 4k}$. 
	%	
	
	Next we establish \ref{prop:outlier-more-in-complement} and \ref{prop:outlier-less-in-complement} for the induction level ${i-1}$. 
	Let $\bar \ell_{i-1}$ is as in the last definition.
	We prove the following claims:
	For $h=1,\dots,i-2$
	\begin{align}
		\abs{\overline{\calP_{i-1}^{(\bar \ell_{i-1})}}\cap {\{Y_j:j\in T^*_{h}\}}}
		&\geq n^*_h- (k-i+2)n\beta -{5(k-i+2)n^*_h\over 4\log(1/(G(\Delta/16\sigma k)))},
		\label{eq:outlier-init-claim3-k-induction}
		\\
		\abs{\overline{\calP_{i-1}^{(\bar \ell_{i-1})}}\cap \{Y_j:j\in T^*_{i-1}\}}&\leq m+{5n^*_{i-1}\over 4\log(1/(G(\Delta/16\sigma k)))}.
		\label{eq:outlier-init-claim4-k-induction}
	\end{align}
	To prove the claim \eqref{eq:outlier-init-claim3-k-induction}, we note that:
	\begin{itemize}
		\item In view of \prettyref{lmm:E-conc-init-k}, for each $h=1,\dots,i-2$, there are at most ${5n^*_h\over 4\log(1/(G({\Delta\over 16\sigma k})))}$ many points from $\sth{Y_j:j\in T^*_h}$ outside $\calB(\theta_h,{\Delta\over 16k})$. In view of \eqref{eq:outlier-init0-k-induction} we get 
		$$
		\calB\pth{\mu_{i-1}^{({i-1},\bar \ell_{i-1})},{\Delta\over 8k}}\subseteq \calB\pth{\theta_{i-1},{\Delta\over 8k}+4c_1}.
		$$ 
		This implies $\calB\pth{\mu_{i-1}^{({i-1},\bar \ell_{i-1})},{\Delta\over 8k}}$ and $\cup_{h=1}^{i-2}\calB(\theta_h,{\Delta\over 16k})$ are disjoint. Hence for $h\in [i-2]$
		\begin{align}\label{eq:outlier-init9-k-induction}
			&\abs{\calP_{i-1}^{(\bar \ell_{i-1})}\cap {\{Y_j:j\in T^*_h\}}\cap \calB\pth{\mu_{i-1}^{(i-1,\bar \ell_{i-1})},{\Delta\over 8k}}}
			\nonumber\\
			&\leq \abs{{\{Y_j:j\in T^*_h\}}/ {\cup_{h=1}^{i-2}\calB\pth{\theta_h,{\Delta\over 16k}}}}
			\leq {5n^*_h\over 4\log(1/(G(\Delta/16\sigma k)))},
		\end{align}
		where the last inequality followed from \prettyref{lmm:E-conc-init-k}.
		\item On the other hand, in view of already proven \ref{prop:outlier-low-dist} at the induction step $i-1$ we get
		$\dist_{i-1}^{(\bar \ell_{i-1})}=\sth{D(\mu_{i-1}^{(i-1,\bar \ell_{i-1})},\calP_{i-1}^{(\bar \ell_{i-1})})}^{\sth{1-\beta}}\leq {\Delta\over 8k}$, which
		implies that for each $h\in [i-2]$
		\begin{align}\label{eq:outlier-init10-k-induction}
			\abs{\calP_{i-1}^{(\bar \ell_{i-1})}\cap {\{Y_j:j\in T^*_h\}}/ \calB\pth{\mu_{i-1}^{(i-1,\bar \ell_{i-1})},{\Delta\over 8k}}}
			\leq {n\beta}.
		\end{align}
	\end{itemize}
	Combining \eqref{eq:outlier-init9-k-induction} and \eqref{eq:outlier-init10-k-induction} we get
	\begin{align*}
		&\abs{\calP_{i-1}^{(\bar \ell_{i-1})}\cap {\{Y_j:j\in T^*_h\}}}
		\nonumber\\
		&\leq 
		\abs{\calP_{i-1}^{(\bar \ell_{i-1})}\cap {\{Y_j:j\in T^*_h\}}\cap \calB\pth{\mu_{i-1}^{(i-1,\bar \ell_{i-1})},{\Delta\over 8k}}}+\abs{\calP_{i-1}^{(\bar \ell_{i-1})}\cap {\{Y_j:j\in T^*_h\}}/ \calB\pth{\mu_{i-1}^{(i-1,\bar \ell_{i-1})},{\Delta\over 8k}}}
		\nonumber\\
		&\leq {n\beta} + {5n^*_h\over 4\log(1/(G(\Delta/16\sigma k)))}.
	\end{align*}
	Combining the above display with \ref{prop:outlier-more-in-complement} at the induction level $i$ we get for each $h\in [i-2]$
	\begin{align}
		\abs{\overline{\calP_{i-1}^{(\bar \ell_{i-1})}}\cap {\{Y_j:j\in T^*_{h}\}}}
		&=\abs{\overline{\calP_{i}^{(\bar \ell_{i})}}\cap {\{Y_j:j\in T^*_{h}\}}}
		-\abs{\calP_{i-1}^{(\bar \ell_{i-1})}\cap {\{Y_j:j\in T^*_h\}}}
		\nonumber\\
		&\geq n^*_h- (k-i+2)n\beta -{5(k-i+2)n^*_h\over 4\log(1/(G(\Delta/16\sigma k)))}.
	\end{align}
	This completes the verification of \ref{prop:outlier-more-in-complement} for the level $i-1$.
	
	The claim \eqref{eq:outlier-init-claim4-k-induction} follows from the following sequence of arguments (note the definition in \eqref{eq:outlier-distance})
	\begin{itemize}
		\item Using $\mu_{i-1}^{(i-1,\bar \ell_{i-1})}=\mu_{i-1}^{(i-1,\bar \ell_{i-1}+1)}$ and \eqref{eq:temp4} we get $$\sth{D(\mu_{i-1}^{({i-1},\bar \ell_{i-1})},\calP_{i-1}^{(\bar \ell_{i-1}+1)})}^{\sth{1-\beta}} = \sth{D(\mu_{i-1}^{({i-1},\bar \ell_{i-1}+1)},\calP_{i-1}^{(\bar \ell_{i-1}+1)})}^{\sth{1-\beta}}
		=\dist_{i-1}^{(\bar \ell_{i-1}+1)} > {\Delta\over 8k},$$
		which implies 
		\begin{align}\label{eq:outlier-init6-k-iterate}
			\calP_{i-1}^{(\bar \ell_{i-1}+1)}\supseteq \overline{\calP_{i}^{(\bar \ell_{i})}}\cap \calB\pth{\mu_{i-1}^{({i-1},\bar \ell_{i-1})},{\Delta\over 8k}}.
		\end{align}
		As we have $\calP_{i-1}^{(\bar \ell_{i-1})}\subset \calP_{i-1}^{(\bar \ell_{i-1}+1)}$ and $|\calP_{i-1}^{(\bar \ell_{i-1})}|\leq |\calP_{i-1}^{(\bar \ell_{i-1}+1)}|\leq |\calP_{i-1}^{(\bar \ell_{i-1})}|+{n\beta^2\over 2}$, we get that there is a set $A\subseteq \{Y_i:i\in [n]\}$ that satisfies
		$$
		\calP_{i-1}^{(\bar \ell_{i-1})}\supseteq\calP_{i-1}^{(\bar \ell_{i-1}+1)}/A,\quad |A|\leq m.
		$$
		In view of \eqref{eq:outlier-init6-k-iterate} the last display implies
		\begin{align}
			\sth{\calP_{i-1}^{(\bar \ell_{i-1})}\cap\{Y_j:j\in T^*_{i-1}\}}
			&\supseteq
			\sth{\calP_{i-1}^{(\bar \ell_{i-1}+1)}\cap\{Y_j:j\in T^*_{i-1}\}/A}
			\nonumber\\
			&\supseteq 
			\sth{\overline{\calP_{i}^{(\bar \ell_{i})}}\cap \calB\pth{\mu_{i-1}^{({i-1},\bar \ell_{i-1})},{\Delta\over 8k}}\cap\{Y_j:j\in T^*_{i-1}\}/A},
		\end{align}
		and hence
		\begin{align}\label{eq:outlier-init7-k-induction}
			\abs{\calP_{i-1}^{(\bar \ell_{i-1})}\cap\{Y_j:j\in T^*_{i-1}\}}&\geq \abs{\overline{\calP_{i}^{(\bar \ell_{i})}}\cap\calB\pth{\mu_{i-1}^{({i-1},\bar \ell_{i-1})},{\Delta\over 8k}}\cap\{Y_j:j\in T^*_{i-1}\}}-|A|
			\nonumber\\
			& \geq \abs{\overline{\calP_{i}^{(\bar \ell_{i})}}\cap\calB\pth{\mu_{i-1}^{({i-1},\bar \ell_{i-1})},{\Delta\over 8k}}\cap\{Y_j:j\in T^*_{i-1}\}}-m.
		\end{align}
		
		\item As we have from \eqref{eq:outlier-init0-k-induction} and $\Delta\geq 48c_1 k$:
		\begin{align*}
			\begin{gathered}
				\calB\pth{\mu_{i-1}^{({i-1},\bar \ell_{i-1})},{\Delta\over 8k}}\supseteq \calB\pth{\theta_{i-1},{\Delta\over 8k}-3c_1}
				\supseteq \calB\pth{\theta_{i-1},{\Delta\over 16k}},
			\end{gathered}	
		\end{align*}
		in view of \prettyref{lmm:E-conc-init-k}(ii) we get
		\begin{align}\label{eq:outlier-init8-k-induction}
			&\abs{\overline{\calP_{i}^{(\bar \ell_{i})}}\cap\calB\pth{\mu_{i-1}^{({i-1},\bar \ell_{i-1})},{\Delta\over 8k}}\cap \{Y_j:j\in T^*_{i-1}\}}
			\nonumber\\
			&\geq \abs{\overline{\calP_{i}^{(\bar \ell_{i})}}\cap\calB\pth{\theta_{i-1},{\Delta\over 16k}}\cap \{Y_j:j\in T^*_{i-1}\}}
			\nonumber\\
			& {\geq} \abs{\overline{\calP_{i}^{(\bar \ell_{i})}}\cap \{Y_j:j\in T^*_{i-1}\}}
			- \abs{\{Y_j:j\in T^*_{i-1}\}/\calB\pth{\theta_{i-1},{\Delta\over 16k}}}
			\nonumber\\
			&\geq \abs{\overline{\calP_{i}^{(\bar \ell_{i})}}\cap \{Y_j:j\in T^*_{i-1}\}}
			- {5n^*_{i-1}\over 4\log(1/(G(\Delta/16\sigma k)))}.\,
		\end{align}
	where the last inequality followed from \prettyref{lmm:E-conc-init-k}.
	\end{itemize}
	Combining \eqref{eq:outlier-init7-k-induction}, \eqref{eq:outlier-init8-k-induction} and $\overline{\calP_i^{(\bar \ell_i)}}={\calP_{i-1}^{(\bar \ell_{i-1})}} \cup \overline{\calP_{i-1}^{(\bar \ell_{i-1})}}$ we get
	\begin{align*}
		\abs{\overline{\calP_{i-1}^{(\bar \ell_{i-1})}}\cap \{Y_j:j\in T^*_{i-1}\}}
		&= \abs{\overline{\calP_{i}^{(\bar \ell_{i})}}\cap \{Y_j:j\in T^*_{i-1}\}}
		-	\abs{\calP_{i-1}^{(\bar \ell_{i-1})}\cap\{Y_j:j\in T^*_{i-1}\}}
		\nonumber\\
		&\leq  m +{5n^*_{i-1}\over 4\log(1/(G(\Delta/16\sigma k)))}.
	\end{align*}
	In view of \ref{prop:outlier-less-in-complement} for the induction level $i$, with $\overline{\calP_{i-1}^{(\bar \ell_{i-1})}}\subseteq \overline{\calP_{i}^{(\bar \ell_{i})}}$ we get
	\begin{align}
		\abs{\overline{\calP_{i-1}^{(\bar \ell_{i-1})}}\cap \sth{\cup_{g=i-1}^k\{Y_j:j\in T^*_g\}}}
		&\leq \abs{\overline{\calP_{i}^{(\bar \ell_i)}}\cap \sth{\cup_{g=i}^k\{Y_j:j\in T^*_g\}}} + \abs{\overline{\calP_{i-1}^{(\bar \ell_{i-1})}}\cap \{Y_j:j\in T^*_{i-1}\}}
		\nonumber\\
		&\leq (k-i+2) m +{5\sum_{g=i-1}^k n^*_g\over 4\log(1/G(\Delta/(16k\sigma)))}.
	\end{align}
	This concludes the verification of \ref{prop:outlier-less-in-complement} for the induction level $i-1$. This also concludes the proof of the induction results.

	In view of the induction arguments, we have that 
	\begin{align}\label{eq:outlier-primary-distances-k}
		\dist_i^{(\bar \ell_i)}\leq {\Delta\over 8k},\quad i=k,k-1,\dots,2.
	\end{align} 
	Finally, to complete the proof of \prettyref{lmm:outlier-small-totdist-k} it remains to show that $\dist_1^{(\ell_2)}\leq {\Delta\over 8k}$. In view of the induction statement we have that 
	\begin{enumerate}
		\item $\abs{\overline{\calP_2^{(\bar \ell_2)}}\cap {\{Y_j:j\in T^*_1\}}}
		\geq n^*_1- (k-1)n\beta -{5(k-1)n^*_1\over 4\log(1/(G(\Delta/16\sigma k)))}.$
		\item $\abs{\overline{\calP_{2}^{(\bar \ell_2)}}\cap \sth{\cup_{g=2}^k\{Y_j:j\in T^*_g\}}}
		\leq (k-1)m+{5\sum_{g=2}^k n^*_g\over 4\log(1/G(\Delta/(16\sigma k)))}
		\leq {n\alpha\beta\over 8k}+{5\sum_{g=2}^k n^*_g\over 4\log(1/G(\Delta/(16\sigma k)))}.$
	\end{enumerate}
	According to \prettyref{algo:init-k}, in the final stage, to find $\mu_1^{(1,\bar \ell_2)}$ we deploy the ${\sf HDP}_{1-\beta}$ algorithm. In view of 
	$$
	\abs{\sth{Y_j:j\in T^*_1}/ \calB\pth{\theta_1,{\Delta\over 16k}}}\leq {5n^*_1\over 4\log(1/(G(\Delta/16\sigma k)))}
	$$ 
	from \prettyref{lmm:E-conc-init-k}, we have  
	\begin{align*}
		&\abs{\overline{\calP_{2}^{(\bar \ell_2)}}/\calB\pth{\theta_1,{\Delta\over 16k}}}
		\nonumber\\
		&=\abs{\sth{\overline{\calP_{2}^{(\bar \ell_2)}}\cap \{Y_j:j\in T^*_1\}/\calB\pth{\theta_1,{\Delta\over 16k}}}
			\cup \sth{\overline{\calP_{2}^{(\bar \ell_2)}}\cap \sth{\cup_{g=2}^k\{Y_j:j\in T^*_g\}}/\calB\pth{\theta_1,{\Delta\over 16k}}}}
		+n^\out
		\nonumber\\
		&\leq \abs{\{Y_j:j\in T^*_1\}/\calB\pth{\theta_1,{\Delta\over 16k}}}
		+ \abs{\overline{\calP_{2}^{(\bar \ell_2)}}\cap \sth{\cup_{g=2}^k\{Y_j:j\in T^*_g\}}}
		+n^\out
		\nonumber\\
		&\leq {5n^*_1\over 4\log(1/(G(\Delta/16\sigma k)))} + (k-1){n\beta^2\over 2}+{5\sum_{g=2}^k n^*_g\over 4\log(1/G(\Delta/(16\sigma k)))}
		+n^\out
		\nonumber\\
		&\leq  {n\alpha\beta\over 8k}+{5n\over 2\log(1/(G(\Delta/(16\sigma k))))}+n^\out
		\leq {n^*_1\beta\over 4},
	\end{align*}
	where for the last inequality we assume that 
	\begin{align}\label{eq:out-req-2}
		\Delta\geq 16 k\sigma\decay^{-1}\pth{c^{-{40\over \alpha\beta}}},
		\quad n^\out \leq {n\alpha\beta\over 16k}.
	\end{align}
	As we have
	\begin{align}
		\abs{\overline{\calP_2^{(\bar \ell_2)}}}\geq \abs{\overline{\calP_2^{(\bar \ell_2)}}\cap {\{Y_j:j\in T^*_1\}}}\geq {n^*_1\over 2},
	\end{align} 
	any subset of $\overline{\calP_2^{(\bar \ell_2)}}$ with size $\pth{1-\beta}\abs{\overline{\calP_2^{(\bar \ell_2)}}}$, discards a set of size at least ${n^*_1\beta\over 2}$ from $\overline{\calP_2^{(\bar \ell_2)}}$. Hence the tightest subset of $\overline{\calP_2^{(\bar \ell_2)}}$ with size $\pth{1-{\beta}}\abs{\overline{\calP_1^{(\ell)}}}$ will have a diameter of at most $\Delta\over 16k$. This implies $\dist_1^{(\bar \ell_2)}\leq {\Delta\over 16 k}$. In view of \eqref{eq:outlier-primary-distances-k} this proves that there is a path of indices $\bar \ell_k,\dots\bar\ell_2$ such that $\dist_1^{(\bar\ell_2)}+\sum_{h=2}^k\dist_h^{(\bar \ell_h)}\leq {\Delta\over 8}$. Hence, when we pick the indices to optimize $\totdist$, we get $\min_{\ell_k}\totdist_k^{(\ell_k)}\leq {\Delta\over 8}$, as required.
\end{proof}

\begin{proof}[Proof of \prettyref{lmm:outlier-small-totdist-k}]
	
Note that the term $\totdist_k^{(\ell_k)}$ is given by $\sum_{i=1}^k\dist_{i}^{(\ell_i)}$ for some sequence of indices originating from the inbuilt {\it for-loops} at different levels $\ell_k,\dots,\ell_2$ and $\ell_2=\ell_1$. Hence it suffices to prove that if the sum $\sum_{i=1}^k\dist_{i}^{(\ell_i)}$ is smaller than ${\Delta\over 8}$ for any sequence of the loop counts we have good centroid approximations. This is summarized in the following result. 

\begin{lemma}
	Suppose that for a sequence of indices $\ell_1,\dots,\ell_k$ we have $\sum_{i=1}^k\dist_{i}^{(\ell_i)}\leq {\Delta\over 8}$. Then if the corresponding centroids are $\{\mu_i^{(i,\ell_i)}\}_{i=1}^k$, with $\ell_2=\ell_1$, we get that there is a permutation $\pi$ of $[k]$ such that $\mu_i^{i,\ell_i}\in \calB(\theta_{\pi(i)},\Delta/3)$ for each $i\in 1,\dots, k$.
\end{lemma}

\begin{proof}

	First we show that all of the centroids lie in $\cup_{i=1}^k\calB(\theta_i,\Delta/3)$. If not, without a loss of generality let $\mu_1^{(1,\ell_1)}$ lie outside  $\cup_{i=1}^k\calB(\theta_i,\Delta/3)$. Then we have
	\begin{align}\label{eq:outlier-init12-k}
		&\abs{\{Y_i:i\in [n]\}\cap \calB\pth{\mu_1^{(1,\ell_1)}, {\Delta\over 8}}}
		\nonumber\\
		&\leq \abs{\{Y_i:i\in [n]\}/ \cup_{i=1}^k\calB(\theta_i,{\Delta\over 8})}
		\leq {5n\over 4\log(1/(G(\Delta/16\sigma k)))}+n^\out,
	\end{align} 
	where the last inequality followed from \prettyref{lmm:E-conc-init-k}.
	For an ease of notation, throughout the proof we define 
	\begin{align}
		\ell_1\eqdef\ell_2,\quad \calP_1^{(\ell_{1})}\eqdef\overline{\calP_1^{(\ell_{2})}},\quad
		\mu_1^{(1,\ell_1)}
		\eqdef \mu_1^{(1,\ell_2)}.
	\end{align}
	Note that in terms of the indices $\ell_2,\dots,\ell_k$ we have the partition of $\{Y_i:i\in [n]\}$ as
	\begin{align}
		\{Y_i:i\in [n]\}={\cup_{g=1}^k{\calP_g^{(\ell_g)}}},\quad
		\calP_g^{(\ell_g)}\cap \calP_h^{(\ell_h)}=\phi,g\neq h\in [k].
	\end{align}
	In view of the assumption $\dist_1^{(\ell_1)}\leq {\Delta\over 8}$ and the fact that the $\abs{\calP_1^{(\ell_1)}\cap\calB\pth{\mu_1^{(1,\ell_1)},\dist_1^{(\ell_1)}}}\geq (1-\beta)\abs{\calP_1^{(\ell_1)}}$ we have 
	\begin{align}
		\abs{\calP_1^{(\ell_1)}/ \calB\pth{\mu_1^{(1,\ell_1)}, {\Delta\over 8}}}
		\leq 
		\abs{\calP_1^{(\ell_1)}/ \calB\pth{\mu_1^{(1,\ell_1)}, \dist_1^{(\ell_1)}}}
		\leq n\beta.
	\end{align}
	In view of \eqref{eq:outlier-init12-k} the last display implies
	\begin{align}
		\abs{\calP_1^{(\ell_1)}}
		&\leq \abs{\calP_1^{(\ell_1)}/ \calB\pth{\mu_1^{(1,\ell_1)}, {\Delta\over 8}}}
		+	\abs{\{Y_i:i\in [n]\}\cap \calB\pth{\mu_1^{(1,\ell_1)}, {\Delta\over 8}}}
		\nonumber\\
		&\leq n\beta + {3n\over 4\log(1/(G(\Delta/16\sigma k)))}+n^\out. 
	\end{align}
	As $\cup_{i=2}^{k}\calP_i^{(\ell_i)}$ and $\calP_1^{(\ell_1)}$ are disjoint and their union covers all the data points, the last display implies for any $j=2,\dots,k$
	\begin{align}\label{eq:outlier-init13-k}
		\abs{\{Y_i:i\in T^*_j\}\cap\sth{\cup_{g=2}^{k}\calP_g^{(\ell_g)}}}
		&=\abs{\{Y_i:i\in T^*_j\}/ \calP_1^{(\ell_1)}}
		\nonumber\\
		&\geq n^*_j-{n\beta}-{5n\over 4\log(1/(G(\Delta/16\sigma k)))}-n^\out
		\geq {7n\alpha\over 8k},
	\end{align}
	where the last inequality follows from $\beta\leq {\alpha\over 12k}$ as $k\geq 3$, $\Delta\geq $ and we assume that 
	\begin{align}\label{eq:out-req-3}
		n^\out\leq {n\beta\over 2}.
	\end{align}
	Then we have for $j=1,\dots,k$
	\begin{align}
		&\abs{\{Y_i:i\in T^*_j\}\cap\qth{\cup_{g=2}^{k}\sth{\calP_g^{(\ell_g)}\cap \calB(\mu_g^{(g,\ell_g)},\dist_g^{(\ell_g)})}}}
		\nonumber\\
		&=\abs{\{Y_i:i\in T^*_j\}\cap \sth{\cup_{g=2}^{k}\calP_g^{(\ell_g)}}\cap\qth{\cup_{g=2}^{k}\sth{\calP_g^{(\ell_g)}\cap \calB(\mu_g^{(g,\ell_g)},\dist_g^{(\ell_g)})}}}
		\nonumber\\
		&= \abs{\{Y_i:i\in T^*_j\}\cap \sth{\cup_{g=2}^{k}\calP_g^{(\ell_g)}}}\nonumber\\
		&\quad -\abs{\{Y_i:i\in T^*_j\}\cap \sth{\cup_{g=2}^{k}\calP_g^{(\ell_g)}}/\qth{\cup_{g=2}^{k}\sth{\calP_g^{(\ell_g)}\cap \calB(\mu_g^{(g,\ell_g)},\dist_g^{(\ell_g)})}}}
		\nonumber\\
		&\geq \abs{\{Y_i:i\in T^*_j\}\cap \sth{\cup_{g=2}^{k}\calP_g^{(\ell_g)}}}\nonumber\\
		&\quad-\abs{ \sth{\cup_{g=2}^{k}\calP_g^{(\ell_g)}}/\qth{\cup_{g=2}^{k}\sth{\calP_g^{(\ell_g)}\cap \calB(\mu_g^{(g,\ell_g)},\dist_g^{(\ell_g)})}}}
		\nonumber\\
		&\stepa{\geq} {7n\alpha\over 8k}
		-\abs{\cup_{g=2}^{k}\sth{\calP_g^{(\ell_g)}/ \calB(\mu_g^{(g,\ell_g)},\dist_g^{(\ell_g)})}}
		\nonumber\\
		& = {7n\alpha\over 8k}
		-\sum_{g=2}^{k}\abs{\calP_g^{(\ell_g)}/ \calB(\mu_g^{(g,\ell_g)},\dist_g^{(\ell_g)})}
		\stepb{\geq} {7n\alpha\over 8k}-n\beta
		\geq {3n\alpha\over 4k},
		\label{eq:outlier-init14-k}
	\end{align}
	where (a) followed \eqref{eq:outlier-init13-k} and the fact that $\{\calP_g^{(\ell_g)}\}_{g=2}^k$ are disjoint and (b) followed from as
	$$
	\sum_{g=2}^{k}\abs{\calP_g^{(\ell_g)}/ \calB(\mu_g^{(g,\ell_g)},\dist_g^{(\ell_g)})}
	\leq 
	\beta\sum_{g=2}^{k}\abs{\calP_g^{(\ell_g)}}\leq n\beta.
	$$
	As we have for each $j=1,\dots,k$
	\begin{align*}
		&\{Y_i:i\in T^*_{j}\}\cap \qth{\cup_{g=2}^{k}\sth{\calP_g^{(\ell_g)}\cap \calB(\mu_g^{(g,\ell_g)},\dist_g^{(\ell_g)})}}
		\nonumber\\
		&=\cup_{g=2}^{k} \sth{\{Y_i:i\in T^*_{j}\}\cap\calP_g^{(\ell_g)}\cap \calB(\mu_g^{(g,\ell_g)},\dist_g^{(\ell_g)})}
	\end{align*} 
	with the union on right side of the above display is disjoint, by the pigeon hole principle there exist indices $g,j_1,j_2$ such that
	\begin{align*}
		&\abs{\sth{\calP_g^{(\ell_g)}\cap \calB(\mu_g^{(g,\ell_g)},\dist_g^{(\ell_g)})}\cap \{Y_i:i\in T^*_{j}\}}
		\nonumber\\
		&\geq {\min_{j=1}^k\abs{\{Y_i:i\in T^*_j\}\cap\qth{\cup_{g=2}^{k}\sth{\calP_g^{(\ell_g)}\cap \calB(\mu_g^{(g,\ell_g)},\dist_g^{(\ell_g)})}}}
			\over k}
		\geq {3n\alpha\over 4k^2},
		\quad j=j_1,j_2,
	\end{align*}
	where the last display followed using \eqref{eq:outlier-init14-k}.
	However, as $\Delta\geq 16\sigma k \decay^{-1\pth{e^{-{5k^2\over \alpha}}}}$ implies
	$$\abs{\{Y_i:i\in T^*_j\}/\calB(\theta_{j},\Delta/3)}
	\leq {5n^*_j\over 4\log(1/G(\Delta/16\sigma k))}
	\leq {n\alpha\over 4k^2}, \quad j=j_1,j_2,
	$$ 
	we get that for $j=j_1,j_2$
	\begin{align*}
		&\abs{\sth{\calP_g^{(\ell_g)}\cap \calB(\mu_g^{(g,\ell_g)},\dist_g^{(\ell_g)})}\cap \{Y_i:i\in T^*_{j}\}\cap \calB(\theta_{j},\Delta/3)}
		\nonumber\\
		&\geq \abs{\sth{\calP_g^{(\ell_g)}\cap \calB(\mu_g^{(g,\ell_g)},\dist_g^{(\ell_g)})}\cap \{Y_i:i\in T^*_{j}\}}
		-\abs{\{Y_i:i\in T^*_{j}\} /\calB(\theta_{j},\Delta/3)}
		\geq {n\alpha\over 2k^2}.
	\end{align*}
	Hence, there exists $x,y\in \{Y_i:i\in [n]\}$ such that 
	\begin{align}
		x&\in \sth{\calP_g^{(\ell_g)}\cap \calB(\mu_g^{(g,\ell_g)},\dist_g^{(\ell_g)})}\cap\{Y_i:i\in T^*_{j_1}\}\cap \calB(\theta_{j_1},\Delta/3),
		\nonumber\\
		y&\in \sth{\calP_g^{(\ell_g)}\cap \calB(\mu_g^{(g,\ell_g)},\dist_g^{(\ell_g)})}\cap\{Y_i:i\in T^*_{j_2}\}\cap \calB(\theta_{j_2},\Delta/3).
	\end{align} 	
	As we have $\dist_g^{(\ell_g)}\leq {\Delta\over 8}$ and $\|\theta_{j_1}-\theta_{j_2}\|\geq \Delta$ we get a contradiction
	\begin{align}
		\|x-y\|&\geq
		\|\theta_{j_1}-\theta_{j_2}\|
		-\|x-\theta_{j_1}\|
		-\|y-\theta_{j_2}\|
		\geq {\Delta\over 3}.
		\nonumber\\
		\|x-y\|&\leq \|x-\mu_g^{(g,\ell_g)}\|
		+\|y-\mu_g^{(g,\ell_g)}\|\leq 2\dist_g^{(\ell_g)}\leq {\Delta\over 4}.
	\end{align}
	Hence all of the centroids lie in $\cup_{i=1}^k\calB(\theta_i,\Delta/3)$.

	Now it remains to show that $\sth{\mu_g^{(g,\ell_g)}}_{g=1}^k$ lie in different balls among $\sth{\calB(\theta_g,\Delta/3)}_{g=1}^k$. If not, then without a loss of generality let $\calB\pth{\theta_1,{\Delta\over 3}}$ contains two of the centroids, say $\mu_{j_1}^{(j_1,\ell_{j_1})},\mu_{j_2}^{(j_2,\ell_{j_2})}$.
	Also, as $\calB\pth{\theta_1,{\Delta\over 3}}$ contains two centroids, by the pigeonhole principle we get that there is an index $g\neq 1$ such that 
	$$
	\mu_j^{(j,\ell_j)}\notin \calB\pth{\theta_g,{\Delta\over 3}},\quad j=1,\dots,k.
	$$ 
	In view of the disjoint union $\sth{\calB\pth{\theta_g,{\Delta\over8}}}_{g=1}^k$,and $\dist_j^{(\ell_j)}\leq{\Delta\over 8}$ the above implies
	\begin{align}
		\calB\pth{\theta_g,{\Delta\over 8}}\cap 
		\calB\pth{\mu_j^{(j,\ell_j)},\dist_j^{(\ell_j)}}=\phi,\quad j=1,\dots,k.
	\end{align}  
	Note that by \prettyref{lmm:E-conc-init-k}
	\begin{align}\label{eq:outlier-init15-k}
		\abs{\{Y_i:i\in [n]\}\cap \calB\pth{\theta_g,{\Delta\over 8}}}
		\geq {n\alpha\over 2k}.
	\end{align}
	As the disjoint union $\cup_{m=1}^k\calP_m^{(\ell_m)}$ is the entire set of data points, 
	we get
	\begin{align}
		\{Y_i:i\in [n]\}\cap \calB\pth{\theta_g,{\Delta\over 8}}
		&=\sth{\cup_{j=1}^k\calP_j^{(\ell_j)}}
		\cap \calB\pth{\theta_g,{\Delta\over 8}}
		\nonumber\\
		&= \cup_{j=1}^k\sth{\calP_j^{(\ell_j)}
			\cap \calB\pth{\theta_g,{\Delta\over 8}}}
		\subseteq 
		\cup_{j=1}^k\sth{\calP_j^{(\ell_j)}
			/ \calB\pth{\mu_j^{(j,\ell_j)},\dist_j^{(\ell_j)}}},
	\end{align}
	which implies 
	\begin{align}
		\abs{\{Y_i:i\in [n]\}\cap \calB\pth{\theta_g,{\Delta\over 8}}}
		&\leq \sum_{j=1}^k
		\sth{\calP_j^{(\ell_j)}
			/ \calB\pth{\mu_j^{(j,\ell_j)},\dist_j^{(\ell_j)}}}
		\nonumber\\
		&\leq \sum_{j=1}^k \beta\abs{\calP_j^{(\ell_j)}}
		\leq n\beta={n\alpha\over 4k^2}.
	\end{align}
	This provides a contradiction to \eqref{eq:outlier-init15-k}. Hence, all the centroids must lie in different $\calB\pth{\theta_j,{\Delta\over 3}}$ sets.
\end{proof}
\end{proof}

\begin{proof}[Proof of \prettyref{lmm:outlier-good-centroid-in-complement}]
	
	In view of \prettyref{lmm:E-conc-init-k}, there is a constant $c_1 = \sigma \decay^{-1}\pth{e^{-{5\over 4\beta^2}}}$ such that
	\begin{align}\label{eq:outlier-init1-k-induction}
		\abs{\{Y_j:j\in T^*_h\}/\calB(\theta_h,c_1)}\leq n^*_h\beta^2,\quad h\in [k].
	\end{align}
	Hence we get that for $h\in [i-1]$
	\begin{align}
		&\abs{\overline{\calP_i^{(\bar \ell_i)}}\cap {\{Y_j:j\in T^*_h\}}\cap\calB(\theta_h,c_1)}
		\nonumber\\
		& =	\abs{\overline{\calP_i^{(\bar \ell_i)}}\cap {\{Y_j:j\in T^*_h\}}} -\abs{\overline{\calP_i^{(\bar \ell_i)}}\cap{\{Y_j:j\in T^*_h\}}/\calB(\theta_h,c_1)}
		\nonumber\\
		& \geq 	\abs{\overline{\calP_i^{(\bar \ell_i)}}\cap {\{Y_j:j\in T^*_h\}}} -\abs{{\{Y_j:j\in T^*_h\}}/\calB(\theta_h,c_1)}
		\nonumber\\
		& \geq 	\abs{\overline{\calP_i^{(\bar \ell_i)}}\cap {\{Y_j:j\in T^*_h\}}} 
		- n^*_h\beta^2
		\label{eq:outlier-init16-k}\\
		&\geq n^*_h- (k-i+1)n\beta -{5(k-i+1)n^*_h\over 4\log(1/(G(\Delta/16\sigma k)))}
		-n^*_h\beta^2
		\geq m_1.
		\nonumber
	\end{align}
	where the last inequality holds whenever $\Delta\geq 16\sigma k\decay^{-1}\pth{e^{-{5k\over \alpha}}}$ as $(k-i+1)n\beta, n\beta^2\leq {n\alpha\over 4k}$. As we have from the lemma statement
	\begin{align}
		\label{eq:outlier-init18-k}
		\abs{\overline{\calP_i^{(\bar \ell_i)}}\cap \sth{\cup_{g=i}^k\{Y_j:j\in T^*_g\}}}
		&\leq {2n\alpha\beta\over 5k}
		\leq {m_1\over 2},
	\end{align}
	we get that the tightest neighborhood around any point in $\overline{\calP_i^{(\bar \ell_i)}}$ with a size $m_1$ will have a radius of at most $2c_{1}$ around that $Y_i$. Let $\mu_{i-1}^{(i-1,1)} = Y_{i^*}$ be the chosen centroid. Hence $\abs{\sth{\overline{\calP_i^{(\bar \ell_i)}}\cap \calB(Y_{i^*},2c_1)}}\geq m_1$. Then $\calB(Y_{i^*},2c_{1})$ and $\cup_{j\in [i-1]}\calB(\theta_j,c_{1})$ can not be disjoint, as in view of \eqref{eq:outlier-init16-k} it will imply that
	\begin{align}
		\abs{\overline{\calP_i^{(\bar \ell_i)}}}
		&\geq\abs{\sth{\overline{\calP_i^{(\bar \ell_i)}}\cap \calB(Y_{i^*},2c_1)}\cup\sth{\overline{\calP_i^{(\bar \ell_i)}}\cap\sth{\cup_{h\in [i-1]}\calB(\theta_h,c_1)}}}
		\nonumber\\
		&\stepa{\geq} m_1+ \sum_{h\in [i-1]} \abs{\overline{\calP_i^{(\bar \ell_i)}}\cap {\{Y_j:j\in T^*_h\}}\cap\calB(\theta_h,c_1)}
		\nonumber\\
		&\stepb{\geq} m_1-n\beta^2+ \sum_{h\in [i-1]} \abs{\overline{\calP_i^{(\bar \ell_i)}}\cap {\{Y_j:j\in T^*_h\}}} 
		\nonumber\\
		&= m_1-n\beta^2+ \abs{\overline{\calP_i^{(\bar \ell_i)}}\cap \qth{\cup_{h\in [i-1]}{\{Y_j:j\in T^*_h\}}}} 
		\nonumber\\
		&\geq m_1-n\beta^2+ \abs{\overline{\calP_i^{(\bar \ell_i)}}}-
		\abs{\overline{\calP_i^{(\bar \ell_i)}}\cap \sth{\cup_{g=i}^k\{Y_j:j\in T^*_g\}}} 
		-n^\out
		\nonumber\\
		&\stepc{=} {m_1\over 2}-n\beta^2+ \abs{\overline{\calP_i^{(\bar \ell_i)}}}-n^\out
		\stepd{\geq} \abs{\overline{\calP_i^{(\bar \ell_i)}}} + {m_1\over 8}.
		\label{eq:outlier-init17-k}
	\end{align}	
	where (a) follows from the fact that $\sth{i\in [n]:Y_i\in \calB(\theta_j,c_1)},j\in [k]$ are disjoint sets as $\min_{g\neq h\in [k]}\|\theta_g-\theta_h\|\geq \Delta$, (b) follows from \eqref{eq:outlier-init18-k}, (c) followed from \eqref{eq:outlier-init16-k} and (d) follows from assuming 
	\begin{align}\label{eq:out-req-4}
		n^\out\leq {n\alpha\over 32k}\leq {m_1\over 8}.
	\end{align}
	Hence, $Y_{i^*}$ is at a distance at most $3c_{1}$ from one of the centroids. Without loss of we can the closest centroid to be $\theta_{i-1}$. 
\end{proof}

\noindent \textbf{Requirement on $n^\out$}: To summarize how many outliers our initialization technique can tolerate, we combine \eqref{eq:out-req-1},\eqref{eq:out-req-2},\eqref{eq:out-req-3},\eqref{eq:out-req-4} to get
\begin{align*}
	n^\out\leq\min\sth{{n\alpha\over 16k},{n\alpha\beta\over 16k},{n\beta\over 2},{n\alpha\over 32k}} = {n\alpha^2\over 64k^3}.
\end{align*}

\end{document}